\documentclass[11pt]{amsart}
\usepackage{amscd,amssymb}
\usepackage{amsthm,amsmath,amssymb}
\usepackage[matrix,arrow]{xy}

\theoremstyle{definition}
\newtheorem{example}[equation]{Example}
\newtheorem{definition}[equation]{Definition}
\newtheorem{theorem}[equation]{Theorem}
\newtheorem{lemma}[equation]{Lemma}
\newtheorem{corollary}[equation]{Corollary}

\newtheorem{question}[equation]{Question}
\newtheorem*{question*}{Question}

\newtheorem*{problem*}{Problem}

\theoremstyle{remark}
\newtheorem{remark}[equation]{Remark}

\makeatletter\@addtoreset{equation}{section} \makeatother

\def\A {\mathrm{A}}
\def\SS {\mathrm{S}}
\def\SL {\mathrm{SL}}
\def\PSL {\mathrm{PSL}}

\def\PGL {\mathrm{PGL}}

\def\PSp {\mathrm{PSp}}
\def\mult {\mathrm{mult}}
\def\rk {\mathrm{rk}}
\def\Cl {\mathrm{Cl}}
\def\P {\mathbb{P}}
\def\Z {\mathbb{Z}}
\def\Q {\mathbb{Q}}
\def\F {\mathbb{F}}
\def\C {\mathbb{C}}
\def\M {\mathcal{M}}

\def\O {\mathcal{O}}
\def\qlin {\sim_{\Q}}

\def\le {\leqslant}
\def\ge {\geqslant}

\author{Ivan Cheltsov and Constantin Shramov}

\title{Five embeddings of one simple group}

\thanks{The first author was supported by the~grants NSF DMS-0701465 and
EPSRC EP/E048412/1, the~second~author was supported by the~grants
RFFI  08-01-00395-a, NSh-1987.2008.1 and EPSRC EP/E048412/1.}

\address{University of Edinburgh, Edinburgh EH9 3JZ, UK}

\address{\texttt{I.Cheltsov@ed.ac.uk}\ \ \texttt{shramov@mccme.ru}}

\begin{document}

\begin{abstract}
We propose a new method to study birational maps between Fano varieties
based on multiplier ideal sheaves. Using this method, we prove
equivariant birational rigidity of four Fano threefolds acted on by the
group~$\A_6$. As an application, we obtain
that $\mathrm{Bir}(\mathbb{P}^{3})$ has at least five
non-conjugate subgroups isomorphic~to~$\A_{6}$.
\end{abstract}

\maketitle

\tableofcontents

We assume that all varieties are projective, normal, and defined
over $\mathbb{C}$.

\section{Introduction}
\label{section:intro}

The group of birational automorphisms of the~projective plane
$\mathbb{P}^2$, also known as the~Cremona group $\mathrm{Cr}_2(\mathbb{C})$, is
a classical object in algebraic geometry.
Classifying finite subgroups in $\mathrm{Cr}_2(\mathbb{C})$ up to
conjugacy is an important research direction originating in the~
works of Kantor, Bertini, Wiman and others. This
classification has been almost completed in \cite{DoIs06}. In
particular, we know how to decide whether two finite isomorphic
subgroups in $\mathrm{Cr}_{2}(\mathbb{C})$ are conjugate or not.

The group of birational automorphisms of the~projective space
$\mathbb{P}^3$, the~Cremona group $\mathrm{Cr}_3(\mathbb{C})$, is
poorly understood. We know so little about finite subgroups in
$\mathrm{Cr}_3(\mathbb{C})$ that Serre asked

\begin{question}[{\cite[Question~6.0]{Serre}}]
\label{question:Serre-1} Does there exist a finite group which is
not embeddable in $\mathrm{Cr}_3(\C)$?
\end{question}

It is well-know that most of finite groups are not embeddable in
$\mathrm{Cr}_2(\C)$. For example, the only finite simple
non-abelian subgroup in $\mathrm{Cr}_2(\C)$ are $\A_5$, $\A_6$,
$\PSL_2(\mathbb{F}_{7})$ (see \cite{DoIs06}). Prokhorov explicitly
answered Question~\ref{question:Serre-1} by proving the following

\begin{theorem}[{\cite[Theorem~1.3]{Pr09}}]
\label{theorem:Yura-Cremona-1} Let $\bar{G}$ be a finite simple
non-abelian subgroup in $\mathrm{Cr}_3(\C)$. Then $\bar{G}$ is one
of the following groups: $\A_5$, $\A_6$, $\PSL_2(\mathbb{F}_{7})$,
$\A_7$, $\SL_2(\mathbb{F}_{8})$, $\PSp_4(\mathbb{F}_{3})$. All the
possibilities occur.
\end{theorem}

During a workshop ``Subgroups of Cremona groups'' at Edinburgh in
Spring 2010, Serre posed

\begin{question}
\label{question:Serre-2} What are normalizers in
$\mathrm{Cr}_3(\C)$ of finite simple non-abelian subgroup in
$\mathrm{Cr}_3(\C)$?
\end{question}

An answer to Question~\ref{question:Serre-2} depends on
classification of finite non-abelian simple subgroups in
$\mathrm{Cr}_3(\mathbb{C})$ up to conjugation. But
Theorem~\ref{theorem:Yura-Cremona-1} gives such classification
only up to isomorphism. So before trying to answer
Question~\ref{question:Serre-2}, we must first consider

\begin{question}
\label{question:conjugacy-classes} How to decide whether two
finite isomorphic subgroups in $\mathrm{Cr}_{3}(\mathbb{C})$ are
conjugate or not?
\end{question}

With a very few exceptions (see \cite{LPR}, \cite{Ch07b},
\cite{Pr09}), we do not know how to answer
Question~\ref{question:conjugacy-classes} in full generality. In
this paper, we show a new technique that can be used to answer
both Questions~\ref{question:Serre-2} and
\ref{question:conjugacy-classes} in some cases. In particular, we
use this technique to prove

\begin{theorem}
\label{theorem:main-cheap} Up to conjugation, there are at least
$5$ subgroups in $\mathrm{Cr}_{3}(\mathbb{C})$ that are isomorphic
to $\A_{6}$. For three of these non-conjugated subgroups, the
normalizer in $\mathrm{Cr}_{3}(\mathbb{C})$ is $\SS_{6}$, and for
one is the~free product of $\SS_{6}$ and $\SS_{6}$ with an
amalgamated subgroup $\A_{6}$.
\end{theorem}

Let us show how to translate
Questions~\ref{question:conjugacy-classes} and
\ref{question:Serre-2} into a  geometric language. Let $\bar{G}$
be a finite subgroup, and let $\tau\colon \bar{G}\to
\mathrm{Cr}_{3}(\mathbb{C})$ be a~monomorphism. It is well-known
(see \cite[Section~4.2]{Pr09}) that there is a~birational map
$\xi\colon V\dasharrow\nolinebreak\mathbb{P}^{3}$ such that
\begin{itemize}
\item the~threefold $V$ is normal and has terminal singularities,

\item there exists a~mono\-mor\-phism $\upsilon\colon \bar{G}\to \mathrm{Aut}(V)$,%

\item for every element $g\in \bar{G}$, we have $\tau\big(g\big)=\xi\circ\upsilon\big(g\big)\circ \xi^{-1}$,%

\item there exists a $\upsilon(\bar{G})$-Mori fibration
$\pi\colon V\to S$, i.\,e. a non-birational
$\upsilon(\bar{G})$-equivariant surjective morphism with connected
fibers such
that the~divisor $-K_V$ is $\pi$-ample
and
for every $\upsilon(\bar{G})$-invariant Weil divisor $D$ on
$V$, there is $\delta\in\mathbb{Q}$ such that
$$
\delta K_{V}+D\sim_{\mathbb{Q}}\pi^{*}\big(H\big)
$$
for some $\mathbb{Q}$-Cartier divisor $H$ on the variety $S$.
\end{itemize}

\begin{definition}
\label{definition:regularization} The~quadruple
$(V,\xi,\upsilon,\pi)$ is a~Mori regularization of the~pair
$(\bar{G},\tau)$.
\end{definition}

\begin{theorem}
\label{theorem:normalizer} The normalizer of the group
$\tau(\bar{G})$ in $\mathrm{Cr}_{3}(\mathbb{C})$ is isomorphic to
the group of all $\upsilon(\bar{G})$-equivariant birational
automorphisms of the variety $V$.
\end{theorem}

\begin{proof}
The proof is obvious and left to the reader.
\end{proof}

Let $\tau^{\prime}\colon \bar{G}\to\mathrm{Cr}_{3}(\mathbb{C})$ be
another~monomorphism, and let
$(V^{\prime},\xi^{\prime},\upsilon^{\prime},\pi^{\prime})$
be~a~Mori~regu\-la\-ri\-zation of the~pair
$(\bar{G},\tau^{\prime})$.

\begin{theorem}[{\cite[Lemma~3.4]{DoIs06}}]
\label{theorem:regularization} The~following assertions are
equivalent:
\begin{itemize}
\item the~subgroups $\tau(\bar{G})$ and $\tau^{\prime}(\bar{G})$ are conjugate in $\mathrm{Cr}_{3}(\mathbb{C})$,%

\item there is a~birational map $\rho\colon V\dasharrow
V^{\prime}$ such that for every $g\in \bar{G}$ there is
$g^{\prime}\in \bar{G}$ such that
$$
\upsilon^{\prime}\big(g^{\prime}\big)=\rho\circ\upsilon\big(g\big)\circ
\rho^{-1}\in\mathrm{Aut}\big(V^{\prime}\big).
$$
\end{itemize}
\end{theorem}

Thus, the~subgroups $\tau(\bar{G})$ and $\tau^{\prime}(\bar{G})$
are conjugate in $\mathrm{Cr}_{3}(\mathbb{C})$ if and only if
there exists a $\bar{G}$-equivariant birational map $V\dasharrow
V^\prime$ with respect to the actions of the group $\bar{G}$ on
$V$ and $V^{\prime}$ induced by the monomorphisms $\upsilon$ and
$\upsilon^\prime$. However, to prove or disprove the existence of
such birational map is not easy in general. The following two
definitions help to deal with this in the case when
$\pi(V)=S$~is~a~point, i.e. in the case when $V$ is a normal Fano
variety with terminal singularities such that the
$\upsilon(\bar{G})$-invariant subgroup in
$\mathrm{Cl}(V)\otimes\mathbb{Q}$ is generated by $-K_{V}$.

\begin{definition}[{cf.~\cite[Definition~0.3.3]{Ch05umn}}]
\label{definition:rigid} If $S$ is a point, then $V$ is called
$\upsilon(\bar{G})$-birationally rigid if for every
Mori~regu\-la\-ri\-zation
$(V^{\prime},\xi^{\prime},\upsilon^{\prime},\pi^{\prime})$ of
the~pair $(\bar{G},\tau)$, we have $V^{\prime}\cong V$,
$\pi^{\prime}(V^{\prime})$ is a point, and the subgroups
$\upsilon(\bar{G})$ and $\upsilon^{\prime}(\bar{G})$ are conjugate
in $\mathrm{Aut}(V)\cong\mathrm{Aut}(V^{\prime})$.
\end{definition}

\begin{definition}[{cf.~\cite[Definition~0.3.4]{Ch05umn}}]
\label{definition:superrigid} If $S$ is a point, then $V$ is
called $\upsilon(\bar{G})$-birationally superrigid if for every
Mori~regu\-la\-ri\-zation
$(V^{\prime},\xi^{\prime},\upsilon^{\prime},\pi^{\prime})$ of
the~pair $(\bar{G},\tau)$, the map $\xi^{-1}\circ\xi^{\prime}$ is
biregular.
\end{definition}

From now on we will identify the group $\bar{G}$ with its image
$\upsilon(\bar{G})$ to simplify the notation. Note that if $S$ is
a point and $V$ is $\bar{G}$-birationally superrigid, then it
follows from Definitions~\ref{definition:rigid}
and~\ref{definition:superrigid} that $V$ is $\bar{G}$-birationally
rigid. Denote by $\mathrm{Aut}^{\bar{G}}(V)$ (by
$\mathrm{Bir}^{\bar{G}}(V)$, respectively) the groups of
$\bar{G}$-equivariant biregular (birational, respectively)
self-maps of $V$. Note also that the group~$\mathrm{Aut}^{\bar{G}}(V)$
(the group~$\mathrm{Bir}^{\bar{G}}(V)$, respectively) coincides
with the normalizer of $\bar{G}$ in $\mathrm{Aut}(V)$ (in
$\mathrm{Bir}(V)$, respectively).

\begin{corollary}
\label{corollary:regularization} Suppose that $S$ is a point, the
threefold $V$ is $\bar{G}$-birationally rigid, and
$V\not\cong V^{\prime}$. Then $\tau(\bar{G})$ and
$\tau^{\prime}(\bar{G})$ are not conjugate in
$\mathrm{Cr}_{3}(\C)$.
\end{corollary}

\begin{corollary}
\label{corollary:normalizers} Suppose that $S$ is a point, and $V$
is $\bar{G}$-birationally superrigid. Then the
normalizer of the group $\tau(\bar{G})$ in
$\mathrm{Cr}_{3}(\mathbb{C})$ is isomorphic to
$\mathrm{Aut}^{\bar{G}}(V)$.
\end{corollary}

Let us consider examples of Mori regularizations in the case when
$\bar{G}$ is one of the following groups: $\A_7$,
$\SL_2(\mathbb{F}_{8})$, and $\PSp_4(\mathbb{F}_{3})$. Note that
in the latter case, the variety $S$ must be a point, since $\A_7$,
$\SL_2(\mathbb{F}_{8})$, and $\PSp_4(\mathbb{F}_{3})$ are not
embeddable in $\mathrm{Cr}_2(\C)$.

\begin{example}
\label{example:P3-2} If $\bar{G}=\A_7$ or
$\bar{G}\cong\PSp_4(\F_3)$, then there is a~monomorphism
$\upsilon\colon \bar{G}\to \mathrm{Aut}(\mathbb{P}^{3})$ (see
\cite{Atlas}).
\end{example}

\begin{example}[{\cite[Example~2.11]{Pr09}}]
\label{example:genus-7} Put $\bar{G}=\SL_2(\mathbb{F}_{8})$. One
can show that there is a~monomorphism $\alpha\colon
\bar{G}\to\mathrm{Aut}(\mathrm{LGr}(4,9))$, and
$\mathrm{Pic}(\mathrm{LGr}(4,9))=\mathbb{Z}[H]$, where $H$ is
an~ample divisor such that $|H|$ gives an embedding
$\zeta\colon\mathrm{LGr}(4,9)\to\mathbb{P}^{15}$, which implies
that the~mo\-no\-morphism $\alpha$ induces a~monomorphism
$\beta\colon \bar{G}\to\mathrm{Aut}(\mathbb{P}^{15})$. Put
$V=\zeta(\mathrm{LGr}(4,9))\cap\Pi$, where $\Pi$ is the~unique
$\beta(\bar{G})$-invariant linear subspace
$\Pi\subset\mathbb{P}^{15}$ such that $\mathrm{dim}(\Pi)=8$. Then
$V$ is a~smooth Fano threefold such that
$\mathrm{Pic}(V)\cong\mathbb{Z}$ and $-K_{V}^{3}=12$ (see
\cite{Muk92}), and $V$ is rational
(see \cite[Corollary~4.4.12]{IsPr99}).
The~monomorphism $\beta$ induces an isomorphism
$\upsilon\colon \bar{G}\to\mathrm{Aut}(V)$.

\end{example}

\begin{example}
\label{example:Burkhardt-quartic} Suppose that $V$ is a~complete
intersection in $\mathbb{P}^{5}$ that is given by
$$
\sum_{i=0}^{5}x_{i}=\sigma_4\big(x_{0},x_{1},x_{2},x_{3},x_{4},x_{5}\big)=0\subset\mathbb{P}^{5}\cong\mathrm{Proj}\Big(\mathbb{C}\big[x_{0},x_{1},x_{2},x_{3},x_{4},x_{5}\big]\Big),%
$$
where $\sigma_{4}$ is the~elementary symmetric form of degree $4$.
Put $\bar{G}=\PSp_4(\mathbb{F}_{3})$. Then there~is a~monomorphism
$\upsilon\colon\bar{G}\to\mathrm{Aut}(V)$
(see~\cite[Chapter~5]{Hu96}). It is known that the~threefold $V$
is rational (see~\cite[Section~5.2.7]{Hu96}), and
the~$\bar{G}$-invariant subgroup in the~group
$\mathrm{Cl}(V)$ is $\mathbb{Z}$ (cf.
Theorem~\ref{theorem:Bukrhardt}).
\end{example}

The threefold constructed in
Example~\ref{example:Burkhardt-quartic} is known as the~Burkhardt
quartic (see \cite{Hu96}).

\begin{theorem}[{\cite[Theorem~1.5]{Pr09}, \cite{Be11}}]
\label{theorem:Yura-Cremona-3} Suppose that $\bar{G}$ is one of
the following subgroup: $\A_7$, $\SL_2(\mathbb{F}_{8})$,
or $\PSp_4(\mathbb{F}_{3})$.
Let $V$ be a normal rational threefold with at most
terminal singularities that admits a faithful action of the group
$\bar{G}$ such that there exists a $\bar{G}$-Mori fibration $\pi\colon V\to S$.
Then $S$ is a point, and
\begin{itemize}
\item if $\bar{G}\cong\A_7$, then $V\cong\mathbb{P}^3$ (see Example~\ref{example:P3-2}),%
\item if $\bar{G}\cong\SL_2(\mathbb{F}_{8})$, then $V$ is the~threefold constructed in Example~\ref{example:genus-7},%
\item if $\bar{G}\cong\PSp_4(\mathbb{F}_{3})$, then either $V\cong\mathbb{P}^3$ or $V$ is the~threefold constructed in Example~\ref{example:Burkhardt-quartic}.%
\end{itemize}
\end{theorem}

Now let us consider  examples of Mori regularizations in the case
when $\bar{G}\cong\A_{6}$.

\begin{example}
\label{example:P3} Put $V=\mathbb{P}^{3}$ and $\bar{G}=\A_{6}$.
Then there exists a~monomorphism $2.\A_{6}\to\SL_4(\mathbb{C})$
(see~\cite{Atlas}), which induces a~monomorphism
$\upsilon\colon\bar{G}\to\mathrm{Aut}(V)\cong \PGL_4(\mathbb{C})$.
Furthermore, up to conjugation this is the unique subgroup
isomorphic to $\A_6$ in $\PGL_4(\mathbb{C})$.
\end{example}

\begin{example}
\label{example:quadric-threefold} Put $\bar{G}=\A_{6}$. Then there
exists a monomorphism $\bar{G}\to\mathrm{SO}_{5}(\mathbb{R})$,
which implies the existence of a~monomorphism
$\upsilon\colon\bar{G}\to\mathrm{Aut}(V)$, where $V$ is a smooth
quadric threefold.
\end{example}

\begin{example}
\label{example:Segre-cubic} Put $\bar{G}=\A_{6}$. Suppose that
$V$ is the~complete intersection
$$
\sum_{i=0}^{5}x_{i}=\sum_{i=0}^{5}x_{i}^{3}=0\subset\mathbb{P}^{5}\cong\mathrm{Proj}\Big(\mathbb{C}\big[x_{0},x_{1},x_{2},x_{3},x_{4},x_{5}\big]\Big).
$$
Then $V$ has exactly $10$ isolated
ordinary double points (and hence $V$ is rational), and there
exists a~natural monomorphism
$\upsilon\colon\bar{G}\to\mathrm{Aut}(V)$. Furthermore, one has
$\Cl(V)\cong\mathbb{Z}^6$ (see \cite{Finkelberg},
\cite{FinkelbergWerner}), which implies that
the~$\bar{G}$-invariant subgroup of the~group
$\mathrm{Cl}(V)$ is $\mathbb{Z}$.
\end{example}

The threefold constructed in Example~\ref{example:Segre-cubic}
is known as the~Segre cubic (see
\cite[Section~3.2]{Hu96}).

Let us show how to prove that $V$ is
$\bar{G}$-birational superrigid in the case when $S$ is
a point.

\begin{theorem}\label{theorem:Bukrhardt} Put $\bar{G}=\mathrm{A}_6$.
Suppose that $V$ is the threefold constructed in
Example~\ref{example:Burkhardt-quartic}. Then
$\mathrm{rk}\mathrm{Cl}^{\bar{G}}(V)=1$ , the variety $V$ is
$\bar{G}$-birationally superrigid, and
$\mathrm{Bir}^{\bar{G}}(V)\cong\SS_{6}$.
\end{theorem}

\begin{proof}
The required assertions are probably well-known to experts. But we
failed to find any relevant reference. The fact that
the~$\bar{G}$-invariant subgroup of the~group $\mathrm{Cl}(V)$ is
$\mathbb{Z}$ follows from the isomorphism
$V\slash\bar{F}\cong\mathbb{P}(1,2,2,3)$, where $\bar{F}$ is a
subgroup in $\bar{G}$ such that $\bar{F}\cong\SS_{4}$ and
$\bar{F}$ fixes a~point in $\mathrm{Sing}(V)$. The isomorphism
$\mathrm{Bir}^{\bar{G}}(V)\cong\SS_{6}$ follows from
$\bar{G}$-birational superrigidity of the threeefold $V$,
Corollary~\ref{corollary:normalizers}, and classification of
primitive subgroups in $\SL_5(\C)$ (see \cite{Fe71}). The fact
that $V$ is $\bar{G}$-birationally superrigid easily follows from
the proof of \cite[Theorem~5]{Me03} (cf. \cite[Section~9]{Shr08}),
but we decided to prove this here to illustrate how to proof
$\bar{G}$-birational superrigidity using only classical technique
that goes back to \cite{IsMa71}.

Suppose that $V$ is not $\bar{G}$-birationally
superrigid. Then \cite[Theorem~4.2]{Co95} implies the~existence of
 a~$\bar{G}$-invariant linear system
$\mathcal{M}$ on the~threefold $V$ such that $\mathcal{M}$ does
not have fixed components, and the~log pair
$(V,\lambda\mathcal{M})$ is not canonical, where
$\lambda\in\mathbb{Q}$ is such that
$K_{V}+\lambda\mathcal{M}\sim_{\mathbb{Q}} 0$.

Let $M_{1}$ and $M_{2}$ be sufficiently general surfaces in
the~linear system $\mathcal{M}$, and let $H$ be a general surface
in $|-K_{V}|$. Consider $V$ as a quartic threefold in
$\mathbb{P}^{4}$.

Suppose that there is an~irreducible curve $C\subset V$ such that
the~log pair $(V,\lambda\mathcal{M})$ is not canonical along
the~curve $C$.
Then the multiplicity of the linear system $|\mathcal{M}|$ along the curve
$C$ is greater than $1/\lambda$ (see e.\,g.~\cite[Exercise~6.18]{CKS04}).
Let $Z$ be the~$\bar{G}$-orbit of
the~curve $C$. Put $d=-K_{V}\cdot Z$. Then
$$
4\slash\lambda^{2}=M_{1}\cdot M_{2}\cdot H\geqslant d\mathrm{mult}_{C}\big(M_{1}\big)\mathrm{mult}_{C}\big(M_{2}\big)>d\slash\lambda^{2},%
$$
which implies that $d\leqslant 3$. So, the~curve $Z$ is contained
in a hyperplane in $\mathbb{P}^{4}$, which is impossible, because
the~corresponding five-dimensional representation of the~group
$\bar{G}$ is irreducible.

Thus, the centers of canonical singularities
(see \cite[Definition 1.3.8]{Ch05umn})
of the~log pair $(V,\lambda\mathcal{M})$ are points.
Take any point $P\in V$ such that the~singularities of the~log
pair $(V,\lambda\mathcal{M})$ are not canonical at the~point $P$.
Suppose that $P\in V$ is a smooth point,
and let $H_P$ be a general hyperplane section of $V\subset\P^4$ passing through
$P$. Then
$$4/\lambda^2=M_1\cdot M_2\cdot H_P\ge \mult_P(M_1\cdot M_2)>4/\lambda^2$$
by~\cite[Corollary~3.4]{Co00}.
The obtained contradiction shows that the~quartic $V$ is singular at $P$.

Let $\Sigma$ be the~$\bar{G}$-orbit of the~point $P\in
V$. Then there is a subset $\Gamma\subset\Sigma$ such that
$|\Gamma|=4$, and the~set $\Gamma$ is not contained in any
two-dimensional linear subspace of $\mathbb{P}^{4}$.

Let $\mathcal{Q}\subset|-2K_{V}|$ be a linear subsystem that
consists of all surfaces in $|-2K_{V}|$ that pass through every
point of the~set $\Gamma$. Then the~base locus of the~linear
system $\mathcal{Q}$ is the~set $\Gamma$.

Let $\pi\colon U\to V$ be a~blow up of the~set $\Gamma$, let
$\bar{\mathcal{M}}$ be the~proper transforms of the~linear
system~$\mathcal{M}$ on the~variety $U$, and let $E_{1},\ldots,E_{4}$ be
exceptional divisors of $\pi$. Then there is  $m\in\mathbb{Z}$
such that $\bar{\mathcal{M}}\sim
\pi^{*}(\mathcal{M})-m\sum_{i=1}^{4}E_{i}$. Moreover, it follows
from  \cite[Theorem~1.7.20]{Ch05umn} that $m>1/\lambda$ (cf.
\cite[Theorem~3.10]{Co00}). Let $\Upsilon$ be the~intersection of
two sufficiently general surfaces in $\mathcal{Q}$, and
let~$\bar{\Upsilon}$ be the~proper transform of the~curve $\Upsilon$
on the~threefold $U$. Then
$$
0>8\slash\lambda-8m>8\slash\lambda-
m\sum_{O\in\Gamma}\mathrm{mult}_{O}\big(\Upsilon\big)=
\bar{M}\cdot\bar{\Upsilon}\geqslant 0,%
$$
where $\bar{M}$ is a general surface in $\bar{\mathcal{M}}$. The
obtained contradiction completes the~proof.
\end{proof}

\begin{corollary}
\label{corollary:Bukrhardt-baby} Up to conjugation, there are at
least $2$ subgroups in $\mathrm{Cr}_{3}(\mathbb{C})$
isomorphic~to~$\A_{6}$.
\end{corollary}

\begin{corollary}
\label{corollary:Bukrhardt} Suppose that
$\bar{G}=\PSp_4(\mathbb{F}_{3})$ and $V$ is the threefold
constructed in Example~\ref{example:Burkhardt-quartic}. Then
the~$\bar{G}$-invariant subgroup of the~group
$\mathrm{Cl}(V)$ is $\mathbb{Z}$, and $V$ is
$\bar{G}$-birationally superrigid.
\end{corollary}

Now applying Theorem~\ref{theorem:regularization},
Corollary~\ref{corollary:regularization}, and
Theorem~\ref{theorem:Yura-Cremona-3}, we obtain

\begin{corollary}
\label{corollary:Yura-Cremona-3-1} Up to conjugation, the~group
$\mathrm{Cr}_{3}(\mathbb{C})$ contains exactly $1$ subgroup that
is isomorphic~to~$\SL_2(\mathbb{F}_{8})$, exactly $2$ subgroups
that are isomorphic~to~$\PSp_4(\mathbb{F}_{3})$, exactly $1$
subgroup that is isomorphic~to~$\A_{7}$.
\end{corollary}

The proof of Theorem~\ref{theorem:Bukrhardt} goes back to the
classical result of Iskovskikh and Manin about the non-rationality
of every smoooth quartic threefold (see \cite{IsMa71}). Of course,
the quartic threefold constructed in
Example~\ref{example:Burkhardt-quartic} is not smooth,  but this
problem is dealt using technique introduced in \cite{Co00} and
\cite{Me03}. So the proof of Theorem~\ref{theorem:Bukrhardt} is a
nice illustration how to answer
Question~\ref{question:conjugacy-classes} in one particular case.
Unfortunately, this approach is hard to apply to Fano varieties
of large anticanonical degree.
The main purpose of this paper is to introduce a new technique to
prove $\bar{G}$-birational rigidity or
$\bar{G}$-birational superrigidity of the threefold $V$
in the case when $S$ is a point. This technique does not always
work. But sometimes it does. We will use this technique to prove

\begin{theorem}
\label{theorem:main} Suppose that $\bar{G}=\A_6$. If
$V\cong\mathbb{P}^{3}$ (cf.~Example~\ref{example:P3}), then
the~threefold $V$ is $\bar{G}$-birationally rigid (but not
$\bar{G}$-birationally superrigid), and
$\mathrm{Bir}^{\bar{G}}(V)$ is a~free product of two copies of
$\SS_{6}$ with amalgamated subgroup $\A_{6}$. If $V$ is either
the~Segre cubic or a smooth quadric threefold, then $V$ is
$\bar{G}$-birationally superrigid and
$\mathrm{Bir}^{\bar{G}}(V)\cong\SS_{6}$.
\end{theorem}

\begin{corollary}
\label{corollary:main} Up to conjugation, there are at least $5$
subgroups in $\mathrm{Cr}_{3}(\mathbb{C})$
isomorphic~to~$\A_{6}$.
\end{corollary}

Unfortunately, we are unable to classify all subgroups in
$\mathrm{Cr}_{3}(\mathbb{C})$ that are isomorphic~to~$\A_{6}$ up
to conjugation similar to what is done in
Corollary~\ref{corollary:Yura-Cremona-3-1} for $\A_7$,
$\SL_2(\mathbb{F}_{8})$, $\PSp_4(\mathbb{F}_{3})$. The main
problem is that $S$ is not necessarily a point if $\bar{G}=\A_6$,
since $\A_{6}$ is embeddable in $\mathrm{Cr}_2(\C)$. If
$\bar{G}=\A_{6}$ and $\mathrm{dim}(S)=1$, i.e. $\pi\colon V\to S$
is a del Pezzo fibration, then
$V\cong\mathbb{P}^2\times\mathbb{P}^1$ by
Theorem~\ref{theorem:del-Pezzo-1}. However, If $\bar{G}=\A_{6}$
and $\mathrm{dim}(S)=2$, i.e. $\pi\colon V\to S$ is a conic
bundle, then we do not have any decent description of the
threefold $V$ and this is the main reason why we are unable to
classify all subgroups in $\mathrm{Cr}_{3}(\mathbb{C})$ that are
isomorphic~to~$\A_{6}$ up to conjugation.

\begin{remark}
\label{remark:conic-bundle} Suppose that $\bar{G}=\A_{6}$ and
$\mathrm{dim}(S)=2$. Then the monomorphism
\mbox{$\upsilon\colon\bar{G}\to\mathrm{Aut}(V)$} induces a
monomorphism $\iota\colon\bar{G}\to\mathrm{Aut}(S)$, because $\pi$
is $\bar{G}$-equivariant and $\bar{G}$  is simple. On the other
hand, the group $\bar{G}$ faithfully acts on
$\mathbb{P}^{2}\times\mathbb{P}^1$ so that its action of the first
factor is induced by a three-dimensional representation of the
group $3.\A_6$, and its action on the second factor is trivial.
Keeping in mind that $V$ is rational, it is very tempting to
expect that there always exists a commutative diagram
$$
\xymatrix{
V\ar@{-->}[rr]^{\alpha}\ar@{->}[d]^{\pi}&&\mathbb{P}^{2}\times\mathbb{P}^1\ar@{->}[d]^{pr_1}\\
S\ar@{-->}[rr]_{\beta}&&\mathbb{P}^{2},}
$$
where $\alpha$ is some $\bar{G}$-equivariant birational
map, $\beta$ is some $\iota(\bar{G})$-equivariant birational map,
and $pr_1\colon\mathbb{P}^{2}\times\mathbb{P}^1\to\mathbb{P}^{2}$
is the projection to the first factor. We do not have any strong
evidence for such expectation. But we failed to construct a
counter-example either.
\end{remark}

Let us describe the strategy how to prove
Theorem~\ref{theorem:main}. Let $X$ be a Fano threefold with
Gorenstein terminal singularities, and let $\bar{G}$ be a subgroup
in $\mathrm{Aut}(X)$ such that $\rk\Cl^{\bar{G}}(X)=1$. We suppose
that $X$ is not $\bar{G}$-birationally superrigid and seek for a
contradiction. It is well-known that in this case
there are a movable
$\bar{G}$-invariant linear system $\M$ without fixed components
and $\lambda\in\Q$ such that $K_X+\lambda\M\qlin 0$ and the log
pair $(X, \lambda\M)$ has non-canonical singularities (see
\cite[Theorem~4.2]{Co95}, \cite[Theorem~1.4.1]{Ch05umn}). Since
$X$ is Gorenstein, the singularities of the log pair $(X,
2\lambda\M)$ must be worse than log canonical by
Corollary~\ref{corollary:mult-by-2}. Now, although the last
condition is much weaker than non-canonicity of the log pair
$(X,\lambda\M)$, we are in a position to use the machinery of
multiplier ideal sheaves to obtain a contradiction. Namely, we
choose $\mu<2\lambda$ such that the log pair $(X,\mu\M)$ is
strictly log canonical, and pick up a minimal center $S$ of log
canonical singularities of $(X,\mu\M)$ (see
Section~\ref{section:preliminaries} for definitions). The
minimality of the center~$S$ implies that the $\bar{G}$-orbit of $S$
is either a finite set, or a disjoint union of irreducible curves.
We use Lemma~\ref{lemma:Kawamata-Shokurov-trick} to observe that
one may assume that every center of log canonical singularities of
the log pair $(X,\mu\M)$ is $g(S)$ for some $g\in\bar{G}$. Then
applying the Nadel--Shokurov vanishing theorem (see
Theorem~\ref{theorem:Shokurov-vanishing},
\cite[Theorem~9.4.8]{La04}) we obtain an upper bound on the number
of irreducible components of the $\bar{G}$-orbit of $S$. If $S$ is
a point, the latter appears to be mostly enough to obtain a
contradiction with the structure of $\bar{G}$-orbits on $X$. If
$S$ is a curve, the Kawamata subadjunction theorem (see
Theorem~\ref{theorem:Kawamata}, \cite[Theorem~1]{Kaw98}) implies
that $S$ is smooth, and we proceed with applying the
Nadel--Shokurov vanishing theorem, the Riemann--Roch theorem, the
Clifford theorem (see Theorem~\ref{theorem:Clifford},
\cite[Theorem~5.4]{Har77}, \cite[Section~III.1]{GrHaArbCor85}),
and the  Castelnuovo bound (see Theorem~\ref{theorem:Castelnuovo},
\cite[Theorem~6.4]{Har77}, \cite[Section~III.2]{GrHaArbCor85}).
Using these, we
obtain various restrictions on the number of irreducible
components of $S$, the degree of the curve $S$, and the genus of
the curve $S$. In many cases these restrictions appear to be
incompatible with a faithful action of $\bar{G}$ on $X$. Namely,
suppose that $X$ is one of the three threefolds from
Theorem~\ref{theorem:main} and $\bar{G}\cong\A_6$. Then all
possibilities for $S$ are ruled out by the above arguments, except
for the cases when either $X$ is the Segre cubic and $S$ is its
singular point, or when $X\cong\P^3$ and~$S$ is line whose
$\bar{G}$-orbit consists of $6$ lines. In the former case the
contradiction is obtained by brute force (see
Lemma~\ref{lemma:Segre-cubic-4}). So we have to deal with the case
when $X\cong\P^3$ and $S$ is line whose $\bar{G}$-orbit consists of
$6$ lines. If $S$ is not a center of canonical singularities of
the log pair $(\P^3,\lambda\M)$, then we use
Lemma~\ref{lemma:quasric-surface-Dokshitzer} to obtain a
contradiction. However, we can not obtain a contradiction here in
general, because $\P^3$ is not $\bar{G}$-birationally superrigid!
Indeed, for every line $L\subset \P^3$ whose $\bar{G}$-orbit
consists of $6$ lines, there exists a non-biregular
$\bar{G}$-equivariant birational involution $\P^3\to\P^3$
discovered by Todd in~\cite{To33} that is undefined along every
line in the $\bar{G}$-orbit of $L$ (see
Lemma~\ref{lemma:Todd-involution} for a description). There are
$12$ such lines in $\P^3$ in total. These $12$ lines split in
two $\bar{G}$-orbits, and each $\bar{G}$-orbit gives us a
non-biregular $\bar{G}$-equivariant birational involution
$\P^3\dasharrow\P^3$, which we denote by $\iota$ and $\iota^\prime$,
respectively.
Applying $\iota$ and $\iota^\prime$ to $\M$ if necessary, we may
assume that lines whose $\bar{G}$-orbit consist of $6$ lines are
not centers of canonical singularities of the log pair
$(X,\lambda\M)$ (see Lemma~\ref{lemma:P3-untwisting}). Combining
this with what we have already proved, we see that the singularities of
the log pair $(X,\lambda\M)$ are canonical up to the action of
$\langle\bar{G},\iota,\iota^\prime\rangle$, which implies that
$\P^3$ is $\bar{G}$-birationally rigid and
$\mathrm{Bir}^{\bar{G}}(\P^3)=\langle\mathrm{Aut}^{\bar{G}}(\P^3),\iota,\iota^\prime\rangle$.
Finally, we observe that
$$
\mathrm{Aut}^{\bar{G}}(\P^3)\cong\langle\bar{G},\iota\rangle\cong\langle\bar{G},\iota^\prime\rangle\cong\SS_6
$$
and use this to prove that
$\langle\mathrm{Aut}^{\bar{G}}(\P^3),\iota,\iota^\prime\rangle$ is
the~free product of two copies of $\SS_{6}$ with amalgamated
subgroup $\A_{6}$. Let us illustrate the described strategy by
proving

\begin{theorem}
\label{theorem:normalizers} Suppose that $\bar{G}$ is one of the
following subgroups: $\A_7$, $\SL_2(\mathbb{F}_{8})$,
$\PSp_4(\mathbb{F}_{3})$. Then~$V$ is
$\bar{G}$-birationally superrigid.
\end{theorem}

\begin{proof}
If $V\cong\P^3$, then either $\bar{G}\cong\A_7$, or
$\bar{G}\cong\PSp_4(\mathbb{F}_{3})$, and the proof of
Theorem~\ref{theorem:main} implies that $V$ is
$\bar{G}$-birationally superrigid, since $\A_7$ and
$\PSp_4(\mathbb{F}_{3})$ contain subgroups isomorphic to~$\A_6$.
In fact, the $\bar{G}$-birational superrigidity of the
threefold $V$ in the latter case is much easier to prove than
Theorem~\ref{theorem:main}, but we leave the details to the
reader.

It follows from Theorem~\ref{theorem:Yura-Cremona-3} and
Corollary~\ref{corollary:Bukrhardt} that we may assume that
$\bar{G}\cong\SL_2(\mathbb{F}_{8})$ and~$V$ is the~threefold
constructed in Example~\ref{example:genus-7}. Suppose that $V$ is
not $\bar{G}$-birationally superrigid. Then it follows from
\cite[Theorem~4.2]{Co95} or \cite[Theorem~1.4.1]{Ch05umn} that
there is a (non-empty) $\bar{G}$-invariant linear system
$\mathcal{M}$ on $V$ such that $\mathcal{M}$ does not have fixed
components, and $(V,\lambda\mathcal{M})$ is not canonical, where
$\lambda$ is a positive rational number such that
$\lambda\mathcal{M}\sim_{\mathbb{Q}} -K_{V}$. By
Corollary~\ref{corollary:mult-by-2} there is $\mu\in\mathbb{Q}$
such that $\mu<2\lambda$ and $(V,\mu\mathcal{M})$ is log canonical
and not Kawamata log terminal (see \cite[Definition~3.5]{Ko97}).

Let $S$ be a subvariety in $V$ that is a~minimal center of log
canonical singularities of the log pair $(V,\mu\mathcal{M})$ (see
\cite{Kaw97}, \cite{Kaw98}, \cite[Definition~2.8]{ChSh09}), and
let $Z$ be the $\bar{G}$-orbit of the subvariety $S$. Then $S$ is
either a point or a curve, since $\mathcal{M}$ does not have fixed
components. Moreover, the subvariety $Z\subset V$ is a disjoint
union of its components (see Lemma~\ref{lemma:centers},
\cite[Proposition~1.5]{Kaw97}).

Arguing as in the proof of \cite[Theorem~1.10]{Kaw97} or
\cite[Theorem~1]{Kaw98}, we~can construct a~$\bar{G}$-invariant
effective $\mathbb{Q}$-divisor $D$ on $V$ such that
$D\sim_{\mathbb{Q}}-\delta K_{V}$ for some positive rational
number $\delta<2$,  the~log pair $(V,D)$ is log canonical, the
subvariety $S$ is a~minimal center of log canonical singularities
of the log pair $(V, D)$, and the only centers of log canonical
singularities (see \cite[Definition~1.3]{Kaw97}) of the log pair
$(V, D)$ are irreducible components of $Z$
(see Corollary~\ref{corollary:mult-by-2}
and Lemma~\ref{lemma:Kawamata-Shokurov-trick}).

Let $\mathcal{I}_{Z}$ be its ideal sheaf. Put
$q=h^{0}(\mathcal{O}_{V}(-K_{V})\otimes\mathcal{I}_Z)=0$. Then
\begin{equation}
\label{equation:genus-7-equality}
h^{0}\Big(\mathcal{O}_{Z}\otimes\mathcal{O}_{V}\big(-K_{V}\big)\Big)=h^{0}\Big(\mathcal{O}_{V}\big(-K_{V}\big)\Big)-q=9-q\leqslant 9%
\end{equation}
by the Nadel--Shokurov vanishing theorem (see
Theorem~\ref{theorem:Shokurov-vanishing},
\cite[Theorem~9.4.8]{La04})
and the~Riemann--Roch theorem. Let $r$
be the number of irreducible components of the curve $Z$. Then
$r\not\in\{2,3,\ldots,8\}$, since $\bar{G}$ can not act transitively
on a finite set consisting of at least $2$ and at most $8$
elements (see \cite{Atlas}). Note that $Z=S$ if $r=1$.

Suppose that $S$ is a point. Then either $Z=S$ or $|Z|=r=9$ by
$(\ref{equation:genus-7-equality})$. If $Z=S$, then $\bar{G}$ must
act faithfully on the tangent space to $V$ at the point $S$, which
is impossible, since $\SL_2(\mathbb{F}_{8})$ has no non-trivial
three-dimensional representations (see \cite{Atlas}). Thus, we see
that $|Z|=9$. Let $F_{56}$ be the stabilizer subgroup in $\bar{G}$
of the point $S$. Then $|F_{56}|=|\bar{G}|/9=56$, and hence one
has \mbox{$F_{56}\cong\mathbb{Z}_2^3:\mathbb{Z}_{7}$} (see~\cite{Atlas}).
On the other
hand, the group $F_{56}$ must faithfully act on the tangent space
to $V$ at the point $S$, which is impossible, since
$\mathbb{Z}_2^3:\mathbb{Z}_{7}$ has no faithful three-dimensional
representations.

Thus, we see that $S$ is a~curve. Put $d=-K_{V}\cdot C$. By the
Kawamata subadjunction theorem (see
Theorem~\ref{theorem:Kawamata}, \cite[Theorem~1]{Kaw98}),
the~curve $S$ is a~smooth curve of genus $g$ such that $d>2g-2$.
In particular, the divisor $-K_{V}\vert_{S}$ is not special. Thus,
it follows from $(\ref{equation:genus-7-equality})$ and
the~Riemann--Roch theorem that
\begin{equation}
\label{equation:genus-7} 9-q=r(d-g+1),
\end{equation}
where either $r=1$, or $r=9$. If $r=9$, then it follows from
$(\ref{equation:genus-7})$ that $d=g$ and $q=0$, which is
impossible, since $d>2g-2$. Thus, we see that $r=1$, which simply
means that $Z=S$. Then $\bar{G}$ acts faithfully on $S$, since
$\bar{G}$ is simple and there are no points in $V$ that are
$\bar{G}$-invariant. Moreover, it follows from
$(\ref{equation:genus-7})$ that $d=8+g-q$. Keeping in mind that
$d>2g-2$, we see that $g<10$. Then $g=7$ by the~refined Hurwitz bound
(see~\cite[Theorem~3.17]{Bre00},
Theorem~\ref{theorem:Hurwitz-bound}).
In particular, we see that
$d=15-q$, since $d=8+g-q$.

Let $M_1$ and $M_2$ be general surfaces in $\mathcal{M}$, and let
$H$ be a general surface in $|-K_{V}|$. Then it follows from
\cite[Theorem~3.1]{Co00} that $\mathrm{mult}_{S}(M_{1}\cdot
M_{2})\geqslant 4/\mu^2$, where $\mu<2\lambda$. Thus, one has
$$
\frac{12}{\lambda^{2}}=M_{1}\cdot M_{2}\cdot H\geqslant \sum_{O\in H\cap S} \mathrm{mult}_{S}(M_{1}\cdot M_{2})\geqslant \frac{4|H\cap S|}{\mu^{2}}>\frac{|H\cap S|}{\lambda^{2}},%
$$
which implies that $|H\cap S|<12$. But $|H\cap S|=d=15-q$, which
implies that $q\geqslant 4$. The latter is impossible, since
$\SL_2(\mathbb{F}_{8})$ has no non-trivial representations of
dimension at most $6$.
\end{proof}

The reader may find it useful to compare the proof of
Theorem~\ref{theorem:Bukrhardt} (that uses a classical approach to
birational rigidity) with
the proof of Theorem~\ref{theorem:normalizers} (that uses our new approach).
Note that the proof of Theorem~\ref{theorem:normalizers}
does not use any information about the geometry of the variety $V$
acted on by the group~$\SL_2(\F_8)$.
Applying Theorem~\ref{theorem:normalizer} and
Corollary~\ref{corollary:normalizers} and using classification of
primitive subgroups in $\SL_4(\C)$ and $\SL_5(\C)$ (see
\cite{Fe71}), we obtain

\begin{corollary}
\label{corollary:Serre} Let $\bar{G}$ be a subgroup in
$\mathrm{Cr}_3(\C)$ that is isomorphic to either $\A_7$, or
$\PSp_4(\mathbb{F}_{3})$. If $\bar{G}\cong\A_7$, then its
normalizer in $\mathrm{Cr}_3(\C)$ is isomorphic to $\SS_7$. If
$\bar{G}\cong\PSp_4(\mathbb{F}_{3})$,
then its normalizer in $\mathrm{Cr}_3(\C)$ is $\bar{G}$ itself.
\end{corollary}

The plan of the paper is as follows. In
Section~\ref{section:preliminaries} we collect the well-known
preliminary results on the log canonical singularities, make some
remarks on varieties with group action and recall some general
classical theorems. In Section~\ref{section:groups} we collect the
necessary facts about the structure of the groups~$\A_6$
and~$2.\A_6$ and the relevant representation theory. In
Sections~\ref{section:space}, \ref{section:cubic}
and~\ref{section:quadric} we study the projective space~$\P^3$,
the Segre cubic and the smooth quadric threefold according to the
plan explained above, proving Theorems~\ref{theorem:A6},
\ref{theorem:auxiliary-amalgamated}, \ref{theorem:Segre-cubic} and
\ref{theorem:quadric-threefold}, respectively. These four theorems
imply Theorem~\ref{theorem:main}. In
Appendix~\ref{section:Klein-cubic}, we~answer (negatively)
a~question on the existence of equivariant birational maps between
two certain remarkable threefolds acted on by the group
$\bar{\Gamma}\cong\PSL_2(\F_{11})$ posed in
\cite[Remark~2.10]{Pr09}. Appendix~\ref{section:del-Pezzo},
written by Yu.\,Prokhorov, classifies terminal threefolds $X$
acted on by the group $\bar{G}\cong\A_6$ such that $X$ admits a
$\bar{G}$-Mori fibration $\pi:X\to\P^1$.

\smallskip
We would like to thank T.\,Dokchitser, I.\,Dolgachev,
Yu.\,Pro\-kho\-rov, L.\,Rybnikov and E.\,Smirnov for very useful
discussions and comments. We would like to thank the referee for
many valuable and useful comments that helped us to improve the
paper significantly.

\section{Preliminaries}%
\label{section:preliminaries}

Throughout the paper we use the standard language of the
singularities of pairs (see \cite{Ko97} and
\cite[Section~6]{CKS04}). By \emph{strictly log canonical} singularities
we mean log canonical singularities that are not Kawamata log
terminal (see \cite[Definition~3.5]{Ko97}).

Let $X$ be a~variety with at most Kawamata log terminal
singularities (see \cite[Definition~3.5]{Ko97}), let $B_{X}$ be a
formal linear combinations with non-negative rational coefficients
of non-empty (but not necessarily complete) linear systems on $X$. Then we may
write $B_{X}=\sum_{i=1}^{r}a_{i}B_{i}$, where
$a_{i}\in\mathbb{Q}$, and~$B_{i}$ is either a~prime Weil divisor
on the~variety $X$, or a linear system without fixed components.
We assume for simplicity that $B_{i}\ne B_{j}$ if $i\ne j$.

\begin{remark}
\label{remark:mobile-boundaries} Usually $B_{X}$ is assumed to be
an effective $\mathbb{Q}$-divisor. But sometimes it is necessary to
assume that some or all components of $B_{X}$ are linear systems
without fixed components (see \cite[Section~6]{CKS04}). We will
need such log pair throughout the paper, but do not want to stress
on using them, since all properties of standard log pairs we plan
to use hold for these generalized log pairs (see
\cite[Definition~4.6]{Ko97} and \cite[Definition~6.16]{CKS04}).
Moreover, we can always treat $B_{X}$ as a $\mathbb{Q}$-divisor by
replacing each its mobile component by a weighted sum of its
sufficiently general members (see \cite[Theorem~4.8]{Ko97}).
\end{remark}

Suppose that $K_{X}+B_{X}$ is $\mathbb{Q}$-Cartier. Let
$\pi\colon\bar{X}\to X$ be a~birational morphism such that
$\bar{X}$ is smooth. Then
$$
K_{\bar{X}}+\sum_{i=1}^{r}a_{i}\bar{B}_{i}\sim_{\mathbb{Q}}\pi^{*}\Big(K_{X}+B_{X}\Big)+\sum_{i=1}^{m}d_{i}E_{i},
$$
where $\bar{B}_{i}$ is the~proper transforms of $B_{i}$ on
the~variety $\bar{X}$, and $E_{i}$ is an exceptional divisor of
the~morphism $\pi$, and $d_{i}$ is a~rational number. We may
assume that $\sum_{i=1}^{r}\bar{B}_{i}+\sum_{i=1}^{m}E_{i}$ is
a~divisor with simple normal crossing and every mobile linear
system among $\bar{B}_{1},\ldots,\bar{B}_{r}$ (if any) is free
from base points.

\begin{lemma}\label{lemma:mult-by-2}
Suppose that $X$ has terminal singularities. Let $n$ be a positive
integer such that~$nK_{X}$ is a Cartier divisor. Then $(X,(1+n)\lambda
B_{X})$ is not log canonical if $(X,\lambda B_{X})$ is not
canonical.
\end{lemma}

\begin{proof}
There are positive rational numbers $c_{1},\ldots,c_r$ such that
$K_{\bar{X}}\sim_{\mathbb{Q}}\pi^{*}(K_{X})+\sum_{i=1}^{m}c_{i}E_{i}$,
because $X$ has terminal singularities and, in particular, $K_{X}$
is a $\mathbb{Q}$-Cartier divisor. Moreover, the numbers
$nc_{1},\ldots,nc_r$ are positive integers, since $nK_{X}$ is a
Cartier divisor. On the other hand, we have
$$
\sum_{i=1}^{r}a_{i}\bar{B}_{i}\sim_{\mathbb{Q}}\pi^{*}\big(B_{X}\big)-\sum_{i=1}^{m}m_{i}E_{i}
$$
for some non-negative rational numbers $m_{1},\ldots,m_r$. Then
$d_i=c_i-m_i$ for every $i\in\{1,\ldots,r\}$.

Suppose that $(X,\lambda B_{X})$ is not canonical. Then there
exists $s\in\{1,\ldots,m\}$ such that $c_s-m_s=d_s<0$ (see
\cite[Definition~6.16]{CKS04}). Thus, we have
$$
c_s-(1+n)m_s=c_s-m_s-nm_s<-nm_s<-1,
$$
because $m_s>c_s\geqslant 1/n$, since $nc_s$ is a positive
integer. On the other hand, we have
$$
K_{\bar{X}}+(1+n)\sum_{i=1}^{r}a_{i}\bar{B}_{i}\sim_{\mathbb{Q}}\pi^{*}\Big(K_{X}+(1+n)B_{X}\Big)+\sum_{i=1}^{m}\big(c_i-(1+n)m_i\big)E_{i},
$$
where $c_s-(1+n)m_s<-1$. Therefore, the log pair $(X,(1+n)\lambda
B_{X})$ is not log canonical.
\end{proof}

\begin{corollary}\label{corollary:mult-by-2}
Suppose that $X$ has terminal Gorenstein singularities and
$(X,\lambda B_{X})$ is not canonical. Then $(X,2\lambda B_{X})$ is
not log canonical.
\end{corollary}

Suppose
that $(X,B_{X})$ is log canonical. Put $\hat{a}_{i}=a_{i}$
if $B_{i}$ is a prime Weil divisor, and $\hat{a}_{i}=0$ if~$B_{i}$
is a linear system without fixed components. Put $\mathcal{I}(X,
B_{X})=\pi_{*}(\sum_{i=1}^{m}\lceil d_{i}\rceil E_{i}-
\sum_{i=1}^{r}\lfloor \hat{a}_{i}\rfloor \bar{B}_{i})$.

\begin{theorem}[{\cite[Theorem~9.4.8]{La04}}]
\label{theorem:Shokurov-vanishing} Let $H$ be a~nef and big
$\mathbb{Q}$-divisor on $X$, and let $D$ be a Cartier divisor on
$X$ such that $D\equiv K_{X}+B_{X}+H$. Then $h^{i}(\mathcal{I}(X,
B_{X})\otimes D)=0$ for every $i\geqslant 1$.
\end{theorem}

Let $\mathcal{L}(X, B_{X})$ be a~subscheme of $X$ that corresponds
to the~ideal sheaf $\mathcal{I}(X, B_{X})$. Put $\mathrm{LCS}(X,
B_{X})=\mathrm{Supp}(\mathcal{L}(X, B_{X}))$. Note that
$\mathcal{L}(X,B_{X})$ is reduced, because $(X,B_{X})$ is log
canonical.

\begin{theorem}[{\cite[Lemma~5.7]{Sho93}, \cite[Theorem~1.4]{Kaw97}}]
\label{theorem:connectedness} Let $\zeta\colon X\to Z$ be
a~surjective morphism with connected fibers, and let $F$ be
a~fiber of $\zeta$. Then the~locus $\mathrm{LCS}(X, B_{X})\cap F$
is connected if $-(K_{X}+B_{X})$ is $\eta$-nef and $\eta$-big.
\end{theorem}

Recall that there are standard names for $\mathcal{I}(X, B_{X})$,
$\mathcal{L}(X, B_{X})$ and $\mathrm{LCS}(X,B_{X})$. Namely, the
ideal $\mathcal{I}(X, B_{X})$ is known as the~multiplier ideal
sheaf (see \cite[Section~9.2]{La04}), the subscheme
$\mathcal{L}(X, B_{X})$ is known as the~log canonical
singularities subscheme (see \cite[Definition~1.7.5]{Ch05umn}),
and the locus $\mathrm{LCS}(X,B_{X})$ is known as the~locus of log
canonical singularities (see \cite[Definition~3.14]{Sho93}).

\begin{theorem}[{\cite[Theorem~6.29]{CKS04}}]
\label{theorem:adjunction} Suppose that $a_{1}=1$, and $B_{1}$ is
a prime Weil divisor. Suppose that $B_{1}$ is~a~Cartier divisor on
$X$, and $B_{1}$ is a normal variety that has at most Kawamata log
terminal singularities. Then $(X, B_{X})$ is log canonical along
$B_{1}$ if and only if the~log pair $(B_{1},
\sum_{i=2}^{r}a_{i}B_{i}\vert_{B_{1}})$ is log canonical.
\end{theorem}

Let $Z$ be a~center of log canonical singularities of the~log pair
$(X, B_{X})$ (see \cite[Definition~1.3]{Kaw97}), and let
$\mathbb{LCS}(X, B_{X})$ be the~set of all centers of log
canonical singularities of the~log pair $(X, B_{X})$.

\begin{lemma}[{\cite[Proposition~1.5]{Kaw97}}]
\label{lemma:centers} Let $Z$ and $Z^{\prime}$ be elements of
the~set $\mathbb{LCS}(X, B_{X})$, and let~$Y$ be any irreducible
reduced component of $Z\cdot Z^\prime$. Then $Y\in\mathbb{LCS}(X,
B_{X})$.
\end{lemma}

Suppose that $Z$ is a~minimal center in $\mathbb{LCS}(X, B_{X})$
(see \cite{Kaw97}, \cite{Kaw98}, \cite[Definition~2.8]{ChSh09}).

\begin{theorem}[{\cite[Theorem~1]{Kaw98}}]
\label{theorem:Kawamata} Let $\Delta$ be an~ample
$\mathbb{Q}$-Cartier $\mathbb{Q}$-divisor on $X$. Then
\begin{itemize}
\item the~variety $Z$ is normal and has at most rational singularities,%
\item there exists an~effective $\mathbb{Q}$-divisor $B_{Z}$ on
$Z$ such that $(K_{X}+B_{X}+\Delta)\vert_{Z}\sim_{\mathbb{Q}}
K_{Z}+B_{Z}$, and $(Z,B_{Z})$ has Kawamata log terminal
singularities.
\end{itemize}
\end{theorem}

Let $\bar{G}\subseteq\mathrm{Aut}(X)$ be a~finite subgroup.

\begin{lemma}
\label{lemma:stabilizers} Suppose that $X$ is a~curve. Let
$\Sigma\subset X$ be a~$\bar{G}$-orbit. If $|\Sigma|=1$, then
$\bar{G}$ is cyclic. If $\bar{G}\cong\A_5$, then $|\Sigma|\in\{12,
20, 30, 60\}$. If $\bar{G}\cong\A_6$, then $|\Sigma|\in\{60, 72,
90, 120, 180, 360\}$.
\end{lemma}

\begin{proof}
If $\Sigma$ is a~point, then $\bar{G}$ acts faithfully on
the~tangent space to $X$ at the~point $\Sigma$.
\end{proof}

\begin{lemma}
\label{lemma:sporadic-genera} Suppose that $X$ is a~curve of genus
$g$ and $\bar{G}\cong\A_6$. If $g\leqslant 34$, then
$g\in\{10,16,19,25,31\}$. If $g=10$, then $X$ does not contain
$\bar{G}$-orbits of length $120$.
\end{lemma}

\begin{proof}
Let $\bar{F}\subset\bar{G}$ be a~stabilizer of a~point in $X$.
Then $\bar{F}\cong\mathbb{Z}_k$ for some $k\in\{1,2,3,4,5,6\}$ by
Lemma~\ref{lemma:stabilizers}. Put $\bar{X}=X\slash\bar{G}$. Then
$\bar{X}$ is a~smooth curve of genus $\bar{g}$.
The~Riemann--Hurwitz formula gives
$$
2g-2=360\big(2\bar{g}-2\big)+180a_2+240a_3+270a_4+288a_5+300a_6,
$$
where $a_k$ is the~number of $\bar{G}$-orbits in $X$ with
a~stabilizer of a~point isomorphic to $\mathbb{Z}_k$.

Suppose that $g\leqslant 34$. Note that $g\neq 0$ by
the~classification of finite subgroups of $\PGL_2(\mathbb{C})$,
and $g\neq 1$ since $\bar{G}$ is non-solvable. Since
$a_k\geqslant 0$, one has $\bar{g}=0$, and
\begin{equation}
\label{eq:RH}
2g-2=-720+180a_2+240a_3+270a_4+288a_5+300a_6,
\end{equation}
which implies that $g\in\{10,16,19,25,31\}$. The only solution to
(\ref{eq:RH}) for $g=10$ is $(a_2, a_3, a_4, a_5, a_6)=(1, 0, 1,
1, 0)$, which completes the~proof.
\end{proof}

\begin{remark}
We do not claim that every case listed in Lemma~\ref{lemma:sporadic-genera}
is realized. The obtained restrictions on the genus are enough for
our purposes.
\end{remark}

Suppose, in addition, that $B_{X}$ is $\bar{G}$-invariant. Then
$g(Z)\in\mathbb{LCS}(X,B_{X})$ for every $g\in\bar{G}$, and
the~locus $\mathrm{LCS}(X,B_{X})$ is $\bar{G}$-invariant. It
follows from Lemma~\ref{lemma:centers} that $Z\cap g(Z)\ne
\varnothing$ if and only if $Z=g(Z)$ for every $g\in\bar{G}$,
because $Z$ is a~minimal center in $\mathbb{LCS}(X, B_{X})$.

\begin{lemma}
\label{lemma:Kawamata-Shokurov-trick} Suppose that the~divisor
$B_{X}$ is ample. Let $\epsilon$ be an~arbitrary rational number
such that $\epsilon>0$. Then there exists
an~effective~$\bar{G}$-in\-va\-riant $\mathbb{Q}$-divisor
$D$~on~the~variety~$X$~such~that the set $\mathbb{LCS}(X, D)$
consists of all components of the $\bar{G}$-orbit of $Z$, the~log
pair $(X,D)$ is log canonical, and the~equivalence
$D\sim_{\mathbb{Q}} (1+\epsilon)B_{X}$ holds.
\end{lemma}

\begin{proof}
See the~proofs of \cite[Theorem~1.10]{Kaw97},
\cite[Theorem~1]{Kaw98}, \cite[Lemma~2.11]{ChSh09}.
\end{proof}

Suppose, in addition, that $X$ is a~Fano variety. Put
$$
\mathrm{lct}\Big(X,\bar{G}\Big)=\mathrm{sup}\left\{\lambda\in\mathbb{Q}\ \left|%
\aligned
&\text{the~log pair}\ \big(X, \lambda D\big)\ \text{has log canonical singularities}\\
&\text{for every $\bar{G}$-invariant effective $\mathbb{Q}$-divisor}\ D\sim_{\mathbb{Q}} -K_{X}\\
\endaligned\right.\right\}\in\mathbb{R}.%
$$

\begin{remark}
\label{remark:alpha-invariant} If $X$ is smooth, then it follows
from \cite[Theorem~A.3]{ChSh08c}~that
$\mathrm{lct}(X,\bar{G})=\alpha_{\bar{G}}(X)$, where
$\alpha_{\bar{G}}(X)$ is the~$\bar{G}$-invariant
$\alpha$-invariant of the~variety $X$ introduced in \cite{Ti87}
and \cite{TiYa87}.
\end{remark}

Suppose that $X\cong\mathbb{P}^{1}\times\mathbb{P}^{1}$.

\begin{lemma}
\label{lemma:orbits-for-standard-A5-in-A6} Suppose that
$\bar{G}\cong\A_5$. Let $\Sigma\subset X$ be
a~$\bar{G}$-invariant subset. Then $|\Sigma|\geqslant 12$.
\end{lemma}

\begin{proof}
The required assertion follows from Lemma~\ref{lemma:stabilizers}.
\end{proof}

Let $L_{1}$ and $L_{2}$ be fibers of the projection to the first
and the second factor, respectively.

\begin{lemma}
\label{lemma:quasric-surface-lct} The~inequality
$\mathrm{lct}(X,\bar{G})\geqslant 1$ holds if and only if
the~linear systems $|L_{1}|$, $|L_{2}|$, $|L_{1}+L_{2}|$ contain
no $\bar{G}$-invariant curves.
\end{lemma}

\begin{proof}
If there is a~$\bar{G}$-invariant curve in $|L_{1}|$ or $|L_{2}|$
or $|L_{1}+L_{2}|$, then $\mathrm{lct}(X,\bar{G})=1/2$ (see
\cite[Theorem~1.7]{Ch07b}).

Suppose that $|L_{1}|$, $|L_{2}|$, $|L_{1}+L_{2}|$ contain no
$\bar{G}$-invariant curves, but $\mathrm{lct}(X,\bar{G})<1$. There
are $\lambda\in\mathbb{Q}$~and~an~effective~$\bar{G}$-invariant
$\mathbb{Q}$-divisor $D$ on $X$ such that $\lambda<1$ and
$D\sim_{\mathbb{Q}}2(L_{1}+L_{2})$, but $(X,\lambda D)$ is not
Kawamata log terminal. We may assume that $(X,\lambda D)$ is log
canonical.

Suppose that there is a~$\bar{G}$-invariant curve $C\subset X$
such that $\lambda D=\mu C+\Omega$, where $\mu\geqslant 1$, and
$\Omega$ is an effective $\bar{G}$-invariant $\mathbb{Q}$-divisor,
whose support does not contain any component of the~curve $C$.
Then $C\sim aL_{1}+bL_{2}$ for some non-negative integers $a$ and
$b$. Then either $a\geqslant 2$ or $b\geqslant 2$. But
$$
2>2\lambda=\lambda \Big(D\cdot L_{1}\Big)=\Big(\mu C+\Omega\Big)\cdot L_{1}\geqslant C\cdot L_{1}=b,%
$$
which implies that $b\leqslant 1$. Similarly, we see that
$a=C\cdot L_{2}\leqslant 1$, which is a~contradiction.

We see that the~locus $\mathrm{LCS}(X,\lambda D)$ is~a~finite
$\bar{G}$-invariant set. Thus, the~locus $\mathrm{LCS}(X,\lambda
D)$ consists of a~single point $O\in X$ by
Theorem~\ref{theorem:connectedness}. Let $H$ be the~unique curve
in $|L_{1}+L_{2}|$ that is singular at $O$. Then $H$ must be
$\bar{G}$-invariant, which is a~ contradiction, because
the~linear system $|L_{1}+L_{2}|$ contains no $\bar{G}$-invariant
divisors.
\end{proof}

Let us identify $X\cong\mathbb{P}^{1}\times\mathbb{P}^{1}$ with
a~smooth quadric surface in $\mathbb{P}^{3}$, and let us identify
$\bar{G}$ with a~subgroup in $\mathrm{Aut}(\mathbb{P}^{3})$ such
that $X$ is $\bar{G}$-invariant. Let
$\phi\colon\SL_4(\mathbb{C})\to\mathrm{Aut}(\mathbb{P}^{3})$~be
the~natural projection. Then there is a~finite subgroup
$G\subset\SL_4(\mathbb{C})$~such~that $\bar{G}=\phi(G)$. Put
$W=\mathbb{C}^{4}$.

\begin{lemma}
\label{lemma:quasric-surface-Dokshitzer} Suppose that
$\bar{G}\cong G\cong\A_5$, and $W\cong\C^4$ is an irreducible
representation of the~group $G$. Assume that
$B_{X}\sim_{\mathbb{Q}}mL_{1}+\big(8-m\big)L_{2}$, where $m$ is
a~rational number such that $1<m<2$. Then the locus
$\mathrm{LCS}(X,B_{X})$ consists of finitely many points.
\end{lemma}

\begin{proof}
Let $U$ be a two-dimensional representation of the group
$2.\A_5$. Then the center of the group $2.\A_5$
acts trivially on $U\otimes U^{\vee}$,
and $U\otimes U^{\vee}\cong W$ as representations of $\A_5$.
Therefore, one has
$$H^0(X, \O_X(L_1+nL_2))\cong U\otimes\mathrm{Sym}^n(U^{\vee})$$
as representations of $\A_5$.

Suppose that the locus $\mathrm{LCS}(X,B_{X})$ does not consist of
finitely many points. Then there is a~$\bar{G}$-invariant curve
$C\subset X$ such that $B_{X}=\mu C+\Omega$, where $\mu\geqslant
1$ and $\Omega$ is an effective $\bar{G}$-invariant
$\mathbb{Q}$-divisor, whose support does not contain any component
of the~curve $C$. Then $C\sim aL_{1}+bL_{2}$ for some non-negative
integers~$a$ and~$b$. Then either $a\geqslant 2$ or $b\geqslant
2$, since the linear systems $|L_1|$, $|L_2|$ and $|L_1+L_2|$ do
not contain $\bar{G}$ invariant divisors (indeed, $W$ is an
irreducible $\bar{G}$ representations, so that $\P(W)$ does not
contain $\bar{G}$-invariant hyperplanes, and both $U$ and
$U^{\vee}$ are irreducible $G$-representations, so that
$\P(U)\cong\P^1$ and $\P(U^{\vee})\cong\P^1$ do not contain
$\bar{G}$-invariant points). On the other hand
$$
2>m=B_{X}\cdot L_{2}=\Big(\mu C+\Omega\Big)\cdot L_{2}\geqslant C\cdot L_{2}=a,%
$$
which implies that either $a=0$ or $a=1$. Similarly, we see that
$$
7>8-m=B_{X}\cdot L_{1}=\Big(\mu C+\Omega\Big)\cdot L_{1}\geqslant C\cdot L_{1}=b,%
$$
which implies that $b\leqslant 6$. If $a=0$, then $C\in |bL_2|$,
and thus $C$ is $\bar{G}$-invariant. The latter is implied by
Lemma~\ref{lemma:stabilizers} (actually,
Lemma~\ref{lemma:stabilizers} shows that the linear system
$|bL_2|$ does not contain $\bar{G}$-invariant curves for $b\le
11$). If $a=1$, then $C\in |L_1+bL_2|$. Now a direct
computation\footnote{We used the Magma software~\cite{Magma} to
carry it out.} shows that $U\otimes\mathrm{Sym}^n(U^{\vee})$ does
not contain one-dimensional $G$-invariant subspaces for $b\le 6$.
The obtained contradiction completes the proof.
\end{proof}

Let $C$ be a~smooth irreducible curve of genus $g$.

\begin{theorem}
\label{theorem:Hurwitz-bound} Suppose that $g\geqslant 2$. Then
$$
\big|\mathrm{Aut}(C)\big|\leqslant\left\{\aligned
&320\ \text {if}\ g=9,\\
&432\ \text {if}\ g=10,\\
&240\ \text {if}\ g=11,\\
&120\ \text {if}\ g=12,\\
&360\ \text {if}\ g=13,\\
&504\ \text {if}\ g=15,\\
\endaligned
\right.
$$
and the~inequality $|\mathrm{Aut}(C)|\leqslant 84(g-1)$ holds for
any $g\geqslant 2$.
\end{theorem}

\begin{proof}
The inequality $|\mathrm{Aut}(C)|\leqslant 84(g-1)$ is the~famous
Hurwitz bound (see~\cite[Theorem~3.17]{Bre00}), the~exact bounds
for particular genera may be found in~\cite[Table~13]{Bre00}.
\end{proof}

Let $D$ be an effective divisor on the~curve $C$.

\begin{theorem}[{\cite[Theorem~5.4]{Har77}}]
\label{theorem:Clifford} If $h^1(\mathcal{O}_{C}(D))\ne 0$, then
$$h^0(\mathcal{O}_{C}(D))\leqslant\frac{\mathrm{deg}(D)}{2}+1.$$
\end{theorem}

Suppose that there is an embedding $\zeta\colon
C\to\mathbb{P}^{n}$ such that $\zeta(C)$ is a~curve of degree $d$.

\begin{theorem}[{\cite[Section~III.2]{GrHaArbCor85}}]
\label{theorem:Castelnuovo} If $\zeta(C)$ is not contained in any
hyperplane in~$\mathbb{P}^{n}$, then
$$
g\leqslant \frac{m(m-1)(n-1)}{2}+m\varepsilon,
$$
where $m=\lfloor{(d-1)/(n-1)}\rfloor$ and
$\varepsilon=d-1-m(n-1)$.
\end{theorem}

\section{Alternating group}
\label{section:groups}

Let $\SS_6$ be the~group of permutations of the~set
$\{1,2,3,4,5,6\}$.

\begin{definition}
\label{definition:standard-subgroups} Let $H\subset\A_{6}$
be a~subgroup that is isomorphic to $\A_5$, $\SS_4$
or $\A_4$. Then~we~say~that the~embedding
$H\subset\A_{6}$ is standard if one of the~following
conditions is satisfied:
\begin{itemize}
\item if $H\cong\A_5$ and $H$ is conjugate to the~subgroup of even permutations of  $\{1,2,3,4,5\}$,%

\item if $H\cong\SS_4$ and $H$ is conjugate to the~subgroup of
permutations of the~set $\{1,2,3,4\}$
such that the~odd permutations of $\{1,2,3,4\}$ are twisted by the~transposition $(5,6)$,%

\item if $H\cong\A_4$ and $H$ is conjugate to the~subgroup of even permutations of  $\{1,2,3,4\}$.%
\end{itemize}
We say that the~embedding $H\subset\A_{6}$ is
non-standard if it is not standard\footnote{Note that the
non-standard and standard embeddings of the~group $H$ into
$\A_{6}$ are interchanged by
$\mathrm{Aut}(\A_6)$.}.
\end{definition}

Let $2.\A_{6}$ be the~group such that there exists
a~non-splitting~exact sequence of groups
$$
\xymatrix{&1\ar@{->}[rr]&&\mathbb{Z}_{2}\ar@{->}[rr]^{\alpha}&&2.\A_{6}\ar@{->}[rr]^{\beta}&&\A_{6}\ar@{->}[rr]&&1}.%
$$
Note that $\alpha(\mathbb{Z}_{2})$ is the~center of
the~group $2.\A_{6}$.

\begin{definition}
\label{definition:standard-central-extension} Let $H\subset
2.\A_{6}$ such that $\beta(H)$ is isomorphic to $\A_5$, $\SS_4$ or
$\A_4$. Then we say that the~embedding $H\subset 2.\A_{6}$ is
standard if $\beta(H)\subset\A_{6}$ is standard, and we say that
the~embedding $H\subset 2.\A_{6}$ is non-standard if
$\beta(H)\subset\A_{6}$ is non-standard.
\end{definition}

Let $\bar{G}$ be a~finite subgroup in
$\mathrm{Aut}(\mathbb{P}^{3})$ such that
$\bar{G}\cong\A_{6}$.

\begin{lemma}[\cite{Atlas}]
\label{lemma:max-subgroup} Every~maximal proper subgroup of
the~group $\bar{G}\cong\A_{6}$ is isomorphic to either to $\A_5$,
or $\SS_4$, or
$(\mathbb{Z}_3\times\mathbb{Z}_3)\rtimes\mathbb{Z}_4$. Moreover,
up to conjugation, the~group $\bar{G}$ contains one subgroup
isomorphic to
$(\mathbb{Z}_3\times\mathbb{Z}_3)\rtimes\mathbb{Z}_4$,  one
subgroup isomorphic to
$(\mathbb{Z}_3\times\mathbb{Z}_3)\rtimes\mathbb{Z}_2$, two
subgroups isomorphic to $\SS_4$ (respectively, $\A_5$, $\A_{4}$).
\end{lemma}

Put $W=\mathbb{C}^{4}$. There exists a~finite subgroup
$G\subset\SL_4(\mathbb{C})$~such~that  $G\cong 2.\A_{6}$ and
$\bar{G}=\phi(G)\subset\mathrm{Aut}(\mathbb{P}^{3})\cong
\PGL_4(\mathbb{C})$ where
$\phi\colon\SL_4(\mathbb{C})\to\mathrm{Aut}(\mathbb{P}^{3})$~is~the~natural
projection.

\begin{remark}
\label{remark:irreducible} The group $G$ has two irreducible
four-dimensional representations (see~\cite{Atlas}), which implies
that we may identify $W$ with one of them, because another one
differs from $W$ by an~outer automorphism of the group $G$. Thus,
we may assume that the~natural action of the group $G$ on
$\Lambda^2\big(W\big)\cong\mathbb{C}^6$ arises from
the~permutation representation of the~group $\bar{G}\cong\A_6$.
\end{remark}

Let $\bar{F}\subset\bar{G}$ be a~subgroup, and let $F\subset G$ be
 a~subgroup such that $\phi(F)=\bar{F}$.

\begin{remark}
\label{remark:standard} Suppose that $F\cong
2.\A_{5}\subset 2.\A_6$. Then
$$
\mathbb{C}^4\cong W\cong\left\{\aligned
&U\oplus U^{\prime}\ \text{if $F\subset G$ is standard},\\
&\mathrm{Sym}^3\big(U\big)\cong\mathrm{Sym}^3\big(U^{\prime}\big)\
\text{if $F\subset G$ is non-standard},\\
\endaligned
\right.
$$
where $U$ and $U^{\prime}$ are different two-dimensional
representations of the~group $F\cong 2.\A_{5}$.
\end{remark}

\begin{lemma}
\label{lemma:A5-invariant-quartic} Suppose that $F\cong
2.\A_5$ and $F\subset G$ is a~non-standard embedding. Then
there exists an~irreducible $\bar{F}$-invariant
smooth rational cubic curve $Z\subset\mathbb{P}^{3}$.
\end{lemma}

\begin{proof}
The required assertion follows from Remark~\ref{remark:standard}.
\end{proof}

Let $z$ and $e$ be the~non-trivial element in the~center of $G$
and the~identity element, respectively.

\begin{lemma}[{cf.~\cite[Lemma~4.7]{ChSh09}}]
\label{lemma:small-semiinvariants} Any semi-invariant of the~group
$G$ is its invariant. The~group $G$ does not have in\-va\-ri\-ants
of odd degree, as well as invariants of degree
at most $7$. On the other hand,
the~group $G$ has two linearly independent invariants
of degree $8$.
\end{lemma}

\begin{proof}
Semi-invariants of the~group $G$ are its invariants, because
the~center of the~group $G$ is contained in its~commutator, and
the~group $\bar{G}$ is a~simple non-abelian group.

The~group $G$ does not have in\-va\-ri\-ants of odd degree,
because $G$ contains a~scalar matrix whose non-zero
entries~are~$-1$. Therefore, to prove that $G$ has no invariants
of degree at most~$7$, it~is~enough to show that $G$ does not have
invariants of degree $4$ and $6$.

Let $\chi_{m}$ be the~character of the~representation
$\mathrm{Sym}^{m}(W)$ (cf. Remark~\ref{remark:irreducible}).
Put $\chi=\chi_{1}$.

The~values of the~characters $\chi$, $\chi_{4}$, $\chi_6$ and
$\chi_8$ are listed in the~following table:
\begin{center}\renewcommand\arraystretch{1.3}
\begin{tabular}{|c|c|c|c|c|c|c|c|c|c|c|}
\hline & $[5,1]_{10}$ & $[5,1]_5$ & $[4,2]_8$ & $[3,3]_6$ &
$[3,3]_3$ &
$[3,1,1,1]_6$ & $[3,1,1,1]_3$ & $[2,2,1,1]_4$ & $z$ & $e$\\
\hline
\# & $144$ & $144$ & $180$ & $40$ & $40$ & $40$ & $40$ & $90$ & $1$ & $1$\\
\hline
$\chi$ & $1$ & $-1$ & $0$ & $-1$ & $1$ & $2$ & $-2$ & $0$ & $-4$ & $4$\\
\hline
$\chi_{4}$ & $0$ & $0$ & $-1$ & $2$ & $2$ & $-4$ & $-4$ & $3$ & $35$ & $35$\\
\hline
$\chi_{6}$ & $-1$ & $-1$ & $0$ & $3$ & $3$ & $3$ & $3$ & $-4$ & $84$ & $84$\\
\hline
$\chi_{8}$ & $0$ & $0$ & $1$ & $3$ & $3$ & $3$ & $3$ & $5$ & $165$ & $165$\\
\hline
\end{tabular}
\end{center}
where the~first row lists the~types of the~elements in $G$ (for
example, the~symbol $[5,1]_{10}$ denotes the~set\footnote{ Note
that these sets do not coincide with conjugacy classes. For
example, the~image of the~set of the~elements of type $[5,
1]_{10}$ under the~natural projection
$2.\A_6\to\A_6$ is a~union of two different
conjugacy classes in $\A_6$.} of elements of order $10$ whose
image in $\A_6$ is a~product of disjoint cycles of
length~$5$~and~$1$).

Recall that there is a natural inner product $\langle\cdot ,
\cdot\rangle$ defined for the~characters $\theta$ and
$\theta^{\prime}$ by
$$\langle\theta , \theta^{\prime}\rangle=\frac{1}{|G|}\sum\limits_{g\in G}
\theta(g)\overline{\theta^{\prime}(g)}.$$ Let $\chi_{0}$ be
the~trivial character of $G$. Then $\langle \chi_{4},
\chi_{0}\rangle=\langle \chi_{6}, \chi_{0}\rangle=0$, so that $G$
has no invariants of degree $4$ and $6$. On the~other hand,
$\langle \chi_{8}, \chi_{0}\rangle=2$, which means that the~group
$G$ has exactly two linearly independent invariants of degree $8$.
\end{proof}

Suppose that $\bar{F}$ is a~stabilizer of a~point
$P\in\mathbb{P}^{3}$.

\begin{lemma}\label{lemma:small-orbits}
Let $\Sigma\subset\mathbb{P}^{3}$ be the~$\bar{G}$-orbit of
the~point $P\in\mathbb{P}^{3}$. Then $|\Sigma|\geqslant 36$.
\end{lemma}

\begin{proof}
It follows from Lemma~\ref{lemma:small-semiinvariants} that
$|\Sigma|\geqslant 8$ and $|\Sigma|$ is even. Suppose that
$|\Sigma|\leqslant 35$. Then
$$
45=\frac{|\bar{G}|}{8}\geqslant\big|\bar{F}\big|\geqslant\frac{|\bar{G}|}{35}>10,%
$$
which implies that $|\bar{F}|\in\{12,18,36\}$ by
Lemma~\ref{lemma:max-subgroup}.

Let us consider the~vector space $W$ as a~representation of the~
group $F$, and let $\chi$~be~its~character. There is
a~homomorphism $\theta\colon F\to\mathbb{C}^{*}$ such that the~
inner product $\langle\theta,\chi\rangle\neq 0$.

Suppose that $|\bar{F}|=36$. Then $\bar{F}\cong
(\mathbb{Z}_3\times\mathbb{Z}_3)\rtimes\mathbb{Z}_4$ by
Lemma~\ref{lemma:max-subgroup}.

The structure of the~group~$F$~and the~values of $\chi$ are given
in the~following table:
\begin{center}\renewcommand\arraystretch{1.3}
\begin{tabular}{|c|c|c|c|c|c|c|c|c|}
\hline & $[4,2]_8$ & $[3,3]_6$ & $[3,3]_3$ &
$[3,1,1,1]_6$ & $[3,1,1,1]_3$ & $[2,2,1,1]_4$ & $z$ & $e$\\
\hline
\# & $36$ & $4$ & $4$ & $4$ & $4$ & $18$ & $1$ & $1$\\
\hline
$\chi$ & $0$ & $-1$ & $1$ & $2$ & $-2$ & $0$ & $-4$ & $4$\\
\hline
\end{tabular}
\end{center}
where we use the~notation that are used in the~proof of
Lemma~\ref{lemma:small-semiinvariants}.

We have $[F,F]\cong 2.(\mathbb{Z}_3\times\mathbb{Z}_3)$, which
gives $\theta(g)=1$ for any $g\in F$ that is not of type $[4,2]_8$
and $[2,2]_4$. Hence $\langle\theta,\chi\rangle=0$, which is
a~contradiction.

Suppose that $|\bar{F}|=18$. Then $\bar{F}\cong
(\mathbb{Z}_3\times\mathbb{Z}_3)\rtimes\mathbb{Z}_2$ by
Lemma~\ref{lemma:max-subgroup}. Arguing as above, we get
$\langle\theta,\chi\rangle=0$.

Suppose that $|\bar{F}|=12$. Then $\bar{F}\cong\A_4$ by
Lemma~\ref{lemma:max-subgroup}.

Up to conjugation, the~group $\bar{G}$ contains two subgroups
isomorphic to $\A_4$.

If $\bar{F}\subset\bar{G}$ is~a~standard embedding, then
the~values of $\chi$ are given in the~following table:
\begin{center}\renewcommand\arraystretch{1.3}
\begin{tabular}{|c|c|c|c|c|c|}
\hline
& $[3,1,1,1]_6$ & $[3,1,1,1]_3$ & $[2,2,1,1]_4$ &$z$ & $e$\\
\hline
\# & $8$ & $8$ & $6$ & $1$ & $1$\\
\hline
$\chi$ & $2$ & $-2$ & $0$ & $-4$ & $4$\\
\hline
\end{tabular}
\end{center}

If $\bar{F}\subset\bar{G}$ is~a~non-standard embedding, then
the~values of $\chi$ are given in the~following table:
\begin{center}\renewcommand\arraystretch{1.3}
\begin{tabular}{|c|c|c|c|c|c|}
\hline
& $[3,3]_6$ & $[3,3]_3$ & $[2,2,1,1]_4$ &$z$ & $e$\\
\hline
\# & $8$ & $8$ & $6$ & $1$ & $1$\\
\hline
$\chi$ & $-1$ & $1$ & $0$ & $-4$ & $4$\\
\hline
\end{tabular}
\end{center}

We have $[F,F]\cong 2.(\mathbb{Z}_2\times\mathbb{Z}_2)$. So
$\theta(g)=1$ for any $g\in F$ of order different from $3$ and
$6$,~and $\theta(g)=\theta(g^2)$ for all $g\in F$ of order $6$.
Now we can check that $\langle\theta,\chi\rangle=0$, which is
a~contradiction.
\end{proof}

\begin{lemma}
\label{lemma:no-cubic} Let $C$ be a~smooth irreducible
$\bar{G}$-invariant curve in $\P^3$ of genus $g\geqslant 13$, and
let~$\mathcal{I}_{C}$ be its ideal sheaf. Then
$h^{0}(\mathcal{O}_{\mathbb{P}^{3}}(i)\otimes\mathcal{I}_C)=0$ for
any $i\in\{1,2,3\}$.
\end{lemma}

\begin{proof}
It follows from Theorem~\ref{theorem:Castelnuovo} and
Lemma~\ref{lemma:sporadic-genera} that $d>4$. Hence
$h^{0}(\mathcal{O}_{\mathbb{P}^{3}}(2)\otimes\mathcal{I}_C)=0$,
since $G$ does not have semi-invariants of degree $2$ by
Lemma~\ref{lemma:small-semiinvariants}.

Suppose that there is a~cubic surface $X\subset\mathbb{P}^{3}$
such that $C\subset X$. Then
$h^{0}(\mathcal{O}_{\mathbb{P}^{3}}(3)\otimes\mathcal{I}_C)\geqslant
2$, since $G$ does not have semi-invariants of degree $3$ by
Lemma~\ref{lemma:small-semiinvariants}. Thus, there is a~cubic
surface $X^{\prime}\subset\mathbb{P}^{3}$ such that $C\subset
X^{\prime}\ne X$. Thus $C\subset X\cap X^{\prime}$, and $X$ and
$X^{\prime}$ are irreducible, because $C$ is contained neither in
a~quadric nor in a~plane.  We see that $d\leqslant 9$. Hence, we
have $g\leqslant 12$ by Theorem~\ref{theorem:Castelnuovo}, which
is a contradiction.
\end{proof}

\section{Projective space}
\label{section:space}

Let $\bar{G}$ be a~subgroup in $\mathrm{Aut}(\mathbb{P}^{3})$ such
that $\bar{G}\cong\A_{6}$. Then there is a~subgroup
$\hat{G}\subset\mathrm{Aut}(\mathbb{P}^{3})$ such that
$\bar{G}\subset\hat{G}\cong\SS_{6}$, which implies that
$\hat{G}\subseteq\mathrm{Aut}^{\bar{G}}(\mathbb{P}^{3})$, because
$\bar{G}$ is a~normal subgroup of the~group $\hat{G}$ (recall that
$\mathrm{Aut}^{\bar{G}}(\mathbb{P}^{3})$ is the normalizer of the
group $\bar{G}$ in the group $\mathrm{Aut}(\P^3)$).

The main purpose of this section is to prove that $\P^3$ is
$\bar{G}$-birationally rigid and to describe the group
$\mathrm{Bir}^{\bar{G}}(\P^3)$ (see Theorem~\ref{theorem:main}).
We will do this in several steps. But first of all, we must study
lines in $\P^3$ whose $\bar{G}$-orbits consists of $6$ lines
(there are $12$ such lines in total, the groups $\hat{G}$ acts
transitively on them, and they form two $\bar{G}$-orbits
consisting of $6$ lines each).

Let $\bar{G}_{1}\subset\bar{G}$ be a subgroup such that
$\bar{G}_{1}\cong\A_{5}$ and the~embedding
$\A_{5}\cong\bar{G}_{1}\subset\bar{G}\cong\A_{6}$ is standard (see
Definition~\ref{definition:standard-subgroups}). Then there are
two disjoint $\bar{G}_{1}$-invariant lines $L_{1}$
and~$L_1^\prime$ in~$\mathbb{P}^{3}$ by
Remark~\ref{remark:standard}. Denote by $L_{1},\ldots, L_{6}$
($L_{1}^{\prime},\ldots, L_{6}^{\prime}$, respectively) the lines
in $\mathbb{P}^{3}$ that are the images of $L_1$ ($L_1^\prime$,
respectively) under the action of $\bar{G}$.
Put $\mathrm{L}=\{L_1, \ldots, L_6\}$
and $\mathrm{L}'=\{L_1', \ldots, L_6'\}$.

\begin{lemma}\label{lemma:L-L-prime}
The~curve $\sum_{i=1}^{6}L_{i}+\sum_{i=1}^{6}L_{i}^{\prime}$ is
a~$\hat{G}$-orbit of the~line $L_{1}$. The $12$ lines
of~\mbox{$\mathrm{L}\cup\mathrm{L}'$}
are pairwise disjoint. For any $4$ lines among the lines
$L_{1},\ldots, L_{6}$ ($L_{1}^{\prime},\ldots, L_{6}^{\prime}$,
respectively), there are~$2$ lines in $\mathbb{P}^{3}$ that
intersect them. There are no lines in $\mathbb{P}^{3}$ that
intersect $5$ lines among $L_{1},\ldots, L_{6}$
($L_{1}^{\prime},\ldots, L_{6}^{\prime}$, respectively).
\end{lemma}

\begin{proof}
The first assertion of the lemma follows from the construction
of the lines $L_i$ and $L_i'$.
To prove the second assertion
note that
the stabilizers of any two of the~$12$ lines
of~\mbox{$\mathrm{L}\cup\mathrm{L}'$}
except for the case of the lines $L_i$ and $L_i'$
corresponding to one and the same stabilizer $\A_5\subset\bar{G}$
(i.\,e.~two standard subgroups
isomorphic to $\A_5$)
generate together the whole group
$\bar{G}\cong\A_6$. In particular, all~$12$ lines
of~\mbox{$\mathrm{L}\cup\mathrm{L}'$}
are pairwise distinct,
since otherwise there would
exist a $\bar{G}$-invariant
line in $\P^3$.
Similarly, if some two of the~$12$ lines
of~\mbox{$\mathrm{L}\cup\mathrm{L}'$} intersected at a point~$P$,
then there would exist a $\bar{G}$-invariant
point in~$\P^3$.

Suppose that there exist more than $2$ lines in $\P^3$
that intersect  some $4$ of the lines
of~$\mathrm{L}$ (say,
$L_1$, $L_2$, $L_3$ and $L_4$). Then the lines $L_1, \ldots, L_4$
are contained in a (unique) smooth quadric $X\subset\P^3$. Note
that there is an element $g\in\bar{G}\cong\A_6$ that preserves
the lines~$L_1$ and~$L_6$, interchanges the lines
$L_2$ and $L_3$ and interchanges the lines
$L_4$ and $L_5$. The quadric $X$ is invariant under $g$ (indeed,
an intersection of two distinct
smooth quadrics in $\P^3$ cannot contain three skew lines). Hence
$X$ contains the line $L_5$. Similarly, $X$ contains
the line $L_6$, and thus~$X$ is $\bar{G}$-invariant. The latter
contradicts Lemma~\ref{lemma:small-semiinvariants}.
Therefore, to prove the third assetion of the lemma
we may suppose that there exist exactly $1$ line $L_{1234}$ in
$\P^3$ that intersects the lines $L_1, \ldots, L_4$.
Note that the group $\bar{G}$ acts transitively on the
$4$-tuples of the lines of~$\mathrm{L}$. Thus for any $4$ lines
$L_{i_1}, \ldots, L_{i_4}$ there exists a unique line
$L_{i_1\ldots i_4}$ that intersects $L_{i_1}, \ldots, L_{i_4}$.
Put $P_{i_2 i_3 i_4}=L_1\cap L_{1 i_2 i_3 i_4}$.
Then the set $\{P_{i_2 i_3 i_4}\}$ is invariant under the group
$\bar{G}_1\cong\A_5$ and consists of at most ${5\choose 3}=10$ points.
This is impossible by Lemma~\ref{lemma:stabilizers}.

Finally, suppose that for some $5$ of the lines
of~$\mathrm{L}$ (say, for $L_2, \ldots, L_6$) there exists
a line $L$ that intersects all $5$ of them. The above argument
implies that there are at most $2$ lines intersecting
$L_2, \ldots, L_6$, so that $L$ is invariant under
the group $\bar{G}_1\cong\A_5$. Since the lines
of~\mbox{$\mathrm{L}\cup\mathrm{L}'$}
are pairwise disjoint, the line $L$ coincides with neither $L_1$
nor $L_1'$. Hence there are at least $3$ lines in $\P^3$ that are
invariant under $\bar{G}_1$, which contradicts
Remark~\ref{remark:standard}.

The same arguments apply if one replaces the lines
of~$\mathrm{L}$ by the lines of~$\mathrm{L}'$.
\end{proof}

Recall that several lines in $\mathbb{P}^{3}$ are said to
lie in a linear complex if the~corresponding points of
the~Grassmaniann $\mathrm{Gr}(2, 4)\subset\mathbb{P}^5$ lie in a
hyperplane section.

\begin{lemma}\label{lemma:linear-complex}
Neither $L_{1},L_{2},L_{3},L_{4},L_{5},L_{6}$ nor
$L_{1}^{\prime},L_{2}^{\prime},L_{3}^{\prime},L_{4}^{\prime},
L_{5}^{\prime},L_{6}^{\prime}$ lie in a~linear complex.
\end{lemma}

\begin{proof}
It follows from Remark~\ref{remark:irreducible} that the~space
$\Lambda^2(V)\cong\mathbb{C}^6$ is the~permutation representation
of the~group $\bar{G}\cong\A_{6}$. The natural action of the~group
$\bar{G}$ on $\mathrm{Gr}(2,
4)\subset\mathbb{P}^5\cong\mathbb{P}(\Lambda^2(V))$ arises from
this representation. Let us identify $\bar{G}$ with a subgroup in
$\mathrm{Aut}(\mathbb{P}^{5})$. Then there is a~unique
$\bar{G}$-invariant hyperplane~$H\subset\mathbb{P}^{5}$. We may
assume that $\mathrm{Gr}(2, 4)\cap H$ is given by
$$
\sum_{i=0}^{5}x_{i}=\sum_{i=0}^{5}x_{i}^{2}=0\subset\mathbb{P}^{5}\cong\mathrm{Proj}\Big(\mathbb{C}\big[x_{0},x_{1},x_{2},x_{3},x_{4},x_{5}\big]\Big),
$$
and $H$ is given by $\sum_{i=0}^{5}x_{i}=0$. Let
$P\in\mathrm{Gr}(2, 4)$ be a point that corresponds to the~line
$L_1\subset\mathbb{P}^{3}$, and let $\bar{G}_{1}$ be the
stabilizer subgroup in $\bar{G}$ of the~point $P$. Then
$\bar{G}_{1}\cong\A_5$, and the~embedding
$\bar{G}_{1}\subset\bar{G}$ is standard. If the~$\bar{G}$-orbit of
the~point $P$ is contained in a hyperplane, then this hyperplane
must be $H$, which implies that the~$\bar{G}$-orbit of the~point
$P$ must contain a~point in~$\mathbb{P}^{5}$ that is given~by
$x_{0}=\ldots=x_{4}=-x_{5}/5$, which~is~impossible, because this
point does not belong to the~intersection  $\mathrm{Gr}(2, 4)\cap
H$.
\end{proof}

Now we are ready to formulate the main technical result of this
section.

\begin{theorem}
\label{theorem:A6-NFI} Let $\mathcal{M}$ be a (non-empty)
$\bar{G}$-invariant
linear system on $\P^3$ that does not have fixed components, and
let $\lambda$ be a positive rational number such that
$\lambda\mathcal{M}\sim_{\mathbb{Q}}-K_{\mathbb{P}^{3}}$. Suppose
that $(\P^3, \lambda\mathcal{M})$ is canonical at a general point of
every line $L_{1},\ldots,L_{6},L_{1}^\prime,\ldots,L_6^\prime$.
Then $(\mathbb{P}^{3},\lambda\mathcal{M})$ is canonical.
\end{theorem}

Before proving Theorem~\ref{theorem:A6-NFI}, let us show that it
implies that $\P^3$ is $\bar{G}$-birationally rigid and
$\mathrm{Bir}^{\bar{G}}(\P^3)$ is isomorphic to a~free product of
$\SS_{6}$ and $\SS_{6}$ with an~amalgamated subgroup $\A_{6}$. To
do this, will use the lines
$L_{1},\ldots,L_{6},L_{1}^\prime,\ldots,L_6^\prime$ to prove that
$\mathrm{Aut}^{\bar{G}}(\mathbb{P}^{3})=\hat{G}$ and to construct
two $\bar{G}$-equivariant birational non-biregular involutions of
$\P^3$ that was described in \cite{To33}. The proof of
$\bar{G}$-birational rigidity of $\P^3$ and description of the
group $\mathrm{Bir}^{\bar{G}}(\P^3)$ crucially depend on these
birational involutions.

\begin{lemma}
\label{lemma:auxiliary-aut} The equality
$\mathrm{Aut}^{\bar{G}}(\mathbb{P}^{3})=\hat{G}$ holds.
\end{lemma}

\begin{proof}
Put $\tilde{G}=\mathrm{Aut}^{\bar{G}}(\mathbb{P}^{3})$. Then
the~curve $\sum_{i=1}^{6}L_{i}+\sum_{i=1}^{6}L_{i}^{\prime}$
is $\tilde{G}$-invariant, and
the~group $\bar{G}$ is a~normal subgroup of the~group $\tilde{G}$.
There is a subgroup $\breve{G}\subset\tilde{G}$ of index $2$ such
that  $\sum_{i=1}^{6}L_{i}$ and $\sum_{i=1}^{6}L_{i}^{\prime}$ are
$\breve{G}$-invariant.

Let $Y$ be the~Hierholzer surface of the~lines
$L_{1},L_{2},L_{3},L_{4},L_{5},L_{6}$ (see
\cite[Section~2.1]{Va01}).~Then the~surface $Y$ is
$\breve{G}$-invariant and there exists a~monomorphism
$\xi\colon\breve{G}\to\mathrm{Aut}(Y)$. Moreover, the~surface $Y$
is birational to a~surface of general type (see
\cite[Theorem~2.1]{Va01}), which implies that $\breve{G}$ is
a~finite group (see e.\,g.\cite[Theorem~1.1]{HMX10}). Thus
$\tilde{G}$ is a~finite group. Hence, we have
$\mathrm{lct}(\mathbb{P}^{3},\tilde{G})\geqslant\mathrm{lct}(\mathbb{P}^{3},\bar{G})\geqslant
5/4$ by \cite[Theorem~4.13]{ChSh09}.

Let $\phi\colon\SL_4(\mathbb{C})\to\mathrm{Aut}(\mathbb{P}^{3})$
be the~natural projection. Then there is a~finite subgroup
$G\subset\SL_4(\mathbb{C})$ such that $\tilde{G}=\phi(G)$. Note
that $G$ may not be uniquely defined. However, by
\cite[Theorem~4.13]{ChSh09},  we may assume~that $G$ is one of the
eight groups described in~\cite[Lemma~4.12]{ChSh09}. Suppose that
$\hat{G}\subsetneq\tilde{G}$. Then $G\not\cong 2.\SS_{6}$.

Recall that $\bar{G}$ is a (non-trivial) normal subgroup in
$\tilde{G}$. Hence $G$ cannot be isomorphic to any of the groups
$2.\A_{7}$ and~$\PSp_4(\F_3)$, because the images of these groups
in $\PGL_4(\C)$ are simple groups. Moreover, $G\not\cong 2.\A_6$,
since there is a subgroup $\hat{G}\subset\tilde{G}$ isomorphic to
$\SS_6$.

For each of the four remaining groups
from~\cite[Lemma~4.12]{ChSh09} one has (see~\cite{Nie92}) an~exact
sequence of groups
$$
\xymatrix{&1\ar@{->}[rr]&&\mathbb{Z}_{2}^{4}\ar@{->}[rr]^{\alpha}&&\tilde{G}\ar@{->}[rr]^{\beta}&&\Gamma\ar@{->}[rr]&&1},%
$$
where $\Gamma$ is a~subgroup of the~group $\SS_{6}$. Thus
$\Gamma\cong\SS_{6}$ because
$\tilde{G}\supset\hat{G}\cong\SS_{6}$. Let $e$ be the~identity
element in $\tilde{G}$. Then $\bar{G}\cap\mathrm{im}(\alpha)=e$,
because $\bar{G}$ is a~simple group, and $\bar{G}$ is a~normal
subgroup of the~group $\tilde{G}$. Hence the subgroup of
$\tilde{G}$ generated by $\bar{G}\cong\A_6$ and
$\mathrm{im}(\alpha)\cong\Z_2^4$ is a direct product of the latter
groups. This leads to a~contradiction because the action of
$\Gamma$ on $\Z_2^4$ is faithful (see~\cite{Nie92}).
\end{proof}

Let $H$ be a general hyperplane in $\P^3$.

\begin{lemma}
\label{lemma:Todd-involution}  There is a $\bar{G}$-equivariant
birational non-biregular involution
$\iota\in\mathrm{Bir}(\mathbb{P}^{3})$ such that $\iota(H)$ is a
surface of degree $19$ that has singularity of multiplicity $5$ in
a general point of every line $L_{1},\ldots,L_{6}$, and the~group
generated by $\bar{G}$ and $\iota$ is isomorphic to $\SS_{6}$.
\end{lemma}

\begin{proof}
Let $\mathcal{H}$ be a~linear subsystem of the~linear system
$|\mathcal{O}_{\mathbb{P}^{3}}(4)|$ consisting of surfaces that
pass through the~lines $L_{1},L_{2},L_{3},L_{4},L_{5},L_{6}$. Then
it follows from \cite{To33} and \cite[Theorem~2.4]{Va01} that
$\mathcal{H}$ does not have fixed components, and $\mathcal{H}$
gives a~rational map $\psi\colon\mathbb{P}^{3}\dasharrow V$, where
$V$ is a~quartic threefold in $\mathbb{P}^{4}$. Let $\alpha\colon
U\to\mathbb{P}^{3}$ be a~blow up along
$L_{1},L_{2},L_{3},L_{4},L_{5},L_{6}$. Then there is a~commutative
diagram
$$
\xymatrix{
&U\ar@{->}[dl]_{\alpha}\ar@{->}[dr]^{\beta}&\\
\mathbb{P}^{3}\ar@{-->}[rr]_{\phi}&&V,}
$$
where $\beta$ is a~birational morphism that contracts finitely
many curves. It follows from \cite[Theorem~2.4]{Va01} that
$V\subset\mathbb{P}^{4}$ can be given by the~equation\footnote{
Note that the~quartic threefold $V\subset\mathbb{P}^{4}$ is
determinantal (see \cite{To33}, \cite[Example~6.4.2]{Pet98}).}
\begin{eqnarray*}
&5\Big(yw-zt+xy+xw-xz-xt-x^2\Big)^2-20xyw\Big(w-x+y-z-t\Big)&\\
&\Vert&\\
&\Big(3x^2+2y^2+2z^2+2t^2+2w^2-3xy+3xz+3xt-3xw-2yz-2yt+yw+zt-2zw-2tw\Big)^2&
\end{eqnarray*}
in appropriate homogeneous coordinates $[x:y:z:t:w]$ in
$\mathbb{P}^{4}$. By \cite{To33} we know that the~singular locus
of the~threefold $V$ consists of $36$ nodes, and the~morphism
$\beta$ contracts the~proper transforms of $30$ lines in
$\mathbb{P}^3$ each of whom intersects $4$ lines among
$L_{1},L_{2},L_{3},L_{4},L_{5},L_{6}$, and proper transforms of
$6$ twisted cubics in $\mathbb{P}^3$ each of whom has the~lines
$L_{1},L_{2},L_{3},L_{4},L_{5},L_{6}$ as chords.

The map $\phi$ is $\bar{G}$-equivariant. We can identify $\bar{G}$
with a~subgroup in $\mathrm{Aut}(V)$ and with a subgroup in
$\mathrm{Aut}(\mathbb{P}^{4})$. By
Lemma~\ref{lemma:small-semiinvariants}, there are no
$\bar{G}$-fixed points in $\mathbb{P}^{4}$. So, the~action of
the~group $\bar{G}$~on~$\mathbb{P}^{4}$ arises from its
irreducible five-dimensional representation (cf.
Remark~\ref{remark:irreducible}). Hence, there~is a~subgroup
$\bar{\Gamma}\subset\mathrm{Aut}(V)$ such that
$\bar{G}\subset\bar{\Gamma}\cong\SS_{6}$. Let $\theta$ be
an~involution in $\bar{\Gamma}$ such that $\theta\not\in\bar{G}$.
Put
$\iota=\phi^{-1}\circ\theta\circ\phi\in\mathrm{Bir}^{\bar{G}}(\mathbb{P}^{3})$.
Then it follows from the~construction of $\iota$ that $\langle
\iota,\bar{G}\rangle\cong\SS_{6}$.

Let us show that $\iota$ is not biregular. Suppose that this is
not true. Let us identify the~subgroup $\langle
\iota,\bar{G}\rangle$ with $\bar{\Gamma}\cong\SS_{6}$. Then
$\sum_{i=1}^{6}L_{i}$ must be $\bar{\Gamma}$-invariant. Let
$\bar{G}_{1}\subset\bar{G}$ and
$\bar{\Gamma}_{1}\subset\bar{\Gamma}$ be stabilizers of the~line
$L_{1}$. Then
$\A_{5}\cong\bar{G}_{1}\subset\bar{\Gamma}_{1}\cong\SS_{5}$, and
there are natural homomorphisms
$\zeta\colon\bar{G}_{1}\to\mathrm{Aut}(L_{1})$ and
$\eta\colon\bar{\Gamma}_{1}\to\mathrm{Aut}(L_{1})$. Thus, we have
$\A_{5}\cong\mathrm{im}(\zeta)\subseteq\mathrm{im}(\eta)$, which
gives $\mathrm{im}(\eta)\cong\SS_{5}$. But
$\mathrm{Aut}(L_{1})\cong\PGL_2(\mathbb{C})$ contains no subgroups
isomorphic to $\SS_{5}$. The obtained contradiction shows that
$\iota$ is not biregular.

Put $\tau=\beta^{-1}\circ\theta\circ\beta\in\mathrm{Bir}(U)$. Then
$\tau$ is a~composition of flops, and the commutative diagram
$$
\xymatrix{
&&U\ar@{->}[dll]_{\alpha}\ar@{->}[drr]^{\beta}\ar@{-->}[rrrr]^{\tau}&&&&U\ar@{->}[dll]_{\theta\circ\beta}\ar@{->}[drr]^{\alpha}&\\
\mathbb{P}^{3}\ar@{-->}[rrrr]_{\phi}&&&&V&&&&\mathbb{P}^{3}\ar@{-->}[llll]^{\phi\circ\iota},}
$$
is a~$\bar{G}$-equivariant Sarkisov link of type $\mathrm{II}$
(see \cite[Definition~3.4]{Co00}). The involution $\tau$ is
biregular outside of the~curves contracted by $\beta$. Then $\tau$
naturally acts on the~group $\mathrm{Pic}(U)$. Moreover, this
action is nontrivial, because $\iota$ is not biregular.

Put $\hat{H}=\alpha^{*}(H)$. Let $E_{k}$ be
the~$\alpha$-exceptional divisor such that $\alpha(E_{k})=L_{k}$.
Then
$$
\tau^{*}\big(\hat{H}\big)\cdot \tau^{*}\big(\hat{H}\big)\cdot\Big(4\hat{H}-\sum_{i=1}^{6}E_{i}\Big)=4,%
$$
because $\tau^{*}(K_{U})\sim K_{U}$ and $\beta$ contracts finitely
many curves. Similarly, we see that
$$
\tau^{*}\big(\hat{H}\big)\cdot
\tau^{*}\big(E_{k}\big)\cdot\Big(4\hat{H}-\sum_{i=1}^{6}E_{i}\Big)=1,\
\tau^{*}\big(E_{i}\big)\cdot
\tau^{*}\big(E_{k}\big)\cdot\Big(4\hat{H}-\sum_{i=1}^{6}E_{i}\Big)=\left\{\aligned
&-2\ \text{if}\ i=k,\\
&0\ \text{if}\ i\ne k,\\
\endaligned
\right.
$$
which immediately implies that either
$\tau^{*}(\hat{H})\sim\hat{H}$ and $\tau^{*}(E_{k})=E_{k}$, or
\begin{equation}\label{action-on-Pic}
\left\{\aligned &\tau^{*}(\hat{H})\sim 19\hat{H}-5\sum_{i=1}^{6}E_{i},\\
&\tau^{*}\big(E_{k}\big)\sim 12\hat{H}-4E_{k}-\sum_{i\ne k}3E_{i},
\endaligned
\right.
\end{equation}

Since $\iota$ is not biregular, the involution $\tau$ must act
non-trivially on $\mathrm{Pic}(U)$. Thus, we see that
$(\ref{action-on-Pic})$ holds, which implies, in particular, that
$\iota(H)$ is a surface of degree $19$ that has singularity of
multiplicity $5$ in a general point of every line
$L_{1},\ldots,L_{6}$, because $\iota(H)=\alpha\circ\tau(\hat{H})$.
\end{proof}

Put $\iota^{\prime}=\nu\circ\iota\circ\nu^{-1}$ for any
$\nu\in\hat{G}\cong \SS_{6}$ such that $\nu\not\in\bar{G}$. Then
$\iota^{\prime}$ is a $\bar{G}$-equivariant birational
non-biregular involution of $\mathbb{P}^{3}$ such that
$\iota^{\prime}(T)$ is a surface of degree $19$ that has
singularity of multiplicity $5$ in a general point of every line
$L^{\prime}_{1},\ldots,L^{\prime}_{6}$, and the~group generated by
$\bar{G}$ and $\iota^{\prime}$ is also isomorphic to $\SS_{6}$.
Note that the~choice of $\iota$ and $\iota^{\prime}$ is not
unique. Put $\Gamma=\langle\iota,\iota^{\prime},\hat{G}\rangle$.
Let us use birational involutions $\iota$ and $\iota^{\prime}$
together with Theorem~\ref{theorem:A6-NFI} to prove

\begin{lemma}
\label{lemma:P3-untwisting} Let $\mathcal{M}$ be a (non-empty)
$\bar{G}$-invariant linear system on $\P^3$ that does not have
fixed components. Then there are $\rho\in\Gamma$ and
$\mu\in\mathbb{Q}_{>0}$ such that
$\mu\rho(\mathcal{M})\sim_{\mathbb{Q}}-K_{\mathbb{P}^{3}}$, and
$(\P^3, \mu\rho(\mathcal{M}))$ is canonical.
\end{lemma}

\begin{proof}
Take $\lambda\in\mathbb{Q}_{>0}$ such that
$\lambda\mathcal{M}\sim_{\mathbb{Q}}-K_{\mathbb{P}^{3}}$. If
$(\P^3, \lambda\mathcal{M})$ is canonical, then we are done. Thus,
we may assume that $(\P^3, \lambda\mathcal{M})$ is not canonical.
By Theorem~\ref{theorem:A6-NFI}, the log pair $(\P^3,
\lambda\mathcal{M})$ is not canonical at a general point of one line
among $L_{1},\ldots,L_{6},L_{1}^\prime,\ldots,L_6^\prime$.

Without loss of generality, we may assume that $(\P^3,
\lambda\mathcal{M})$ is not canonical at a general point of the
line $L_1$. Then $(\P^3, \lambda\mathcal{M})$  is not canonical at
a general point of the lines $L_2, \ldots, L_6$, since $\mathcal{M}$
is $\bar{G}$-invariant. Let $M$ be a~general surface in
$\mathcal{M}$. Then $\mathrm{mult}_{L_{i}}(M)>1/\lambda$ for every
$i\in\{1,\ldots,6\}$ (see \cite[Exercise~6.18]{CKS04}). Note that
$M$ is a surface of degree $d=4/\lambda$. Let $\tilde{d}$ be the
degree of the surface $\iota(M)$. Put
$\tilde{\lambda}=4/\tilde{d}$. Then
$\tilde{\lambda}\iota(\mathcal{M})\sim_{\mathbb{Q}}-K_{\mathbb{P}^{3}}$.
Moreover, it follows from Lemma~\ref{lemma:Todd-involution} that
$$
\tilde{d}=19d-12\sum_{i=1}^{6}\mathrm{mult}_{L_{i}}(M)<19d-18d=d.
$$

If $(\P^3, \tilde{\lambda}\iota(\mathcal{M}))$ is canonical, then
we are done. If it is not canonical, then it follows from
Theorem~\ref{theorem:A6-NFI} that $(\P^3,
\tilde{\lambda}\iota(\mathcal{M}))$ is not canonical at a general
point of one line among
$L_{1},\ldots,L_{6},L_{1}^\prime,\ldots,L_6^\prime$. Thus, we can
iterate the above process. Since $\tilde{d}<d$, our iterations
must terminate in at most $d-1$ steps, which completes the proof.
\end{proof}

Now we ready to prove

\begin{theorem}
\label{theorem:A6}  The variety $\mathbb{P}^{3}$ is
$\bar{G}$-birationally rigid and
$\mathrm{Bir}^{\bar{G}}(\mathbb{P}^{3})=\Gamma$.
\end{theorem}

\begin{proof}
Suppose that there exists a rational normal threefold $V$ with at
most terminal singularities that admits a faithful action of the
group $\bar{G}$ such that there are a $\bar{G}$-equivariant Mori
fibration $\pi\colon V\to S$ and a $\bar{G}$-equivariant
birational map $\xi\colon V\dasharrow \P^3$. Let $D$ be a
sufficiently general very ample divisor on $V$. Put
$\mathcal{M}=\xi(|D|)$. Then $\mathcal{M}$ does not have fixed
components and is $\bar{G}$-invariant. Thus, it follows from
Lemma~\ref{lemma:P3-untwisting} that there are $\rho\in\Gamma$ and
$\mu\in\mathbb{Q}_{>0}$ such that
$\mu\rho(\mathcal{M})\sim_{\mathbb{Q}}-K_{\mathbb{P}^{3}}$, and
$(\P^3, \mu\rho(\mathcal{M}))$ is canonical. Then $\rho\circ\xi$
is biregular by \cite[Theorem~4.2]{Co95}. In particular, we see
that $V\cong\P^3$ and $S$ is a point. Since all subgroups in
$\mathrm{Aut}(\mathbb{P}^{3})$ that are isomorphic to $\bar{G}$
are conjugate to each other, we see that $\mathbb{P}^{3}$ is
$\bar{G}$-birationally rigid and
$\mathrm{Bir}^{\bar{G}}(\mathbb{P}^{3})=\Gamma$.
\end{proof}

Before proving Theorem~\ref{theorem:A6-NFI}, let us use it one
more time to prove

\begin{theorem}
\label{theorem:auxiliary-amalgamated} The group $\Gamma$ is
the~free product of $\hat{G}$ and $\langle\iota,\bar{G}\rangle$
with~amalgamated~subgroup~$\bar{G}$.
\end{theorem}

\begin{proof}
Let $\xi$ be a~birational automorphism in
$\mathrm{Bir}^{\bar{G}}(\mathbb{P}^{3})$ such that
$\xi\not\in\hat{G}$. Put $\mathcal{M}=\xi(|T|)$ and take
$\lambda\in\mathbb{Q}$  such that
$\lambda\mathcal{M}\sim_{\mathbb{Q}}-K_{\mathbb{P}^{3}}$. Then
$(\mathbb{P}^{3},\lambda\mathcal{M})$ is not canonical by
\cite[Theorem~4.2]{Co95}. By Theorem~\ref{theorem:A6-NFI}, the log
pair $(\P^3, \lambda\mathcal{M})$ is not canonical at a general
point of one line among
$L_{1},\ldots,L_{6},L_{1}^\prime,\ldots,L_6^\prime$. Arguing as in
the proof of Lemma~\ref{lemma:P3-untwisting}, we see that there
exists a combination
$$
\zeta=\underbrace{\ldots\circ\iota\circ\iota^{\prime}\circ\iota\circ\iota^{\prime}\circ\iota\circ\ldots}_{m\ \mathrm{times}},%
$$
such that $(\mathbb{P}^{3},\mu\zeta(\mathcal{M}))$ is canonical,
where $\mu$ is a positive rational number such that one has
\mbox{$\mu\zeta(\mathcal{M})\sim_{\mathbb{Q}}-K_{\mathbb{P}^{3}}$}. Then
$\zeta\circ\xi$ is biregular by \cite[Theorem~4.2]{Co95}, which
implies that $\zeta\circ\xi\in\hat{G}$ by
Lemma~\ref{lemma:auxiliary-aut}.

Let us show that $\zeta$ is uniquely determined by $\xi$ and the
algorithm hidden in the proof Lemma~\ref{lemma:P3-untwisting}.

Let $M$ be a~general surface in $\mathcal{M}$, and let $d$ be its
degree. Then $d=4/\lambda$. Arguing as in the proof of
Lemma~\ref{lemma:P3-untwisting}, we see that if $(\P^3,
\lambda\mathcal{M})$ is not canonical at a general point of one line
among $L_{1},\ldots,L_{6}$, then the degree of $\iota(M)$ is
smaller than $d$. Similarly, we see that if $(\P^3,
\lambda\mathcal{M})$ is not canonical at a general point of one line
among $L^\prime_{1},\ldots,L_{6}^\prime$, then the degree of
$\iota^\prime(M)$ is smaller than $d$. We can use this to
construct $\xi$ step by step. Thus, to show that $\zeta$ is
uniquely determined by $\xi$, we must show that $(\P^3,
\lambda\mathcal{M})$ can not be non-canonical at a general point of
one line among $L_{1},\ldots,L_{6}$ and the same time at a general
point of one line among $L^\prime_{1},\ldots,L_{6}^\prime$.

Suppose that $(\P^3, \lambda\mathcal{M})$ is not canonical at
a general point of one line among $L_{1},\ldots,L_{6}$ and the same
time at a general point of one line among
$L^\prime_{1},\ldots,L_{6}^\prime$. Then $(\P^3,
\lambda\mathcal{M})$ is not canonical at a general point of every
line among $L_{1},\ldots,L_{6},L^\prime_{1},\ldots,L_{6}^\prime$,
since $\mathcal{M}$ is $\bar{G}$-invariant. Then
$\mathrm{mult}_{L_{i}}(M)>1/\lambda$ and
$\mathrm{mult}_{L_{i}^\prime}(M)>1/\lambda$ for every
$i\in\{1,2,\ldots,6\}$. Let $\Pi$ be a~general plane in
$\mathbb{P}^{3}$ that contains the~line $L_{1}$. Put
$\mathcal{M}\vert_{\Pi}=\mathrm{mult}_{L_{1}}(M)L_{1}+\mathcal{B}$,
where $\mathcal{B}$ is a~linear system on $\Pi$ that does not have
fixed components. Let $B_{1}$ and $B_{2}$ be general curves in
$\mathcal{B}$, let $O_{i}$ be the point $\Pi\cap L_{i}$ for every
$i\in\{2,\ldots,6\}$, and let $O_{i}^{\prime}$ be the point
$\Pi\cap L_{i}^{\prime}$ for every $i\in\{1,\ldots,6\}$. Then
$\mathrm{mult}_{O_{i}}(B_{1})=\mathrm{mult}_{O_{i}}(B_{2})=\mathrm{mult}_{L_{1}}(M)$
for every $i\in\{2,\ldots,6\}$, and
$\mathrm{mult}_{O_{i}^{\prime}}(B_{1})=\mathrm{mult}_{O_{i}^{\prime}}(B_{2})=\mathrm{mult}_{L^{\prime}_{1}}(M)$
for every $i\in\{1,\ldots,6\}$. Then
$$
B_{1}\cdot
B_{2}\geqslant\sum_{i=2}^{6}\mathrm{mult}_{O_{i}}\big(B_{1}\big)\mathrm{mult}_{O_{i}}\big(B_{2}\big)+\sum_{i=1}^{6}\mathrm{mult}_{O^{\prime}_{i}}\big(B_{1}\big)\mathrm{mult}_{O^{\prime}_{i}}\big(B_{2}\big)>\frac{11}{\lambda^{2}},
$$
because $\mathrm{mult}_{L_{1}}(M)>1/\lambda$ and
$\mathrm{mult}_{L_{1}^{\prime}}(M)>1/\lambda$.
On the~other hand,
we have
$$
B_{1}\cdot B_{2}=\Big(\frac{4}{\lambda}\Pi\Big\vert_{\Pi}-\mathrm{mult}_{L_{1}}\big(M\big)L_{1}\Big)^{2}=\Big(\frac{4}{\lambda}-\mathrm{mult}_{L_{1}}\big(M\big)\Big)^{2}<\frac{9}{\lambda^{2}},%
$$
which is a~contradiction. Thus we see that $(\P^3,
\lambda\mathcal{M})$ is not canonical either at a general point of
every line among $L_{1},\ldots,L_{6}$, or at a general point of
every line among $L^\prime_{1},\ldots,L_{6}^\prime$. But it can
not be non-canonical at a general point of all these $12$ lines.

Recall that we defined $\iota^{\prime}$ ad
$\nu\circ\iota\circ\nu^{-1}$ for some $\nu\in\hat{G}$ such that
$\nu\not\in\bar{G}$. Without loss of generality we may assume that
$\nu$ is an involution. Thus, we see that every
$\alpha\in\mathrm{Bir}^{\bar{G}}(\mathbb{P}^{3})$ can be uniquely
written as
$$
\alpha=\underbrace{\ldots\circ\iota\circ\nu\circ\iota\circ\nu\circ\iota\circ\nu\circ\iota\circ\nu\circ\iota}_{r\ \mathrm{times}}\circ\beta%
$$
for some $\beta\in\hat{G}$ such that $r=0$ if and only if
$\alpha\in\hat{G}$. The latter implies that $\Gamma$ is the~free
product of the~ groups
$\hat{G}=\langle\nu,\bar{G}\rangle\cong\SS_{6}$ and
$\langle\iota,\bar{G}\rangle\cong\SS_{6}$
with~amalgamated~subgroup~$\bar{G}\cong\A_{6}$.
\end{proof}

\begin{remark}
\label{remark:non-amalgamated} Note that $\hat{G}\not\subset
\langle\iota,\iota^{\prime},\bar{G}\rangle$ (see the~proof of
Theorem~\ref{theorem:auxiliary-amalgamated}).
\end{remark}

In the~rest of this section, we prove
Theorem~\ref{theorem:A6-NFI}. Let $\mathcal{M}$ be a (non-empty)
$\bar{G}$-invariant linear system on $\P^3$ that does not have
fixed components, and let $\lambda$ be a positive rational number
such that
$\lambda\mathcal{M}\sim_{\mathbb{Q}}-K_{\mathbb{P}^{3}}$. Suppose
that $(\P^3, \lambda\mathcal{M})$ is canonical at a general point of
every line $L_{1},\ldots,L_{6},L_{1}^\prime,\ldots,L_6^\prime$. We
must prove that $(\mathbb{P}^{3},\lambda\mathcal{M})$ is
canonical.

Suppose that the~log pair $(\mathbb{P}^{3},\lambda\mathcal{M})$ is
not canonical. Let us seek for a contradiction. By
Corollary~\ref{corollary:mult-by-2} there is $\mu\in\mathbb{Q}$
such that $\mu<2\lambda$ and $(\mathbb{P}^{3},\mu\mathcal{M})$ is
strictly log canonical. Let $S\subset X$ be a~minimal center in
$\mathbb{LCS}(\mathbb{P}^{3},\mu\mathcal{M})$. Then $S$ is not
a~surface, since $\mathcal{M}$ has no fixed components.

Recall that we denote by $H$ a general hyperplane in
$\mathbb{P}^{3}$. By Lemma~\ref{lemma:Kawamata-Shokurov-trick},
there~is $\delta\in\mathbb{Q}$ and there~is a~$\bar{G}$-invariant
effective $\mathbb{Q}$-divisor $D\sim_{\mathbb{Q}}\delta H$ on
the~threefold $\mathbb{P}^{3}$ such that $0<\delta<8$, the~log
pair $(\mathbb{P}^{3},D)$ is log canonical and the set
$\mathbb{LCS}(\mathbb{P}^{3}, D)$ consists of all irreducible
components of the $\bar{G}$-orbit of $S$.

\begin{remark}
\label{remark:mult-2-inequality} It follows from the~proof of
Lemma~\ref{lemma:Kawamata-Shokurov-trick} that we may assume~that
$\mathrm{mult}_{L_{i}}(D)<2$ and
$\mathrm{mult}_{L_{i}^{\prime}}(D)<2$ for every
$i\in\{1,2,\ldots,6\}$, since $(\P^3, \lambda\mathcal{M})$ is
canonical at a general point of every line
$L_{1},\ldots,L_{6},L_{1}^\prime,\ldots,L_6^\prime$.
\end{remark}

Let $\mathcal{I}$ be the~multiplier ideal sheaf of the~log pair
$(\mathbb{P}^{3}, D)$, and let $\mathcal{L}$ be the~log canonical
singularities subscheme of the~log pair $(\mathbb{P}^{3}, D)$.
Then it follows from Theorem~\ref{theorem:Shokurov-vanishing} that
\begin{equation}
\label{equation:appendix-equality}
h^{0}\Big(\mathcal{O}_{\mathcal{L}}\otimes\mathcal{O}_{\mathbb{P}^{3}}\big(4H\big)\Big)=h^{0}\Big(\mathcal{O}_{\mathbb{P}^{3}}\big(4H\big)\Big)-h^{0}\Big(\mathcal{O}_{\mathbb{P}^{3}}\big(4H\big)\otimes\mathcal{I}\Big)=35-h^{0}\Big(\mathcal{O}_{\mathbb{P}^{3}}\big(4H\big)\otimes\mathcal{I}\Big).%
\end{equation}

\begin{lemma}
\label{lemma:auxiliary-4} The center $S$ is a~curve.
\end{lemma}

\begin{proof}
If $S$ is a~point, then $|\mathrm{LCS}(\mathbb{P}^{3},
D)|\leqslant 35$ by (\ref{equation:appendix-equality}), which
contradicts Lemma~\ref{lemma:small-orbits}.
\end{proof}

By Theorem~\ref{theorem:Kawamata}, the~curve $S$ is a~smooth curve
in $\mathbb{P}^{3}$ of degree $d$ and genus $g$ such that
\begin{equation}\label{eq:g-2d}
g\leqslant 2d.
\end{equation}
Put $q=h^{0}(\mathcal{O}_{\mathbb{P}^{3}}(4H)\otimes\mathcal{I})$.
Let $Z$ be the~$\bar{G}$-orbit of the~curve $S$, let $r$ be
the~number of irreducible components of $Z$. Then $Z=\mathcal{L}$,
since $(\mathbb{P}^{3},D)$ is log canonical. Moreover, the curve
$Z$ is smooth by Lemma~\ref{lemma:centers}, and $Z$ is not
contained in a~hyperplane in $\P^3$ by
Remark~\ref{remark:irreducible}.

\begin{lemma}
\label{lemma:35-q-equality} The equality $r(4d-g+1)=35-q$ holds.
\end{lemma}
\begin{proof}
The equality follows from (\ref{equation:appendix-equality}) and
the~Riemann--Roch theorem, because $2d\geqslant g$.
\end{proof}

\begin{corollary}
\label{corollary:g-bound} The inequality $g\leqslant 34$ holds.
\end{corollary}

\begin{lemma}
\label{lemma:small-q} Suppose that $1\leqslant q\leqslant 7$. Then
$q\in\{5,6\}$.
\end{lemma}

\begin{proof}
Put $W_4=H^{0}(\mathcal{O}_{\mathbb{P}^{3}}(4H)\otimes\mathcal{I})$.
Since the~center of $G$ acts trivially on polynomials of even
degree, the space $W_4$ has a natural structure of
a~$\bar{G}$-representation. Suppose that $q\not\in\{5,6\}$. Then
$W_4$ has a~trivial~sub\-rep\-re\-sen\-ta\-tion of the~group
$\bar{G}$ by dimension reasons (see~\cite{Atlas}), which is impossible by
Lemma~\ref{lemma:small-semiinvariants}.
\end{proof}

\begin{lemma}
\label{lemma:A6-reducible-curves} Suppose that $r\ne 1$. Then
$r=6$ and $d=1$.
\end{lemma}

\begin{proof}
Since $r\leqslant 35$ by Lemma~\ref{lemma:35-q-equality}, one has
$r\in\{6,10,15,20,30\}$ by Lemma~\ref{lemma:max-subgroup}. In
particular, the~inequality $q\geqslant 1$ holds by
Lemma~\ref{lemma:35-q-equality}.

Suppose that $q\geqslant 6$. Then $r(4d-g+1)<30$ by
Lemma~\ref{lemma:35-q-equality}. We see that $g+1\leqslant
4d-g+1<5$, which implies that $g\leqslant 3$. Then $d=1$ and
$g=0$, so that $4d-g+1=5$, which is a~contradiction.

By Lemma~\ref{lemma:small-q} we may assume that $q=5$. Then
$r\in\{6, 10, 15, 30\}$ by Lemma~\ref{lemma:35-q-equality}.

If $r=30$, then $g+1\leqslant 4d-g+1=1$ by
Lemma~\ref{lemma:35-q-equality}, which is a~contradiction.

If $r=15$, then $g+1\leqslant 4d-g+1=2$ by
Lemma~\ref{lemma:35-q-equality},  which leads to a~contradiction.

If $r=10$, then $g+1\leqslant 4d-g+1=3$ by
Lemma~\ref{lemma:35-q-equality}, which leads to a~contradiction.

If $r=6$, then $g+1\leqslant 4d-g+1=5$ by
Lemma~\ref{lemma:35-q-equality}, which gives $g=0$ and $d=1$.
\end{proof}

\begin{lemma}
\label{lemma:six-lines} The equality $r=1$ holds.
\end{lemma}

\begin{proof}
Suppose that $r\ne 1$. Then $Z$ is a~disjoint union of $6$ lines
by Lemma~\ref{lemma:A6-reducible-curves}, which implies that
either $Z=L_{1}+L_{2}+L_{3}+L_{4}+L_{5}+L_{6}$, or
$Z=L^{\prime}_{1}+L^{\prime}_{2}+L^{\prime}_{3}+L^{\prime}_{4}+L^{\prime}_{5}+L^{\prime}_{6}$.

Without loss of generality, we may assume that $S=L_{1}$ and
$Z=L_{1}+L_{2}+L_{3}+L_{4}+L_{5}+L_{6}$.

Let $\bar{F}\subset\bar{G}$ be the~stabilizer of $L_{1}$. Then
$\bar{F}\cong\A_{5}$. Let $\pi\colon U\to\mathbb{P}^{3}$
be the~blow up of $L_{1}$. Then
$$
K_{U}+\bar{D}+\Big(\mathrm{mult}_{L_{1}}\big(D\big)-1\Big)E\sim_{\mathbb{Q}}\pi^{*}\Big(K_{\mathbb{P}^{3}}+D\Big),
$$
where $E$ is the~$\pi$-exceptional divisor, and $\bar{D}$ is
the~proper transform of $D$ on $U$. One has
$2>\mathrm{mult}_{L_{1}}(D)>1$ by
Remark~\ref{remark:mult-2-inequality}. The group $\bar{F}$
naturally acts on  $E$ so that the~divisor $\bar{D}\vert_{E}$ is
$\bar{F}$-invariant.

We can identify the~surface $E$ with a~smooth quadric in
$\mathbb{P}^{3}$. The action of the~group $\bar{F}$~extends to
the~ambient space $\mathbb{P}^3$. Note that this action arises
from the~standard four-dimensional representation of the~group
$\bar{F}\cong\A_{5}$.

It follows from the~inequality $\mathrm{mult}_{L_{1}}(D)<2$ that
the~set $\mathbb{LCS}(U,\bar{D}+(\mathrm{mult}_{L_{1}}(D)-1)E)$
contains an irreducible reduced curve $C\subset E$ such that
$\pi(C)=L_{1}$. One has $C\in\mathbb{LCS}(E,\bar{D}\vert_{E})$ by
Theorem~\ref{theorem:adjunction}, which is impossible by
Lemma~\ref{lemma:quasric-surface-Dokshitzer}.
\end{proof}

We see that $r=1$, so that $Z=S$. By Lemma~\ref{lemma:sporadic-genera}
one has $g\in\{31,25,19,16,10\}$.

\begin{lemma}
\label{lemma:g31} The inequality $g\neq 31$ holds.
\end{lemma}

\begin{proof}
Suppose that $g=31$. Then it follows from
Lemma~\ref{lemma:35-q-equality} that
\begin{equation}\label{eq:g-31}
d=\frac{65-q}{4},
\end{equation}
which gives $q>0$. Hence $q\geqslant 5$ by
Lemma~\ref{lemma:small-q} and thus $d\le 15$ by~(\ref{eq:g-31}).
By~(\ref{eq:g-2d}) we have $g/2\leqslant d\leqslant 15$,
which is impossible.
\end{proof}

\begin{lemma}
\label{lemma:g25} The inequality $g\neq 25$ holds.
\end{lemma}

\begin{proof}
Suppose that $g=25$. Then it follows from
Lemma~\ref{lemma:35-q-equality} that
\begin{equation}\label{eq:g-25}
d=\frac{59-q}{4},
\end{equation}
which gives $q>0$. Hence $q\geqslant 11$ by
Lemma~\ref{lemma:small-q}
and thus $d\le 15$ by~(\ref{eq:g-25}).
By~(\ref{eq:g-2d}) we have $g/2\leqslant d\leqslant 12$,
which is impossible.
\end{proof}

\begin{lemma}
\label{lemma:g19} The inequality $g\neq 19$ holds.
\end{lemma}

\begin{proof}
Suppose that $g=19$. Then $S$ is not contained in a~cubic surface
by Lemma~\ref{lemma:no-cubic}. We~have
$h^0(\mathcal{O}_{\mathbb{P}^3}(3)\otimes\mathcal{O}_{S})\geqslant
h^0(\mathcal{O}_{\mathbb{P}^3}(3))=20$, because there is an~exact
sequence of the~cohomology groups.
$$
0\longrightarrow H^0\Big(\mathcal{O}_{\mathbb{P}^3}\big(3\big)\otimes\mathcal{I}\Big)\longrightarrow H^0\Big(\mathcal{O}_{\mathbb{P}^3}\big(3\big)\Big)\longrightarrow H^0\Big(\mathcal{O}_{\mathbb{P}^3}\big(3\big)\otimes\mathcal{O}_{S}\Big).%
$$

By the~Riemann--Roch theorem, we have
$$
20\leqslant h^0\Big(\mathcal{O}_{\mathbb{P}^3}\big(3\big)\otimes\mathcal{O}_{S}\Big)=3d-g+1+h^1\Big(\mathcal{O}_{\mathbb{P}^3}\big(3\big)\otimes\mathcal{O}_{S}\Big)=3d-18+h^1\Big(\mathcal{O}_{\mathbb{P}^3}\big(3\big)\otimes\mathcal{O}_{S}\Big),%
$$
which implies that
$h^1(\mathcal{O}_{\mathbb{P}^3}(3)\otimes\mathcal{O}_{S})\ne 0$,
because $d\leqslant 12$ by Lemma~\ref{lemma:35-q-equality}. Then
it follows from Theorem~\ref{theorem:Clifford} that $19\geqslant
3d/2+1\geqslant
h^0(\mathcal{O}_{\mathbb{P}^3}(3)\otimes\mathcal{O}_{S})\geqslant
20$, which is a contradiction.
\end{proof}

\begin{lemma}
\label{lemma:g16} The inequality $g\neq 16$ holds.
\end{lemma}

\begin{proof}
Suppose that $g=16$. Arguing as in the~proof of
Lemma~\ref{lemma:g19}, and keeping in mind that $d\le 11$ by
Lemma~\ref{lemma:35-q-equality}, we see that $18\geqslant
3d/2+1\geqslant
h^0(\mathcal{O}_{\mathbb{P}^3}(3)\otimes\mathcal{O}_{S})\geqslant
20$, which is a contradiction.
\end{proof}

Therefore, we see that $g=10$. Then $d\geqslant 9$ by
Theorem~\ref{theorem:Castelnuovo}.

\begin{lemma}
\label{lemma:g-10-d-9} The inequality $d\neq 9$ holds.
\end{lemma}

\begin{proof}
Suppose that $d=9$. It follows from
Lemma~\ref{lemma:small-semiinvariants} that
$h^0(\mathcal{O}_{\mathbb{P}^3}(2)\otimes\mathcal{I})\ne 1$, which
implies that $S=F_{1}\cap F_{2}$, where $F_{1}$ and $F_{2}$ are
cubic surfaces in $\mathbb{P}^{3}$ (see
\cite[Example~6.4.3]{Har77})

The group $\bar{G}$ cannot act non-trivially on the~pencil
generated by $F_1$ and $F_2$, which implies that the~surfaces
$F_1$ and $F_2$ must be $\bar{G}$-invariant. The latter is
impossible by Lemma~\ref{lemma:small-semiinvariants}.
\end{proof}

\begin{lemma}
\label{lemma:g-10-d-10} The inequality $d\neq 10$ holds.
\end{lemma}

\begin{proof}
Suppose that $d=10$. Then $q=4$ by
Lemma~\ref{lemma:35-q-equality}, which is impossible by
Lemma~\ref{lemma:small-q}.
\end{proof}

Thus we see that $d\geqslant 11$. Then $q=0$ and $d=11$ by
Lemma~\ref{lemma:35-q-equality}.

Take a subgroup~$\bar{F}\subset\bar{G}$  such that
$\bar{F}\cong\A_{5}$ and the~embedding
$\bar{F}\subset\bar{G}$~is~non-standard~(see~De\-fi\-ni\-tion~\ref{definition:standard-subgroups}).
Then there is an~$\bar{F}$-invariant twisted cubic curve
$C\subset\mathbb{P}^{3}$ by
Lemma~\ref{lemma:A5-invariant-quartic}.

Let $R$ be the~quartic surface in $\mathbb{P}^{3}$ that is swept
out by the~lines that are tangent to $C$. Then
the~surface $R$ is $\bar{F}$-invariant, and
the~curve $S$ is not contained in the~surface $R$, because
$q=0$.
Put $\Sigma=R\cap S$. We have
$\Sigma=\Sigma_{1}\cup\ldots\cup\Sigma_{r}$, where $\Sigma_{i}$ is
a~$\bar{F}$-orbit. Hence, we have
$$
44=R\cdot S=\sum_{i=1}^{r}a_{i}\big|\Sigma_{i}\big|
$$
for some positive integers $a_{1},\ldots,a_{r}$. We may assume
that $|\Sigma_{1}|\geqslant\ldots\geqslant |\Sigma_{r}|$. But
$|\Sigma_{i}|\in\{12,20,30,60\}$ for every $i\in\{1,\ldots,r\}$ by
Lemma~\ref{lemma:stabilizers}. Thus, we see that
$|\Sigma_{1}|=20$.

Let $O\in\Sigma_{1}$ be a~point, and let
$\bar{F}_{O}\subset\bar{F}$ be the~stabilizer of the~point $O$.
Then $\bar{F}_{O}\cong\mathbb{Z}_3$.

Let $\Gamma$ be the~$\bar{G}$-orbit of the~point $O$, and let
$\bar{G}_{O}\subset\bar{G}$ be the~stabilizer of the~point $O$.
Then
the~order of the~group $\bar{G}_{O}$ must divide $18$. Moreover,
the~group $\bar{G}_{O}$ is cyclic by Lemma~\ref{lemma:stabilizers},
which implies that $\bar{G}_{O}\cong\mathbb{Z}_3$. Hence, we see
that $|\Gamma|=120$, which contradicts
Lemma~\ref{lemma:sporadic-genera}.

The obtained contradiction completes the~proof of
Theorem~\ref{theorem:A6-NFI}.

\section{Segre cubic}
\label{section:cubic}

Let $G\subset\SL_5(\mathbb{C})$ be a~subgroup such that
$G\cong\A_{6}$. Let
$\phi\colon\SL_5(\mathbb{C})\to\mathrm{Aut}(\mathbb{P}^{4})$~be
the~natural projection. Put $W=\mathbb{C}^{5}$ and
$\bar{G}=\phi(G)\subset\mathrm{Aut}(\mathbb{P}^{4})$. Then
the~space $W$ is an irreducible representation of the~group
$G\cong \bar{G}\cong \A_{6}$, and there is a unique cubic
hypersurface $X\subset\mathbb{P}^{4}$ that is $\bar{G}$-invariant.
Let us identify $X$ with a complete intersection in
$\mathbb{P}^{5}$ that is given by the~equation
$$
\sum_{i=0}^{5}x_{i}=\sum_{i=0}^{5}x_{i}^{3}=0\subset\mathbb{P}^{5}\cong\mathrm{Proj}\Big(\mathbb{C}\big[x_{0},x_{1},x_{2},x_{3},x_{4},x_{5}\big]\Big),
$$
and let us identify $\bar{G}$ with a~subgroup of the~group
$\mathrm{Aut}(X)$ (cf. Example~\ref{example:Segre-cubic}).

Let $O\in X$ be a~point, and let $\bar{F}\subset\bar{G}$ be its
stabilizer.

\begin{remark}
\label{remark:tangent-space-representation} Let $\tilde{T}$ be
the~ affine tangent space to $\mathbb{P}^{4}$ at the~point
$O\in\mathbb{P}^{4}$. Then $\bar{F}$ naturally acts on the~space
$\tilde{T}$. Let us consider $W$ as a~representation of the~group
$\bar{F}$. One has
$$\tilde{T}\cong W\slash W_{O}\otimes W_{O}^{*},$$
where $W_{O}$ is the~one-dimensional subrepresentation
of $\bar{F}$ that corresponds to the~point $O\in\mathbb{P}^{4}$.
\end{remark}

Let $\Sigma$ be the~$\bar{G}$-orbit of the~point $O\in X$.

\begin{lemma}
\label{lemma:Segre-small-orbits}  Suppose that $|\Sigma|\leqslant
15$. Then $|\Sigma|\in\{10,15\}$.
\end{lemma}

\begin{proof}
One has $|\Sigma|\ne 1$, because $W$ is an irreducible
representation of the~group $\bar{G}$. Hence
$|\Sigma|\in\{6,10,15\}$ by Lemma~\ref{lemma:max-subgroup}.
Suppose that $|\Sigma|=6$. Then $\bar{F}\cong\A_5$ by
Lemma~\ref{lemma:max-subgroup}.

Let us consider $W$ as a~representation of the~group $\bar{F}$.
Then $W$ is reducible and \mbox{$W\cong W_{t}\oplus W_4$}, where $W_{t}$
and $W_4$ are the~trivial and a four-dimensional representations
of the~group~$\bar{F}$, respectively. The embedding
$\bar{F}\subset \bar{G}$ is standard (see
Definition~\ref{definition:standard-subgroups}), because $W$ is
reducible.

Note that $W_t$ is the~only one-dimensional subrepresentation of
the representation $W$, because the representation $W_4$ is
irreducible. Hence, the~set $\Sigma$ must contain a~point in
$\mathbb{P}^{5}$ that is given~by $x_{0}=\ldots=x_{4}=-x_{5}/5$,
which~is~impossible, because this point does not belong to
the~hypersurface $X\subset\mathbb{P}^{5}$.
\end{proof}

Let $T$ be the~affine tangent space to $X$ at the~point $O$. Then
$\bar{F}$ naturally acts on~$T$.

\begin{lemma}
\label{lemma:Segre-small-orbits-10}  Suppose that $|\Sigma|=10$.
Then $\Sigma=\mathrm{Sing}(X)$ and
$\bar{F}\cong(\mathbb{Z}_3\times\mathbb{Z}_3)\rtimes
\mathbb{Z}_4$. Moreover,~$T$ is an~irreducible four-dimensional
representation of the~group $\bar{F}$.
\end{lemma}

\begin{proof}
One has $\bar{F}\cong
(\mathbb{Z}_3\times\mathbb{Z}_3)\rtimes \mathbb{Z}_4$ by
Lemma~\ref{lemma:max-subgroup}.

Up to conjugation, the~group $\bar{G}$ has a~unique subgroup that
is isomorphic to $\bar{F}$, which implies that we may assume that
$\bar{F}$ fixes the~point  $[1:-1:1:-1:1:-1]\in\mathrm{Sing}(X)$.

Let us consider $W$ as a~representation of the~group $\bar{F}$,
and let $W_{1}\subset W$ be the~one-dimensional subrepresentation
of the~group $\bar{G}$ that corresponds to the~point
$[1:-1:1:-1:1:-1]$. Then $W\cong W_{1}\oplus W_{4}$, where $W_{4}$
is some~four-dimensional representation of the~group $\bar{F}$.

Let $\chi_1$  and $\chi$ be the~characters of the~representations
$W_{1}$ and $W$, respectively. The~values of the~characters
$\chi_1$ and $\chi$ and the~structure of the~subgroup $\bar{F}$
are given in the~following table:
\begin{center}\renewcommand\arraystretch{1.3}
\begin{tabular}{|c|c|c|c|c|c|}
\hline & $[4,2]$ & $[3,3]$ &
$[3,1,1,1]$ & $[2,2,1,1]$ & $e$\\
\hline
\# & $18$ & $4$ & $4$ & $9$ & $1$\\
\hline
$\chi$ & $-1$ & $-1$ & $2$ & $1$ & $5$\\
\hline
$\chi_1$ & $-1$ & $1$ & $1$ & $1$ & $1$\\
\hline
\end{tabular}
\end{center}
where we use notation similar to the~ones used in the~proof of
Lemma~\ref{lemma:small-semiinvariants}.

We see that $W_{4}$ is an~irreducible representation of the~group
$\bar{F}$. Thus, without loss of generality, we may assume that
$O=[1:-1:1:-1:1:-1]$. Then $T\cong W_{4}\otimes W_1^{*}$ by
Remark~\ref{remark:tangent-space-representation}, which implies
that $T$ is an~irreducible representation of the~group $\bar{F}$.
\end{proof}

Let $\Gamma$ be the~$\bar{G}$-orbit of the~point
$[0:0:0:0:1:-1]\in X\setminus\mathrm{Sing}(X)$. Then
$|\Gamma|=15$.

\begin{lemma}
\label{lemma:Segre-small-orbits-15} Suppose that $|\Sigma|=15$.
Then $\Sigma=\Gamma$, and $\bar{F}\cong\SS_4$.
Moreover, $T$ is an~irreducible three-dimensional representation of
the~group $\bar{F}$.
\end{lemma}

\begin{proof}
One has $\bar{F}\cong\SS_4$ by
Lemma~\ref{lemma:max-subgroup}. Up to conjugation, the~group
$\bar{G}$ contains exactly  two non-conjugate subgroups that are
isomorphic~to the~group $\bar{F}\cong\SS_{4}$ (see
Lemma~\ref{lemma:max-subgroup}).

Let us consider $W$ as a~representation of the~group $\bar{F}$.
Then $W$ contains some one-dimensional subrepresentation  $U$ of
the~group $\bar{F}$ that correspond to the~point $O\in X$.

Let $W_{t}$ be the~trivial~one-dimensional representations of
the~group~$\bar{F}$, and let $W_{1}$ be the~non-trivial
one-dimensional representations of the~group~$\bar{F}$. Then
either $U\cong W_{t}$ or $U\cong W_{1}$.

Let $\chi$ and $\chi_{1}$ be the~characters of $W$ and $W_{1}$,
respectively.

Suppose that
$\SS_{4}\cong\bar{F}\subset\bar{G}\cong\A_{6}$ is
a~non-standard embedding. The~values~of~$\chi_1$~and~$\chi$~and
the~struc\-ture of the~subgroup $\bar{F}$ are given in
the~following table:
\begin{center}\renewcommand\arraystretch{1.3}
\begin{tabular}{|c|c|c|c|c|c|}
\hline  & $[4,2]$ &
$[3,3]$ & $[2,2,1,1]$ & $[2,2,1,1]$ & $e$\\
\hline
\# & $6$ & $8$ & $3$ & $6$ & $1$\\
\hline
$\chi$ & $-1$ & $-1$ & $1$ & $1$ & $5$\\
\hline
$\chi_{1}$ & $-1$ & $1$ & $1$ & $-1$ & $1$\\
\hline
\end{tabular}
\end{center}
where we use notation similar to the~ones used in the~proof of
Lemma~\ref{lemma:small-semiinvariants}, and we divide the~elements
of type $[2,2,1,1]$ that are contained in the~subgroup $\bar{F}$
into two classes with respect to the~values of the~character
$\chi_{1}$.  We see that $W$ contains no one-dimensional
subrepresentations.

Thus $\SS_{4}\cong\bar{F}\subset\bar{G}\cong\A_{6}$
is a~standard embedding. The~values of $\chi_1$ and $\chi$ and
the~structure of the~subgroup $\bar{F}$ are given in the~following
table:
\begin{center}\renewcommand\arraystretch{1.3}
\begin{tabular}{|c|c|c|c|c|c|}
\hline  & $[4,2]$ &
$[3,1,1,1]$ & $[2,2,1,1]$ & $[2,2,1,1]$ & $e$\\
\hline
\# & $6$ & $8$ & $3$ & $6$ & $1$\\
\hline
$\chi$ & $-1$ & $2$ & $1$ & $1$ & $5$\\
\hline
$\chi_{1}$ & $-1$ & $1$ & $1$ & $-1$ & $1$\\
\hline
\end{tabular}
\end{center}
which implies that $W\cong W_{t}\oplus W_{1}\oplus W_3$, where
$W_3$ is an~irreducible three-dimensional representation of
the~group~$\bar{F}$. Hence $T$ is an~irreducible representation of
the~group~$\bar{F}$ by
Remark~\ref{remark:tangent-space-representation}.

Suppose that $\Sigma\ne\Gamma$. Then $[1:1:1:1:-2:-2]\in\Sigma$.
But $[1:1:1:1:-2:-2]\not\in X$, which gives a contradiction.
\end{proof}

Let $H$ be a~general hyperplane section of the~cubic
$X\subset\mathbb{P}^{4}$. Then the $\bar{G}$-invariant subgroup of
the~group $\mathrm{Cl}(X)$ is generated by $H$ (see
Example~\ref{example:Segre-cubic}). The main purpose of this
section is to prove

\begin{theorem}
\label{theorem:Segre-cubic}  The threefold $X$ is
$\bar{G}$-birationally superrigid and
$\mathrm{Bir}^{\bar{G}}(X)\cong\SS_{6}$.
\end{theorem}

By Corollary~\ref{corollary:normalizers}, the isomorphism
$\mathrm{Bir}^{\bar{G}}(X)\cong\SS_{6}$ follows from
$\bar{G}$-birational superrigidity of $X$ and classification of
primitive subgroups in $\SL_5(\C)$ (see \cite{Fe71}). In the
remaining part of this section, we will prove that $X$ is
$\bar{G}$-birationally superrigid.

Suppose that $X$ is not $\bar{G}$-birationally superrigid. Let us
derive a~contradiction. It follows from \cite[Theorem~4.2]{Co95}
or \cite[Theorem~1.4.1]{Ch05umn} that there is a (non-empty)
$\bar{G}$-invariant
linear system~$\mathcal{M}$ on $X$ such that $\mathcal{M}$ does
not have fixed components, and $(X,\lambda\mathcal{M})$ is not
canonical, where $\lambda\in\mathbb{Q}$ such that
$\lambda\mathcal{M}\sim_{\mathbb{Q}} -K_{X}$. By
Corollary~\ref{corollary:mult-by-2} there is $\mu\in\mathbb{Q}$
such that $\mu<2\lambda$ and $(X,\mu\mathcal{M})$ is strictly log
canonical. Let $S\subset X$ be a~minimal center in
$\mathbb{LCS}(X,\mu\mathcal{M})$. Then it follows from
Lemma~\ref{lemma:centers} that $S\cap g(S)\ne\varnothing$ if and
only if $S=g(S)$ for every $g\in\bar{G}$. Moreover, the center $S$
is not a~surface, since $\mathcal{M}$ has no fixed components.

By Lemma~\ref{lemma:Kawamata-Shokurov-trick}, there~is
a~$\bar{G}$-invariant effective $\mathbb{Q}$-divisor $D$ on $X$
such that the set $\mathbb{LCS}(X, D)$ consists of irreducible
components of the $\bar{G}$-orbit of $S$, the~log pair $(X,D)$ is
log canonical and $D\sim_{\mathbb{Q}}-\delta K_{X}$ for some
$\delta\in\mathbb{Q}$ such that $0<\delta<2$. Arguing as in
the~proof of~\cite[Lemma~2.11]{ChSh09}, we see that we can replace
$D$ by $(\mu-\epsilon)\mathcal{M}+\epsilon^{\prime}\mathcal{H}$,
where $\epsilon\in\mathbb{Q}_{>0}\ni\epsilon^{\prime}$ such that
$\epsilon\ll 1$ and $\epsilon^{\prime}\ll 1$, and $\mathcal{H}$ is
a~$\bar{G}$-invariant linear system such that $\mathcal{H}\sim
-nK_{X}$ holds for some $n\gg 0$, the~base locus of the~linear
system $\mathcal{H}$ coincides with $\mathbb{LCS}(X, D)$, and
$\delta=(\mu-\epsilon)/\lambda+\epsilon^{\prime}n$. Thus, without
loss of generality, we can replace $D$ by $\mu\mathcal{M}$ (cf.
Lemma~\ref{lemma:Kawamata-Shokurov-trick}). Therefore, without
loss of generality, we may assume that the set $\mathbb{LCS}(X,
\mu\mathcal{M})$ also consists of irreducible components of the
$\bar{G}$-orbit of $S$.

Let $\mathcal{I}$ be the~multiplier ideal sheaf of
$(X,\mu\mathcal{M})$, and let $\mathcal{L}$ be the~subscheme given
by $\mathcal{I}$.~Then
\begin{equation}
\label{equation:Segre-cubic-equality}
h^{0}\Big(\mathcal{O}_{\mathcal{L}}\otimes\mathcal{O}_{X}\big(2H\big)\Big)=h^{0}\Big(\mathcal{O}_{X}\big(2H\big)\Big)-h^{0}\Big(\mathcal{O}_{X}\big(2H\big)\otimes\mathcal{I}\Big)=15-h^{0}\Big(\mathcal{O}_{X}\big(2H\big)\otimes\mathcal{I}\Big)%
\end{equation}
by Theorem~\ref{theorem:Shokurov-vanishing}.

\begin{lemma}
\label{lemma:Segre-cubic-4} The center $S$ is not a~singular point
of the~threefold $X$.
\end{lemma}

\begin{proof}
Suppose that $S$ is a~singular point of the~threefold $X$. Then
$\mathrm{LCS}(X,\mu\mathcal{M})=\mathrm{Sing}(X)$, because
$\bar{G}$ acts transitively on the~set $\mathrm{Sing}(X)$. Recall
that $|\mathrm{Sing}(X)|=10$.

Let $\bar{F}\subset\bar{G}$ be a~stabilizer of the~point $S$. Then
$\bar{F}\cong (\mathbb{Z}_3\times\mathbb{Z}_3)\rtimes
\mathbb{Z}_4$ by Lemma~\ref{lemma:Segre-small-orbits-10}.

Let $\pi\colon U\to X$ be a~blow up of the~points
$\mathrm{Sing}(X)$, and let $E_{1},\ldots,E_{10}$ be the~
$\pi$-exceptional divisors. Then there is a~positive rational
number $m$ such that
$$
K_{U}+\mu\bar{\mathcal{M}}+\big(m-1\big)\sum_{i=1}^{10}E_{i}\sim_{\mathbb{Q}}\pi^{*}\Big(K_{X}+\mu\mathcal{M}\Big),
$$
where $\bar{\mathcal{M}}$ is the~proper transforms of the~linear
system $\mathcal{M}$ on the~variety $U$.

Note that it follows from \cite[Theorem~1.7.20]{Ch05umn} that
$m\geqslant 1$ (cf. \cite[Theorem~3.10]{Co00}).

We may assume that $\pi(E_{1})=S$. There is a~natural homomorphism
$\upsilon\colon\bar{F}\to\mathrm{Aut}(E_{1})$. Then
$\mathrm{lct}(E_{1},\upsilon(\bar{F}))\geqslant 1$ by
Lemmas~\ref{lemma:quasric-surface-lct} and
\ref{lemma:Segre-small-orbits-10}, because
$E_{1}\cong\mathbb{P}^{1}\times\mathbb{P}^{1}$.

Let us show that $\mathrm{lct}(E_{1},\upsilon(\bar{F}))<1$ to
derive a contradiction.

One can easily check that there exists a~two-dimensional linear
subspace $\Pi\subset\mathbb{P}^{4}$ such that
$|\Pi\cap\mathrm{Sing}(X)|=4$ and $\Pi\subset X$ (see
\cite[Section~3.2]{Hu96}). Let $C$ be a~general conic in $\Pi$
that contains $\Pi\cap\mathrm{Sing}(X)$. Then $C$ is not contained
in the base locus of the linear system $\mathcal{M}$, and $C$ is
irreducible. Let $\bar{C}$ be a~proper transform of the~curve $C$
on the~threefold $U$. Then
$$
\frac{8}{\mu}-4\frac{m}{\mu}>\frac{4}{\lambda}-4\frac{m}{\mu}=\bar{\mathcal{M}}\cdot\bar{C}\geqslant 0,%
$$
which implies that $m<2$. In particular, the~log pair $(U,\
\mu\bar{\mathcal{M}}+(m-1)\sum_{i=1}^{10}E_{i})$ is not Kawamata
log terminal along $E_{1}$. Hence
$\mathrm{lct}(E_{1},\upsilon(\bar{F}))<1$ by
Theorem~\ref{theorem:adjunction}.
\end{proof}

\begin{lemma}
\label{lemma:Segre-cubic-5} The center $S$ is a~curve.
\end{lemma}

\begin{proof}
Suppose that $S$ is a~point. Then $S\not\in\mathrm{Sing}(X)$ by
Lemma~\ref{lemma:Segre-cubic-5}. It follows from
(\ref{equation:Segre-cubic-equality})~that
$|\mathrm{LCS}(X,\mu\mathcal{M})|\leqslant 15$, so that
$\mathrm{LCS}(X,\mu\mathcal{M})$ is the~$\bar{G}$-orbit of
the~point \mbox{$[0:0:0:0:1:-1]$} by
Lemma~\ref{lemma:Segre-small-orbits}. There are $15$
two-dimensional linear subspaces $\Pi_1,\ldots,\Pi_{15}$ in
$\mathbb{P}^{4}$  such that
$|\Pi_{i}\cap\mathrm{LCS}(X,\mu\mathcal{M})|=6$ and
$X\cap\Pi_{i}=L_{i}^{1}+L_{i}^{2}+L_{i}^{3}$ for any
$i\in\{1,\ldots,15\}$, where $L_{i}^{j}$ is a~line such that
$$
\Big(\big(L_{i}^{1}\cap L_{i}^{2}\big)\cup\big(L_{i}^{1}\cap L_{i}^{3}\big)\cup\big(L_{i}^{2}\cap L_{i}^{3}\big)\Big)\bigcap\mathrm{LCS}\big(X,\mu\mathcal{M}\big)=\varnothing%
$$
and $L_{i}^{j}=L_{r}^{s}$ if and only if $(i,j)=(r,s)$. Note that
the~curve $\sum_{i=1}^{15}(L_{i}^{1}+L_{i}^{2}+L_{i}^{3})$ is
a~$\bar{G}$-orbit of the~line $L_{1}^{1}$. Without loss of
generality, we may assume that $S\in L_{1}^{1}$.

Let $M_{1}$ and $M_{2}$ be general surfaces in the~linear system
$\mathcal{M}$. Put
$$M_{1}\cdot
M_{2}=\gamma\sum_{i=1}^{15}(L_{i}^{1}+L_{i}^{2}+L_{i}^{3})+\Omega,$$
where $\Omega$ is an effective cycle such that
$L_{i}^{j}\not\subseteq\mathrm{Supp}(\Omega)$ for every
$i\in\{1,\ldots,15\}$ and $j\in\{1,2,3\}$, and $\gamma$ is
a~non-negative rational number. Put
$m=\mathrm{mult}_{S}(M_{1}\cdot M_{2})$, and let $D$ be a~general
surface in $|H|$ that contains the~lines $L_{1}^{1}$, $L_{1}^{2}$
and $L_{1}^{3}$. Then
$$
\frac{12}{\lambda^{2}}-3\gamma=D\cdot \Big(\Omega+\gamma\sum_{i=2}^{15}\Big(L_{i}^{1}+L_{i}^{2}+L_{i}^{3}\Big)\Big)\geqslant 6(m-\gamma),%
$$
which implies that $\gamma\geqslant 2m-4/\lambda^{2}$. Therefore,
we see that
$$
\frac{12}{\lambda^{2}}=H\cdot M_{1}\cdot M_{2}=45\gamma+H\cdot\Omega\geqslant 45\gamma\geqslant 45\Big(2m-\frac{4}{\lambda^{2}}\Big),%
$$
which implies that $m\leqslant 32/(15\lambda^{2})$. In particular,
we see that
\begin{equation}\label{eq:15}
\mathrm{mult}_{S}\big(\mathcal{M}\big)\leqslant\sqrt{m}\leqslant\frac{2\sqrt{2}}{\sqrt{15}\lambda}<\frac{3}{2\lambda}<\frac{3}{\mu}.%
\end{equation}

Let $\pi\colon U\to X$ be a~blow up of the~point $S$, and let $E$
be the~$\pi$-exceptional divisor. Then
$$
K_{U}+\mu\bar{\mathcal{M}}+\Big(\mu\mathrm{mult}_{S}\big(\mathcal{M}\big)-2\Big)E\sim_{\mathbb{Q}}\pi^{*}\Big(K_{X}+\mu\mathcal{M}\Big),
$$
where $\bar{\mathcal{M}}$ is the~proper transforms of
$\mathcal{M}$ on the~variety $U$.

Let $\bar{F}\subset\bar{G}$ be a~stabilizer of the~point $S$. Then
$\bar{F}\cong \SS_4$ by
Lemma~\ref{lemma:Segre-small-orbits-15}, and there is a~natural
homomorphism $\upsilon\colon\bar{F}\to\mathrm{Aut}(E)$. Note that
$\upsilon$ is
a~monomorphism by~Lemma~\ref{lemma:Segre-small-orbits-15}.

There is an irreducible proper subvariety $C\subsetneq
E\cong\mathbb{P}^{2}$ such that
$$g(C)\in\mathbb{LCS}(U,\
\mu\bar{\mathcal{M}}+(\mu\mathrm{mult}_{S}(\mathcal{M})-2)E)$$
for every $g\in\upsilon(\bar{F})$. Then $C$ is a~curve by
Theorem~\ref{theorem:connectedness} and
Lemma~\ref{lemma:Segre-small-orbits-15}.

Let $\bar{M}_{1}$ and $\bar{M}_{2}$ be general surfaces in
$\bar{\mathcal{M}}$. Then it follows from \cite[Theorem~3.1]{Co00}
that
$$
\mathrm{mult}_{g(C)}\Big(\bar{M}_{1}\cdot \bar{M}_{2}\Big)\geqslant \frac{4}{\mu^{2}}\Big(3-\mu\mathrm{mult}_{S}\big(\mathcal{M}\big)\Big).%
$$
for every $g\in\upsilon(\bar{F})$. Let $\delta$ be the~degree in
$E\cong\mathbb{P}^2$ of the~ $\upsilon(\bar{F})$-orbit of
the~curve $C$. Then
$$
\frac{128}{15\mu^{2}}\geqslant\frac{32}{15\lambda^{2}}\geqslant
m\geqslant \mathrm{mult}_{S}^{2}\big(\mathcal{M}\big)+
\delta\mathrm{mult}_{C}\Big(\bar{M}_{1}\cdot \bar{M}_{2}\Big)\geqslant \mathrm{mult}_{S}^{2}\big(\mathcal{M}\big)+\frac{8}{\mu^{2}}\Big(3-\mu\mathrm{mult}_{S}\big(\mathcal{M}\big)\Big),
$$
because $\delta\geqslant 2$ by Lemma~\ref{lemma:Segre-small-orbits-15}.
The latter value is greater than $9/\mu^{2}$,
which can be easily seen using the elementary properties of quadratic
forms and~(\ref{eq:15}).
The obtained contradiction completes the~proof.
\end{proof}

By Theorem~\ref{theorem:Kawamata}, the~curve $S$ is a~smooth curve
in $\mathbb{P}^{4}$ of degree $d$ and genus $g\leqslant d$. Put
$q=h^{0}(\mathcal{O}_{X}(2H)\otimes\mathcal{I})$, let $Z$ be
the~$\bar{G}$-orbit of the~curve $S$, let $r$ be the~number of
irreducible components of the~curve~$Z$.

\begin{lemma}
\label{lemma:15-q-equality} The equality $r(2d-g+1)=15-q$ holds.
\end{lemma}
\begin{proof}
The equality follows from (\ref{equation:Segre-cubic-equality})
and the~Riemann--Roch theorem, because $d\geqslant g$.
\end{proof}

\begin{corollary}
\label{corollary:Segre-cubic-g-bound} The inequality $g\leqslant
14$ holds.
\end{corollary}

Note that $Z$ is not contained in a~hyperplane in
$\mathbb{P}^{4}$, since $W$ is an irreducible
$\bar{G}$-representation.

\begin{lemma}
\label{lemma:Segre-A6-reducible-curves} The equality $r=1$ holds.
\end{lemma}

\begin{proof}
Suppose that $r\geqslant 2$. Then $r\leqslant 15$ by
Lemma~\ref{lemma:15-q-equality}, which implies that
$r\in\{6,10,15\}$ by Lemma~\ref{lemma:max-subgroup}. If $q=0$,
then $2d-g+1=1$ by Lemma~\ref{lemma:15-q-equality}, which is
impossible, because $g\leqslant d$.

We have $q\geqslant 1$. Then $r(g+1)\leqslant 14$ and $g\leqslant
1$. If $g=0$, then $r(2d+1)\leqslant 14$ by
Lemma~\ref{lemma:15-q-equality}, which implies a~contradictory
inequality $d\leqslant 0$. We see that $g=1$. Thus $2rd\leqslant
14$ by Lemma~\ref{lemma:15-q-equality}, which implies that $d=1$
and $g=0$, that contradicts the~equality $g=1$.
\end{proof}

There is a~natural monomorphism
$\theta\colon\bar{G}\to\mathrm{Aut}(S)$ (see
Lemma~\ref{lemma:Segre-small-orbits}).

\begin{lemma}
\label{lemma:Segre-g-must-be-10} The equality $g=10$ holds.
\end{lemma}

\begin{proof}
The required assertion follows from
Lemmas~\ref{lemma:15-q-equality} and \ref{lemma:sporadic-genera}.
\end{proof}

The equality $g=10$ and Lemma~\ref{lemma:15-q-equality} imply that
$d\leqslant 12$.

\begin{lemma}
\label{lemma:Segre-q-must-be-0} The equality $q=0$ holds.
\end{lemma}

\begin{proof}
Let $\bar{\Psi}\subset\bar{G}$ be a~subgroup such that
$\bar{\Psi}\cong\A_{5}$ and the~embedding
$\A_5\cong\bar{\Psi}\subset\bar{G}\cong\A_6$ is standard. There is
a~$\bar{\Psi}$-invariant hyperplane section $H\subset X$. Note
that $S\not\subset H$. We have $|H\cap S|\leqslant d\leqslant 12$,
which implies that $|H\cap S|=12$ by
Lemma~\ref{lemma:stabilizers}, because $H\cap S$ is
$\bar{\Psi}$-invariant. Then $q=0$ by
Lemma~\ref{lemma:15-q-equality}.
\end{proof}

Let $Q$ be the~$\bar{G}$-invariant quadric in $\mathbb{P}^4$ (cf.
Example~\ref{example:quadric-threefold}). Then $S\not\subset Q$,
because $q=0$.

Put $\Delta=Q\cap S$. Then $|\Delta|\leqslant 24$.  Let
$\bar{\Psi}\subset\bar{G}$ be a~stabilizer of a~point in $\Delta$.
Then
$$
\big|\bar{\Psi}\big|\geqslant
\frac{|\bar{G}|}{|\Delta|}\geqslant\frac{360}{24}>6,
$$
which is impossible by Lemma~\ref{lemma:stabilizers}.

The obtained contradiction completes the~proof of
Theorem~\ref{theorem:Segre-cubic}.

\section{Quadric threefold}
\label{section:quadric}

Let $G\subset\SL_5(\mathbb{C})$ be a~subgroup such that
$G\cong\A_{6}$, and let
$\phi\colon\SL_5(\mathbb{C})\to\mathrm{Aut}(\mathbb{P}^{4})$ be
the~natural projection. Put $W=\mathbb{C}^{5}$ and
$\bar{G}=\phi(G)$. Then there is a~smooth quadric hypersurface
$Q\subset\mathbb{P}^{4}$ that is $\bar{G}$-invariant. Let us
identify $Q$ with a~smooth complete intersection in
$\mathbb{P}^{5}$ that is given by the~equation
$$
\sum_{i=0}^{5}x_{i}=\sum_{i=0}^{5}x_{i}^{2}=0\subset\mathbb{P}^{5}\cong\mathrm{Proj}\Big(\mathbb{C}\big[x_{0},x_{1},x_{2},x_{3},x_{4},x_{5}\big]\Big),
$$
and let us identify $\bar{G}$ with a~subgroup of the~group
$\mathrm{Aut}(Q)$ (cf. Example~\ref{example:quadric-threefold}).

Let $O\in Q$ be a~point, let $\bar{F}\subset\bar{G}$ be its
stabilizer, and let $\Sigma$ be its $\bar{G}$-orbit.

\begin{lemma}
\label{lemma:quadric-small-orbits} Suppose that $|\Sigma|\leqslant
30$. Then $|\Sigma|=30$, and there exists a cubic hypersurface
$X\subset\mathbb{P}^{4}$~such that $\Sigma\subset X$ and
$Q\not\subset X$.
\end{lemma}

\begin{proof}
One has $|\Sigma|\ne 1$, because $W$ is an irreducible
representation of the~group $\bar{G}$. Then
$|\Sigma|\in\{6,10,15,20,30\}$ by Lemma~\ref{lemma:max-subgroup}.
Arguing as in the~proof of Lemmas~\ref{lemma:linear-complex},
\ref{lemma:Segre-small-orbits},~\ref{lemma:Segre-small-orbits-10}
and~\ref{lemma:Segre-small-orbits-15}, we see that
$|\Sigma|\in\{20,30\}$.

Let us consider $W$ as a~representation of the~group $\bar{F}$.
Then $W$ contains some one-dimensional subrepresentation $U$ of
the~group $\bar{F}$ corresponding to the~point $O\in Q$.

Suppose that $|\Sigma|=20$. Then $\bar{F}\cong
(\mathbb{Z}_3\times\mathbb{Z}_3)\rtimes\mathbb{Z}_2$ by
Lemma~\ref{lemma:max-subgroup}.

Let $W_{t}$ and $W_{1}$ be the~trivial~one-dimensional~and
the~non-trivial one-dimensional representations of~$\bar{F}$ (see
the~proof of Lemma~\ref{lemma:Segre-small-orbits-10}),
respectively. Then either $U\cong W_{t}$ or $U\cong W_{1}$.

Let $\chi$ and $\chi_{1}$ be the~characters of the~representations
$W$ and $W_{1}$, respectively. The~values of the~cha\-rac\-ters
$\chi_1$ and $\chi$ and the~structure of the~subgroup $\bar{F}$
are given in the~following table:
\begin{center}\renewcommand\arraystretch{1.3}
\begin{tabular}{|c|c|c|c|c|c|c|c|}
\hline & $[3,3]$ & $[3,1,1,1]$ &
$[2,2,1,1]$ & $e$\\
\hline
\# & $4$ & $4$ & $9$ & $1$\\
\hline
$\chi$ & $-1$ & $2$ & $1$ & $5$\\
\hline
$\chi_{1}$ & $1$ & $1$ & $-1$ & $1$\\
\hline
\end{tabular}
\end{center}
where we use notation similar to the~ones used in the~proofs of
Lemmas~\ref{lemma:small-semiinvariants},~\ref{lemma:Segre-small-orbits-10}
and~\ref{lemma:Segre-small-orbits-15}.

We see that $U\cong W_{1}$. Thus $[1:-1:1:-1:1:-1]\in\Sigma$. But
$[1:-1:1:-1:1:-1]\not\in Q$.

Therefore, we see that $|\Sigma|=30$. Then $\bar{F}\cong\A_4$ by
Lemma~\ref{lemma:max-subgroup}. The embedding
$\A_4\cong\bar{F}\subset\bar{G}\cong\A_6$ must be standard,
because otherwise the~representation $W$ would be an~irreducible
representation of the~group $\bar{F}$ (cf. the~proof of
Lemma~\ref{lemma:Segre-small-orbits-15}).

There are exactly two $\bar{F}$-invariant points in $Q$. These
points form a~subset
$$
\Big\{\big[1:1:1:1:-2+\sqrt{-2}:-2-\sqrt{-2}\big], \big[1:1:1:1:-2-\sqrt{-2}:-2+\sqrt{-2}\big]\Big\}\subset\Sigma.%
$$

Let $X$ be the~cubic threefold in $\mathbb{P}^{5}$ that is given
by
$$
\sum_{i=0}^{5}x_i=\big(x_0-x_1\big)\big(x_2-x_3\big)\big(x_4-x_5\big)=0,
$$
let $P$ be the~point
$[1:\omega:\omega^2:1:\omega:\omega^2]\in\mathbb{P}^{5}$, where
$\omega$ is a~non-trivial cube root of unity.~Then $\Sigma\subset
X\not\ni P\in Q$, which completes the~proof.
\end{proof}

The main purpose of this section is to prove the~following result.

\begin{theorem}
\label{theorem:quadric-threefold}  The quadric threefold $Q$ is
$\bar{G}$-birationally superrigid and
$\mathrm{Bir}^{\bar{G}}(Q)\cong\SS_{6}$.
\end{theorem}

By Corollary~\ref{corollary:normalizers}, the isomorphism
$\mathrm{Bir}^{\bar{G}}(Q)\cong\SS_{6}$ follows from
$\bar{G}$-birational superrigidity of~$Q$ and classification of
primitive subgroups in $\SL_5(\C)$ (see \cite{Fe71}). In the
remaining part of this section, we will prove that $Q$ is
$\bar{G}$-birationally superrigid.

Suppose that $Q$ is not $\bar{G}$-birationally superrigid. Arguing
as in the~proof of Theorem~\ref{theorem:Segre-cubic}, we~see that
there~is a~$\bar{G}$-invariant effective $\mathbb{Q}$-divisor $D$
on $Q$ such that the set $\mathbb{LCS}(Q, D)$ consists of
irreducible components of the $\bar{G}$-orbit of $S$, the~log pair
$(Q,D)$ is log canonical, and $D\sim_{\mathbb{Q}}-\delta K_{Q}$
for some positive rational number $\delta<2$, where $S$ is
a~minimal center in $\mathbb{LCS}(Q, D)$ such that either $S$ is
a~point or a~smooth curve.

Let $\mathcal{I}$ be the~multiplier ideal sheaf of the~log pair
$(Q, D)$, let $\mathcal{L}$ be subscheme that is given~by
the~ideal sheaf $\mathcal{I}$, and let $H$ be a~general hyperplane
section of the~threefold $Q\subset\mathbb{P}^{3}$. Then
\begin{equation}
\label{equation:quadric-threefold-equality}
h^{0}\Big(\mathcal{O}_{\mathcal{L}}\otimes\mathcal{O}_{Q}\big(3H\big)\Big)=h^{0}\Big(\mathcal{O}_{Q}\big(3H\big)\Big)-
h^{0}\Big(\mathcal{O}_{Q}\big(3H\big)\otimes\mathcal{I}\Big)=30-h^{0}\Big(\mathcal{O}_{Q}\big(3H\big)\otimes\mathcal{I}\Big)%
\end{equation}
by Theorem~\ref{theorem:Shokurov-vanishing}.

\begin{lemma}
\label{lemma:quadric-threefold-points} The center $S$ is a~curve.
\end{lemma}

\begin{proof}
The required assertion follows from
(\ref{equation:quadric-threefold-equality}) and
Lemma~\ref{lemma:quadric-small-orbits}.
\end{proof}

By Theorem~\ref{theorem:Kawamata}, the~curve $S$ is a~smooth curve
of degree $d$ and genus $g$ such that
\begin{equation}\label{eq:g-3d1}
g\leqslant \frac{3d+1}{2}.
\end{equation}
Put $q=h^{0}(\mathcal{O}_{Q}(3H)\otimes\mathcal{I})$, let $Z$ be
the~$\bar{G}$-orbit of the~curve $S$, let $r$ be the~number of
irreducible components of the~curve~$Z$.

\begin{lemma}
\label{lemma:quadric-30-q-equality} The equality $r(3d-g+1)=30-q$
holds.
\end{lemma}
\begin{proof}
The equality follows from
(\ref{equation:quadric-threefold-equality}) and the~Riemann--Roch
theorem, because $3d\geqslant 2g-1$.
\end{proof}

Let $X$ be the~cubic threefold in $\mathbb{P}^{5}$ that is given
by $\sum_{i=0}^{5}x_{i}=\sum_{i=0}^{5}x_{i}^{3}=0$
(cf. Example~\ref{example:Segre-cubic}).

\begin{lemma}
\label{lemma:quadric-small-q} Suppose that $q\geqslant 2$. Then
$q\geqslant 5$.
\end{lemma}

\begin{proof}
Suppose that $2\leqslant q\leqslant 4$. It follows
from~\cite{Atlas} that $\bar{G}$ acts trivially on
the~$q$-dimensional vector space
$H^{0}(\mathcal{O}_{Q}(3H)\otimes\mathcal{I})$. But $Q\cap X$ is
the~only $\bar{G}$-invariant surface in $|\mathcal{O}_{Q}(3H)|$.
\end{proof}

Note that $Z$ is not contained in a~hyperplane in
$\mathbb{P}^{4}$, since $W$ is an irreducible
$\bar{G}$-representation.

\begin{lemma}
\label{lemma:quadric-threefold-reducible-curves} The curve $S$ is
irreducible.
\end{lemma}
\begin{proof}
Suppose that $r\geqslant 2$. By Lemma~\ref{lemma:max-subgroup},
either $r=6$, or $r\geqslant 10$.

Suppose that $q=0$. If $r=6$, then it follows from
Lemma~\ref{lemma:quadric-30-q-equality} that
\begin{equation}\label{eq:quadric-2}
3d-g+1=5.
\end{equation}
From (\ref{eq:g-3d1}) and (\ref{eq:quadric-2}) we obtain that $g\le 5$.
Applying~(\ref{eq:quadric-2}) we see that $d\le 3$ and thus $g\le 1$.
Applying~(\ref{eq:quadric-2}) once again, we obtain a contradiction.
Therefore, we see that $r\geqslant 10$ and
\begin{equation}\label{eq:quadric-3}
3d-g+1\leqslant 3
\end{equation}
by Lemma~\ref{lemma:quadric-30-q-equality}.
By~(\ref{eq:g-3d1}) and~(\ref{eq:quadric-3})
we have $g\leqslant 3$.
Applying~(\ref{eq:quadric-3}) again, we obtain $d=1$ and thus $g=0$,
which is incompatible with~(\ref{eq:quadric-3}).

Suppose that $q=1$. Then $r(3d-g+1)=29$ by
Lemma~\ref{lemma:quadric-30-q-equality}, which is impossible by
Lemma~\ref{lemma:max-subgroup}.

Therefore, we see that $q\geqslant 2$. Hence $q\geqslant 5$ by
Lemma~\ref{lemma:quadric-small-q}. It follows from
Lemma~\ref{lemma:quadric-30-q-equality} that $r(3d-g+1)\leqslant
25$, which implies that $3d-g\leqslant 3$. We have $g\leqslant 4$,
because $2g-1\leqslant 3d$. Thus $d\leqslant 2$ and $g=0$.

Applying Lemma~\ref{lemma:quadric-30-q-equality}, we get $d=1$ and
$r=6$. Then $S$ is a~line, and $Z$ is a~union of six lines.

Let $\bar{\Psi}\subset\bar{G}$ be a~stabilizer of the~line $S$.
Then $\bar{\Psi}\cong\A_5$ by
Lemma~\ref{lemma:max-subgroup}.

Let us consider $W$ as a~representation of the~group
$\bar{\Psi}\cong\A_5$. Then either $W$ is irreducible, or $W\cong
W_{t}\oplus W_{4}$, where $W_{t}$ and $W_{4}$ are the~trivial
one-dimensional and the~standard four-dimensional representations
of the~group $\bar{\Psi}\cong\A_5$, respectively. In both cases,
the~line $S$ can not be $\bar{\Psi}$-invariant.
\end{proof}

We see that $r=1$. Hence $g\leqslant 30$ and $d\leqslant 19$ by
(\ref{equation:quadric-threefold-equality}), because $3d\geqslant
2g-1$.

\begin{lemma}
\label{lemma:in-Segre-cubic} The equalities $d=12$ and $g=10$
hold.
\end{lemma}

\begin{proof}
Let $\Pi_{i}\subset Q$ be a~hyperplane section that is cut out by
$x_{i}=0$, where $i\in\{0,\ldots,5\}$. Then
$\Pi_0\cap\ldots\cap\Pi_5=\varnothing$, which implies that,
without loss of generality, we may assume that
$S\not\subset\Pi_{0}$.

Let $\bar{\Psi}\subset\bar{G}$ be a~stabilizer of the~surface
$\Pi_{0}$. Then $\bar{\Psi}\cong\A_{5}$ and the~embedding
$\bar{\Psi}\subset \bar{G}$ must be standard. Hence we have
$$
19\geqslant d=\Pi_{0}\cdot S\geqslant\big|\Pi_{0}\cap S\big|,%
$$
which implies that $d=\Pi_{0}\cdot S=|\Pi_{0}\cap S|=12$ by
Lemma~\ref{lemma:orbits-for-standard-A5-in-A6}, because
$\Pi_{0}\cong\mathbb{P}^{1}\times\mathbb{P}^{1}$.

It follows from Theorem~\ref{theorem:Castelnuovo} that $g\leqslant
15$. Thus $g=10$ by Lemma~\ref{lemma:sporadic-genera}.
\end{proof}

Thus, it follows from Lemma~\ref{lemma:quadric-30-q-equality} that
$q=3$, which is impossible by Lemma~\ref{lemma:quadric-small-q}.

The obtained contradiction completes the~proof of
Theorem~\ref{theorem:quadric-threefold}.

\appendix

\section{Klein cubic threefold}
\label{section:Klein-cubic}

Put $\bar{\Gamma}=\PSL_2(\mathbb{F}_{11})$. Let $V$ be a~Fano
threefold~with terminal singularities such that $V$ admits a
non-trivial action of the~group $\bar{\Gamma}$, and
the~$\bar{\Gamma}$-invariant subgroup of the~group
$\mathrm{Cl}(V)$ is $\mathbb{Z}$.

\begin{example}
\label{example:Klein-cubic} If $V$ is a~smooth hypersurface in
$\mathbb{P}^4$ that is given by the~equation
$$
x_0^2x_1+x_1^2x_2+x_2^2x_3+x_3^2x_4+x_4^2x_0=0\subset\mathbb{P}^{4}\cong\mathrm{Proj}\Big(\mathbb{C}\big[x_{0},x_{1},x_{2},x_{3},x_{4}\big]\Big),
$$
then $V$ is non-rational by \cite[Theorem~0.12]{ClGr72}, and
$\mathrm{Aut}(V_{3})\cong\bar{\Gamma}$ by~\cite{Ad78}.
\end{example}

\begin{example}[{\cite[Example~2.9]{Pr09}}]
\label{example:V14} There is a~non-trivial action of the group
$\bar{\Gamma}$ on $\mathrm{Gr}(2,6)$, and
$\mathrm{Pic}(\mathrm{Gr}(2,6))=\mathbb{Z}[H]$ for some very ample
divisor $H$. The linear system~$|H|$ gives an embedding
$\zeta\colon\mathrm{Gr}(2,6)\to\mathbb{P}^{14}$, which induces
a~non-trivial action of the group $\bar{\Gamma}$ on
$\mathbb{P}^{14}$. Put $V=\zeta(\mathrm{Gr}(2,6))\cap\Pi$, where
$\Pi$ is the~unique $\bar{\Gamma}$-invariant linear subspace
$\Pi\subset\mathbb{P}^{14}$ such that $\mathrm{dim}(\Pi)=9$. Then
the~variety $V$ is a~smooth Fano threefold such that
$\mathrm{Pic}(V)\cong\mathbb{Z}$ and $-K_{V}^{3}=14$ (see
\cite{IsPr99}). Furthermore, the~variety $V$ admits a non-trivial
action of the~group $\bar{\Gamma}$.
\end{example}

Let $V_{3}$ and $V_{14}$ be the~threefolds that are constructed in
Examples~\ref{example:Klein-cubic} and~\ref{example:V14},
respectively.

\begin{theorem}[{\cite[Theorem~1.5]{Pr09}}]
\label{theorem:Yura-Cremona-non-rational} There is a
$\bar{\Gamma}$-equivariant birational map $\chi\colon V\dasharrow
U$ such~that either $U\cong V_{3}$ or $U\cong V_{14}$.
\end{theorem}

\begin{remark}[{see~\cite[Remark~2.10]{Pr09}}]
\label{remark:Palatini-quartic} The varieties $V_{3}$ and
$V_{14}$ are birationally equivalent.
\end{remark}

The main purpose of this section is to prove the~following result
(cf. \cite[Remark~2.10]{Pr09}).

\begin{theorem}
\label{theorem:Klein-cubic}
The varieties $V_3$ and $V_{14}$ are $\bar{\Gamma}$-birationally superrigid.
\end{theorem}

\begin{corollary}
\label{corollary:Klein-cubic-baby} The birational map $\chi\colon
V\dasharrow U$ in Theorem~\ref{theorem:Yura-Cremona-non-rational}
is biregular.
\end{corollary}

\begin{corollary}
\label{corollary:Klein-cubic-main} There exists no
$\bar{\Gamma}$-equivariant birational map $V_{14}\dasharrow V_{3}$.
\end{corollary}

\begin{corollary}
\label{corollary:Klein-cubic} Up to conjugation, the~group
$\mathrm{Bir}(V_{3})\cong\mathrm{Bir}(V_{14})$ contains exactly
$2$ subgroups~that are isomorphic~to~the~simple group
$\PSL_2(\mathbb{F}_{11})$.
\end{corollary}

Let us prove that $V_{14}$ is $\bar{\Gamma}$-birationally
superrigid. Suppose that $V_{14}$ is not
$\bar{\Gamma}$-birationally super\-ri\-gid. There is a (non-empty)
$\bar{\Gamma}$-invariant linear system $\mathcal{M}$ without fixed
components such that
$\lambda\mathcal{M}\sim_{\mathbb{Q}}-K_{V_{14}}$ for some
$\lambda\in\mathbb{Q}_{>0}$. Then $(V_{14},\lambda\mathcal{M})$ is
not canonical (see \cite[Theorem~1.4.1]{Ch05umn}).

There is $\mu\in\mathbb{Q}$ such that $\mu<2\lambda$ and
$(V_{14},\mu\mathcal{M})$ is strictly log canonical.

Let $S\subset V_{14}$ be a~minimal center in
$\mathbb{LCS}(V_{14},\mu\mathcal{M})$. Then
$\mathrm{dim}(S)\in\{0,1\}$.

By Lemma~\ref{lemma:Kawamata-Shokurov-trick}, there~is
a~$\bar{\Gamma}$-invariant effective $\mathbb{Q}$-divisor $D$ on
the~threefold $V_{14}$ such that the set $\mathbb{LCS}(V_{14}, D)$
consists of irreducible components of the $\Gamma$-orbit of $S$,
and $D\sim_{\mathbb{Q}} -\epsilon K_{V_{14}}$, where $\epsilon$ is
a~positive rational number such that $\epsilon<2$.

Let $\mathcal{I}$ be the~multiplier ideal sheaf of the~log pair
$(V_{14}, D)$, and let $\mathcal{L}(V_{14},D)$ be the~log canonical
singularities subscheme of the~log pair $(V_{14}, D)$. Then it
follows from Theorem~\ref{theorem:Shokurov-vanishing} that
\begin{equation}
\label{equation:Klein-cubic-equality}
h^{0}\Big(\mathcal{O}_{\mathcal{L}(V_{14},D)}\otimes\mathcal{O}_{V_{14}}\big(H\big)\Big)=h^{0}\Big(\mathcal{O}_{V_{14}}\big(H\big)\Big)-h^{0}\Big(\mathcal{O}_{V_{14}}\big(H\big)\otimes\mathcal{I}\Big)=10-h^{0}\Big(\mathcal{O}_{V_{14}}\big(H\big)\otimes\mathcal{I}\Big),%
\end{equation}
where $H\in |-K_{V_{14}}|$.

\begin{lemma}
\label{lemma:Klein-cubic-4} The equality $\mathrm{dim}(S)=0$ is
impossible.
\end{lemma}

\begin{proof}
Suppose that $\mathrm{dim}(S)=0$. Let $\bar{F}\subset\bar{\Gamma}$
be a~stabilizer of a~point in $\mathrm{LCS}(V_{14}, D)$. Then
$|\mathrm{LCS}(V_{14}, D)|\leqslant 10$ by
(\ref{equation:Klein-cubic-equality}). Thus, we see that
$|\bar{F}|\geqslant|\bar{\Gamma}|/10=66$. Hence, we must have
$\mathrm{LCS}(V_{14}, D)=S$ and  $\bar{F}=\bar{\Gamma}$, because
there are no proper subgroups of $\bar{\Gamma}$ of order greater
than~$60$ (see \cite{Atlas}).

The~action of  $\bar{\Gamma}$ on the~tangent space to $V_{14}$ at
the~point $S$ gives a~faithful three-dimensional representation of
the~group $\bar{\Gamma}$, which does not exist (see~\cite{Atlas}).
\end{proof}

It follows from Theorem~\ref{theorem:Kawamata} that $S$ is
a~smooth curve of genus $g$ such that $2g-2<S\cdot H$.

\begin{lemma}
\label{lemma:auxiliary-5} The curve $S$ is $\bar{\Gamma}$-invariant.
\end{lemma}

\begin{proof}
Let $Z$ be the~$\bar{\Gamma}$-orbit of the~curve $S$. Then
$Z=\mathcal{L}(V_{14},D)$, because $(V_{14},D)$ is log canonical.

Suppose that  $S\ne Z$. Let $r$ be the~number of irreducible
components of $Z$. Then $r\geqslant 11$, because
there is no nontrivial homomorphism
$\bar{\Gamma}\to\SS_{r}$ if $1<r\leqslant 10$.

Using (\ref{equation:Klein-cubic-equality}) and the~Riemann--Roch
theorem, we see~that
$$
10\geqslant 10-h^{0}\Big(\mathcal{O}_{V_{14}}\big(H\big)\otimes\mathcal{I}\Big)
=h^{0}\Big(\mathcal{O}_{\mathcal{L}(V_{14},D)}\otimes\mathcal{O}_{V_{14}}\big(H\big)\Big)
=r\Big(S\cdot H-g+1\Big)\ge r,%
$$
because $\mathcal{L}(V_{14},D)=Z$ and $2g-2<S\cdot H$. Thus, we see that
$r\leqslant 10$.
\end{proof}

Therefore, there is a~natural homomorphism
$\theta\colon\bar{\Gamma}\to\mathrm{Aut}(S)$.

\begin{lemma}
\label{lemma:Klein-cubic-7} The homomorphism $\theta$ is
a~monomorphism.
\end{lemma}

\begin{proof}
Suppose that $\theta$ is not a~monomorphism. Then
$\mathrm{ker}(\theta)=\bar{\Gamma}$, because $\bar{\Gamma}$
simple. Let $P$ be a point in $S$. Then $P$ is
$\bar{\Gamma}$-invariant. The action of the~group~$\bar{\Gamma}$
on the~tangent space to the~threefold $V_{14}$ at the~point $P$
gives its faithful three-dimensional representation, which does
not exist (see~\cite{Atlas}).
\end{proof}

\begin{lemma}
\label{lemma:Klein-cubic-6} The inequality $g\geqslant 11$ holds.
\end{lemma}

\begin{proof}
Suppose that $g\leqslant 10$. Then $g\leqslant 1$ by
Theorem~\ref{theorem:Hurwitz-bound} since $|\bar{\Gamma}|=660$.

Moreover, it follows from the~classification of finite subgroups
of $\PSL_2(\mathbb{C})$ that the~monomorphism $\theta$
does not exist if $g=0$. Thus, we see that $g=1$ and
$\mathrm{Aut}(S)$ contains a~simple non-abelian subgroup
$\theta(\bar{\Gamma})\cong\PSL_2(\mathbb{F}_{11})$, which is
impossible, because $\mathrm{Aut}(S)$ is solvable.
\end{proof}

Using (\ref{equation:Klein-cubic-equality}) and the~Riemann--Roch
theorem, we see~that
$$
10\geqslant h^{0}\Big(\mathcal{O}_{\mathcal{L}(V_{14},D)}\otimes\mathcal{O}_{V_{14}}\big(H\big)\Big)=S\cdot H-g+1,%
$$
because $\mathcal{L}(V_{14},D)=Z$ and $2g-2<S\cdot H$. Then
$g\geqslant S\cdot H-9$. But $2g-2<S\cdot H$. Hence $2(S\cdot
H)-20\leqslant 2g-2<S\cdot H$, which implies that $S\cdot
H\leqslant 19$. Therefore, $g\leqslant 10$, which is impossible by
Lemma~\ref{lemma:Klein-cubic-6}.

The obtained contradiction shows that $V_{14}$ is
$\bar{\Gamma}$-birationally superrigid.

To complete the proof of Theorem~\ref{theorem:Klein-cubic}, we
assume that the~threefold $V_3$ is not~$\bar{\Gamma}$-birationally
superrigid. Then there is a $\bar{\Gamma}$-invariant linear system
$\mathcal{M}$ without fixed components such that
$\lambda\mathcal{M}\sim_{\mathbb{Q}} -K_{V_{3}}$ for some
$\lambda\in\mathbb{Q}_{>0}$. Then $(V_{3},\lambda\mathcal{M})$ is
not canonical (see \cite[Theorem~1.4.1]{Ch05umn}).

There is $\mu\in\mathbb{Q}$ such that $\mu<2\lambda$ and
$(V_{3},\mu\mathcal{M})$ is strictly log canonical.

Let $S\subset V_{3}$ be a~minimal center in
$\mathbb{LCS}(V_{3},\mu\mathcal{M})$. Then
$\mathrm{dim}(S)\in\{0,1\}$.

By Lemma~\ref{lemma:Kawamata-Shokurov-trick}, there~is
a~$\bar{\Gamma}$-invariant effective $\mathbb{Q}$-divisor $D$ on
the~threefold $V_{3}$ such that the set $\mathbb{LCS}(V_{3}, D)$
consists of irreducible components of the $\bar{\Gamma}$-orbit of
$S$, and $D\sim_{\mathbb{Q}} -\epsilon K_{V_{3}}$, where~$\epsilon$
is a~positive rational number such that $\epsilon<2$.

Let $\mathcal{I}$ be the~multiplier ideal sheaf of the~log pair
$(V_{3}, D)$, and let $\mathcal{L}(V_3,D)$ be the~log canonical
singularities subscheme of the~log pair $(V_{3}, D)$. Then it
follows from Theorem~\ref{theorem:Shokurov-vanishing} that
\begin{equation}
\label{equation:Klein-cubic-equality-cubic}
h^{0}\Big(\mathcal{O}_{\mathcal{L}(V_3,D)}\otimes\mathcal{O}_{V_{3}}\big(2H\big)\Big)=
h^{0}\Big(\mathcal{O}_{V_{3}}\big(2H\big)\Big)-
h^{0}\Big(\mathcal{O}_{V_{3}}\big(2H\big)\otimes\mathcal{I}\Big)=
15-h^{0}\Big(\mathcal{O}_{V_{3}}\big(2H\big)\otimes\mathcal{I}\Big),%
\end{equation}
where $H$ is an ample generator of the group $\mathrm{Pic}(V_3)$.

\begin{lemma}
\label{lemma:Klein-cubic-4-cubic} The equality $\mathrm{dim}(S)=0$ is
impossible.
\end{lemma}

\begin{proof}
Suppose that $\mathrm{dim}(S)=0$. Let $\bar{F}\subset\bar{\Gamma}$ be
a~stabilizer of a~point in $\mathrm{LCS}(V_{3}, D)$. Then
$$
\big|\bar{F}\big|=\frac{\big|\bar{\Gamma}\big|}{\Big|\mathrm{LCS}\Big(V_{3}, D\Big)\Big|}\geqslant\frac{\big|\bar{\Gamma}\big|}{15}=44%
$$
by (\ref{equation:Klein-cubic-equality-cubic}). Thus, if
$\bar{F}\ne\bar{\Gamma}$, then $\bar{F}$ is isomorphic to either $\A_5$
or $\Z_{11}\rtimes\Z_5$ (see~\cite{Atlas}).

We may identify $\bar{\Gamma}$ with a subgroup in
$\mathrm{Aut}(\mathbb{P}^{4})$. There is a subgroup $G\subset
\SL_5(\mathbb{C})$ such that
$$
\bar{\Gamma}=\phi\big(G\big)\subset\mathrm{Aut}\Big(\mathbb{P}^{4}\Big)\cong
\PGL_5\Big(\mathbb{C}\Big)
$$
and $G\cong\bar{\Gamma}$, where
$\phi\colon\SL_5(\mathbb{C})\to\mathrm{Aut}(\mathbb{P}^{4})$~is~the~natural
projection.

Put $W=\mathbb{C}^{5}$. Then $W$ is an irreducible representation
of the group $G$ (see~\cite{Ad78}), which implies that
$\bar{F}\ne\bar{\Gamma}$. If $F$ is a subgroup of the group $G$
such that $\phi(F)=\bar{F}$, then one can show that $W$ is an
irreducible representation of the group $F$, which is a
contradiction.
\end{proof}

It follows from Theorem~\ref{theorem:Kawamata} that $S$ is
a~smooth curve of genus $g$ such that $g\le S\cdot H$.

\begin{lemma}
\label{lemma:auxiliary-5-cubic} The curve $S$ is $\bar{\Gamma}$-invariant.
\end{lemma}

\begin{proof}
Let $Z$ be the~$\bar{\Gamma}$-orbit of the~curve $S$. Then
$Z=\mathcal{L}(V_3,D)$, because $(V_{3},D)$ is log canonical.

Suppose that  $S\ne Z$. Let $r$ be the~number of irreducible
components of $Z$. Then $r\geqslant 11$, because there is no
non-trivial homomorphism $\bar{\Gamma}\to\SS_{r}$ in the case when
$1<r\leqslant 10$.

Using (\ref{equation:Klein-cubic-equality-cubic}) and the~Riemann--Roch
theorem, we see~that
$$
15\geqslant 15-h^{0}\Big(\mathcal{O}_{V_{3}}\big(2H\big)\otimes\mathcal{I}\Big)=
h^{0}\Big(\mathcal{O}_{\mathcal{L}(V_3,D)}\otimes\mathcal{O}_{V_{3}}
\big(2H\big)\Big)=r\Big(2S\cdot H-g+1\Big),%
$$
because $\mathcal{L}(V_3,D)=Z$ and $g\le S\cdot H$. Since $r\ge
11$, one has $2S\cdot H-g+1\le 1$, which contradicts the
inequality $g\le S\cdot H$.
\end{proof}

We see that there is a~natural homomorphism
$\theta\colon\bar{\Gamma}\to\mathrm{Aut}(S)$. Arguing as in
Lemma~\ref{lemma:Klein-cubic-7}, we see that $\theta$ is a
monomorphism. Arguing as in Lemma~\ref{lemma:Klein-cubic-6}, one
obtains $g\ge 11$.

By Theorem~\ref{theorem:Hurwitz-bound} we may assume that $g=14$.
Let $a_i$ be the number of points on $S$ whose stabilizers in
$\bar{\Gamma}$ are isomorphic to $\Z_i$. Using the~Riemann--Hurwitz
formula, we see that
$$
2g-2=\big(2\bar{g}-2\big)\cdot\big|\bar{\Gamma}\big|+330a_2+440a_3+528a_5+550a_6+600a_{11},%
$$
where $\bar{g}$ is the genus of the quotient curve $S/\bar{\Gamma}$
(cf. the proof of Lemma~\ref{lemma:sporadic-genera}). Then~\mbox{$\bar{g}=0$}.
We~have
$$
1294-528a_5=330a_2+440a_3+550a_6+600a_{11},%
$$
which leads to a contradiction. The obtained contradiction
completes the proof of Theorem~\ref{theorem:Klein-cubic}.

\section{Del Pezzo fibrations}
\label{section:del-Pezzo}\renewcommand{\thefootnote}{\fnsymbol{footnote}}
\centerline{by \textsc{Yuri Prokhorov}\footnote[1]{This work is
partially supported by the~grants  RFBR 08-01-00395-a,
NSh-1983.2008.1 and NSh-1987.2008.1.}}
\bigskip

Let $X$ be a~threefold with at worst terminal singularities such
that the~group $\mathrm{Aut}(X)$ has a~subgroup
$\bar{G}\cong\A_{6}$, and let $\pi\colon X\to \mathbb{P}^{1}$ be
a~$\bar{G}$-Mori fibration.

The goal of this~appendix is to prove the~following result.

\begin{theorem}
\label{theorem:del-Pezzo-1}
The isomorphism
$X\cong\mathbb{P}^1\times\mathbb{P}^2$ holds, and $\pi$ is
the~projection to the~first factor.
\end{theorem}

Recall that there exists no monomorphism
$\bar{G}\to\PGL_2(\mathbb C)$.

\begin{lemma}[{cf. \cite[Lemma~4.5]{Pr09}}]
\label{lemma:1}
Let $Y$ be a~threefold with at worst terminal singularities such that
$\mathrm{Aut}(Y)$ has a subgroup $\bar{G}\cong\A_{6}$.
Then $Y$ contains no $\bar{G}$-invariant points.
\end{lemma}

\begin{proof}
Suppose that $Y$ contains a~$\bar{G}$-invariant point $P\in Y$.
Let us show that this assumption leads to a contradiction.  Let
$T_{P,Y}$ be the~ Zariski tangent space to $Y$ at the~point
$P$.~Then $\mathrm{dim}(T_{P,Y})\geqslant 5$, because $\bar{G}$
has no faithful representations of dimension less than $5$. We see
that $Y$ is not~Gorenstein at the~point $P$ (see
\cite[Section~3]{YPG}). Let us regard $(Y\ni P)$ as an analytic
germ.

Let $r$ be the~Gorenstein index of the~singularity $(Y,P)$, let
$\pi\colon (Y^{\sharp}, P^{\sharp})\to (Y,P)$ be the~index one
cover (see~\cite[Section~3.5]{YPG}), where $P^{\sharp}=
\pi^{-1}(P)$. Then
$$
\pi\colon Y^{\sharp}\setminus \big\{ P^{\sharp}\big\}\to Y\setminus \big\{P\big\}%
$$
is the~ topological universal cover (of degree $r$). Thus, there
is an exact~sequence of groups
$$
\xymatrix{1\ar@{->}[rr]&&\mathbb{Z}_{r}\ar@{->}[rr]^{\alpha}&&\bar{G}^{\sharp}\ar@{->}[rr]^{\beta}&&\bar{G}\ar@{->}[rr]&&1},%
$$
where $\bar G^{\sharp}$ is a~finite subgroup in
$\mathrm{Aut}(Y^{\sharp}, P^{\sharp})$. Since $\bar G$ is simple,
this is a~central extension.

The group $\bar{G}^{\sharp}$ naturally acts on the~Zariski tangent
space $T_{P^{\sharp},Y^{\sharp}}$ to $Y^{\sharp}$ at the~point
$P^{\sharp}$.

Recall that $(Y^{\sharp}, P^{\sharp})$ is a~hypersurface
singularity. Hence, we have
$\mathrm{dim}(T_{P^{\sharp},Y^{\sharp}})\leqslant 4$.

By the~classification of three-dimensional terminal singularities
(see~\cite[Section~6.1]{YPG}) the~action of the~group
$\alpha(\mathbb{Z}_{r})$ on $T_{P^{\sharp},Y^{\sharp}}$ in some
coordinate system has one of the~following forms:
\begin{itemize}
 \item either $\dim (T_{P^{\sharp},Y^{\sharp}})=3$ and $(x_1,x_2,x_3)\longmapsto (\varepsilon x_1,\varepsilon^{-1} x_2,\varepsilon^{a} x_3)$,%
 \item or $\dim (T_{P^{\sharp},Y^{\sharp}})=4$ and $(x_1,x_2,x_3,x_4)\longmapsto (\varepsilon x_1,\varepsilon^{-1} x_2,\varepsilon^{a} x_3, x_4)$,%
 \item or $\dim (T_{P^{\sharp},Y^{\sharp}})=r=4$ and $(x_1,x_2,x_3, x_4)\longmapsto (\sqrt{-1} x_1,-\sqrt{-1} x_2,\pm \sqrt{-1} x_3, - x_4)$,%
\end{itemize}
where $\varepsilon$ is a~primitive $r$-th root of unity and $\gcd
(r,a)=1$.

The case $\mathrm{dim}(T_{P^{\sharp},Y^{\sharp}})=3$ and $r=2$ is
impossible, because the~group $2.\A_6$ does not have~faithful
three-dimensional representations. Therefore, the~central subgroup
$\alpha(\mathbb{Z}_{r})\subset \bar{G}^{\sharp}$ has at least~$2$
different eigenvalues. Then $T_{P^{\sharp},Y^{\sharp}}$ is
a~reducible representation of the~ group $\bar{G}^{\sharp}$.

If $\dim(T_{P^{\sharp},Y^{\sharp}})=4$, then the~subgroup
$\alpha(\mathbb{Z}_{r})$ has at least $3$ different eigenvalues.

Hence, in every possible case, the~group $\bar{G}^{\sharp}$ has
a~subrepresentation of dimension at least~$2$, which is
impossible, because the~group $\bar{G}\cong\A_{6}$ does no
admit any embedding to $\PGL_2(\mathbb C)$.
\end{proof}

\begin{lemma}[cf. {\cite{DoIs06}}]
\label{lemma:Belousov} Let  $S$ be a smooth del Pezzo surface such
that $S$ admits a non-trivial action of the~group $\bar G$. Then
$S\cong\mathbb{P}^2$.
\end{lemma}

\begin{proof}
Note that $S\not\cong\mathbb{P}^1\times \mathbb{P}^1$, because
there exists no monomorphism $\bar{G}\to\PGL_2(\C)$.

Suppose that $S\not\cong\mathbb{P}^2$. Let us derive a
contradiction.

If $K_{S}^{2}\geqslant 5$, then the~action of $\bar G$ on
$\operatorname{Pic} (S)$ is trivial, because $\operatorname{rk}
\operatorname{Pic} (S)\leqslant 5$ and the~canonical class $K_S$
is $\bar{G}$-invariant. Then any $(-1)$-curve on $S$ must be
invariant, a contradiction.

Let $C$ be a $\bar{G}$-invariant curve in $|-K_{S}|$. Then every
component of the~curve $C$ is either rational or elliptic~curve.
Moreover, the curve  $C$ consists of at most $4$ components, which
immediately implies that  $C$ is not $\bar{G}$-invariant, because
$\bar{G}$~is~simple.

Put $V=H^0(\mathcal{O}_{S}(-K_S))$. By the~above, the~group
$\bar{G}$ acts non-trivially on $V$. Then
$$
4\geqslant K_{S}^{2}=h^0\Big(\mathcal{O}_{S}\big(-K_S\big)\Big)-1\geqslant 4,%
$$
which implies that $K_{S}^{2}=4$, and the~space $V$ is
an~irreducible five-dimensional rep\-re\-senta\-tion of the~group
$\bar{G}$, because $\bar{G}$ has no non-trivial representations of
dimension less than~$5$ (see~\cite{Atlas}).

We see that $S=Q_{1}\cap Q_{2}\subset\mathbb{P}^{4}$, where
$Q_{1}$ and $Q_{2}$ are irreducible quadric hypersurfaces.

The action of the~group $\bar{G}$ on the~space $V$ induces its
action on $\mathbb{P}^{4}$.

Let $\mathcal{P}$ be the~pencil generated by $Q_{1}$ and $Q_{2}$
is $\bar{G}$-invariant. Then $\mathcal{P}$ is $\bar{G}$-invariant.
Since there is no monomorphism $\bar{G}\to\PGL_2(\C)$,
both quadrics $Q_{1}$ and $Q_{2}$ are $\bar{G}$-invariant,
which is impossible, because otherwise the vertex of a degenerate
quadric in pencil is a fixed point.
\end{proof}

\begin{corollary}\label{corollary:Belousov}
Let $F_\pi$ be a~general fiber of the~morphism $\pi$. Then
$K_{F_\pi}^2=9$
\end{corollary}

Let $F$ be any~scheme fiber of the~morphism $\pi\colon
X\to\mathbb{P}^{1}$. Then $F$ is $\bar{G}$-invariant.

\begin{lemma}
\label{lemma:2} The threefold $X$ is smooth and
$F\cong\mathbb{P}^2$.
\end{lemma}

\begin{proof}
First we show that $\operatorname{Supp}(F)$ is irreducible, normal
and has at worst Kawamata log terminal singularities. This step is
similar to the~proof of \cite[Proposition~2.6]{MoriProkhorov}.

Take $\mu\in\mathbb{Q}$ such that $(X, \mu F)$ is strictly
log-canonical. Then there are a~$\bar{G}$-invariant $\pi$-ample
divisor $H$ and small positive $\delta_1$ and $\delta_2\in \mathbb
Q$ such that $(X, (\mu-\delta_1) F+\delta_2H)$ is log canonical
and
$$
\mathbb{LCS}\Big(X, \big(\mu-\delta_1\big) F+\delta_2H\Big)=\bigcup_{g\in \bar{G}} g\big(S\big),%
$$
where $S$ is a~minimal center in $\mathbb{LCS}(X, \mu F)$ (cf.
Lemma~\ref{lemma:Kawamata-Shokurov-trick} or the~proof of
\cite[Proposition~2.6]{MoriProkhorov}).

Put $D=(\mu-\delta_1) F+\delta_2H$. Then $\mathrm{LCS}(X, D)$ is
a~$\bar{G}$-orbit of the~center $S$.

By Theorem~\ref{theorem:connectedness}, the~locus $\mathrm{LCS}(X,
D)$ is connected. It follows from Lemma~\ref{lemma:centers} that
$\mathbb{LCS}(X, D)=S$, which implies that $S$ is not a~point by
Lemma~\ref{lemma:1}.

If $S$ is a~curve, then $S\cong\mathbb{P}^{1}$ by
Theorem~\ref{theorem:Kawamata}, which is again impossible by
Lemma~\ref{lemma:1}.

We see that $S$ is a~$\bar{G}$-invariant surface. Then
$S=\mathrm{Supp}(F)$, because $\pi$ is
a~$\bar{G}$-Mori~fibred~space.

By Theorem~\ref{theorem:Kawamata}, the~surface $S$ is normal and
has Kawamata log terminal singularities.

Let $\Sigma\subset S$ be a~subset consisting of points where $S$
is not Cartier. If $F$ is not reduced, then
$1\leqslant|\Sigma|\leqslant 4$ by
\cite[Theorem~1.1]{MoriProkhorov}. Since $\Sigma$ is $\bar
G$-invariant, we see that $\Sigma=\varnothing$ by
Lemma~\ref{lemma:1}.

We see that $F$ is a reduced normal surface having only quotient
singularities, which implies that $F$ is a degeneration of
$\mathbb{P}^2$ by Corollary~\ref{corollary:Belousov}. Hence,
the~inequality $|\mathrm{Sing}(F)|\leqslant 3$ holds
(see~\cite[Main~Theorem]{Manetti},
\cite[Corollary~1.2]{HackingProkhorov}). So, the surface $F$ is
smooth by Lemma~\ref{lemma:1}, which immediately implies that
$F\cong\mathbb{P}^2$.
\end{proof}

\begin{proof}[Proof of Theorem~{\textup{\ref{theorem:del-Pezzo-1}}}]
It follows from Lemma~\ref{lemma:2} that every scheme fiber of
the~morphism $\pi$~is isomorphic to $\mathbb{P}^{2}$. By
\cite[Proposition~V.4.1]{BPV}, there are integers $b\geqslant
a\geqslant 0$ such that
$$
X\cong\mathrm{Proj}\Big(\mathcal{O}_{{\mathbb{P}^1}}\oplus\mathcal{O}_{{\mathbb{P}^1}}\big(a\big)\oplus\mathcal{O}_{{\mathbb{P}^1}}\big(b\big)\Big).
$$

Recall that the~ action of the~group $\bar G$ on the~base $\mathbb
P^1$ is trivial.

By Lemma~\ref{lemma:1}, every fiber of the~morphism $\pi$ contains
no $\bar{G}$-fixed points nor $\bar{G}$-invariant lines, which
implies that $a=b=0$ and so
$X\cong\mathbb{P}^1\times\mathbb{P}^2$.
\end{proof}


\begin{thebibliography}{99}

\bibitem{Ad78}
A.\,Adler, \emph{On the~automorphism group of a~certain cubic threefold}\\
American Journal of Mathematics \textbf{100} (1978), 1275--1280

\smallskip

\bibitem{GrHaArbCor85}
E.\,Arbarello, M.\,Cornalba, P.\,Griffiths, J.\,Harris, \emph{Geometry of algebraic curves. Volume I}\\
Grundlehren der mathematischen Wissenschaften \textbf{267} (1985) Springer-Verlag, New York-Heidelberg %

\smallskip

\bibitem{BPV}
W.\,Barth, C.\,Peters, A.\,Van~de~Ven, \emph{Compact complex surfaces}\\
Springer-Verlag, Berlin, 1984%

\smallskip
\bibitem{Be11}
A.\,Beauville, \emph{Non-rationality of the symmetric sextic Fano threefold}\\
arXiv:math/1102.1255 (2011).

\smallskip

\bibitem{Magma}
W.\,Bosma, J.\,Cannon, C.\,Playoust, \emph{The Magma algebra system. I. The user language}\\
Journal of Symbolic Computation, \textbf{24} (1997), 235--265

\smallskip

\bibitem{Bre00}
T.\,Breuer, \emph{Characters and automorphism groups of compact Riemann surfaces}\\
London Mathematical Society Lecture Note Series \textbf{280}, Cambridge University Press, 2000%

\smallskip

\bibitem{Ch05umn}
I.\,Cheltsov, \emph{Birationally rigid Fano varieties}\\
Russian Mathematical Surveys \textbf{60} (2005), 875--965

\smallskip

\bibitem{Ch07b}
I.\,Cheltsov, \emph{Log canonical thresholds of del Pezzo surfaces}\\
Geometric and Functional Analysis, \textbf{18} (2008), 1118--1144%

\smallskip

\bibitem{ChSh08c}
I.\,Cheltsov, C.\,Shramov, \emph{Log canonical thresholds of smooth Fano threefolds}\\
Russian Mathematical Surveys \textbf{63} (2008), 73--180%

\smallskip

\bibitem{ChSh09}
I.\,Cheltsov, C.\,Shramov, \emph{On exceptional quotient singularities}\\
Geometry and Topology,  \textbf{15} (2011), 1843--1882%

\smallskip


\bibitem{ClGr72}
C.\,Clemens, P.\,Griffiths, \emph{The intermediate Jacobian of the~cubic threefold}\\
Annals of Mathematics \textbf{95} (1972), 73--100%

\smallskip

\bibitem{Co95}
A.\,Corti, \emph{Factorizing birational maps of threefolds after Sarkisov}\\
Journal of Algebraic Geometry \textbf{4} (1995), 223--254%


\smallskip

\bibitem{Co00}
A.\,Corti, \emph{Singularities of linear systems and 3-fold birational geometry}\\
L.M.S. Lecture Note Series \textbf{281} (2000), 259--312%

\smallskip

\bibitem{CKS04}
A.\,Corti, K.\,E.\,Smith, J.\,Koll\'ar, \emph{Rational and nearly rational varieties}\\
Cambridge Studies in Advanced Mathematics \textbf{92} (2004), Cambridge University Press%

\smallskip

\bibitem{Atlas}
J.\,Conway, R.\,Curtis, S.\,Norton, R.\,Parker, R.\,Wilson, \emph{Atlas of finite groups}\\
Clarendon Press, Oxford, 1985

\smallskip

\bibitem{DoIs06}
I.\,Dolgachev, V.\,Iskovskikh, \emph{Finite subgroups of the~plane Cremona group}\\
Progress in Mathematics, Birkhauser, Boston \textbf{269} (2009), 443--548%

\smallskip

\bibitem{Fe71}
W.\,Feit, \emph{The current situation in the~theory of finite simple groups}\\
Actes du Congr\`es International des Math\'ematiciens, Gauthier--Villars, Paris (1971), 55--93%

%
%
\smallskip

\bibitem{Finkelberg}
H.\,Finkelnberg, \emph{Small resolutions of the Segre cubic}\\
Indagationes Mathematicae \textbf{90} (1987), 261--277

\smallskip

\bibitem{FinkelbergWerner}
H.\,Finkelnberg, J.\,Werner, \emph{Small resolutions of nodal cubic threefolds}\\
Indagationes Mathematicae \textbf{92} (1989), 185-–198

\smallskip

\bibitem{HackingProkhorov}
P.\,Hacking, Yu.\,Prokhorov, \emph{Smoothable del Pezzo surfaces with quotient singularities}\\
Compositio Mathematica \textbf{146} (2010), 169--192

\smallskip

\bibitem{HMX10}
C.\,Hacon, J.\,McKernan, Ch.\,Xu,
\emph{On the birational automorphisms of varieties of general type}\\
arXiv:math/1011.1464 (2010)

\smallskip

\bibitem{Har77}
R.\,Hartshorne, \emph{Algebraic geometry}\\
Graduate Texts in Mathematics \textbf{52} (1977) Springer-Verlag, New York-Heidelberg%

\smallskip

\bibitem{Hu96}
B.\,Hunt, \emph{The geometry of some special arithmetic quotients}\\
Lecture Notes in Mathematics \textbf{1637} (Springer--Verlag, New York, 1996)%

\smallskip

\bibitem{IsMa71}
V.\,Iskovskikh, Yu.\,Manin, \emph{Three-dimensional quartics and counterexamples to the L\"uroth problem}\\
Matematiskii Sbornik \textbf{86} (1971), 140--166

\smallskip

\bibitem{IsPr99}
V.\,Iskovskikh, Yu.\,Prokhorov, \emph{Fano varieties}\\
Encyclopaedia of Mathematical Sciences \textbf{47} (1999) Springer, Berlin%

\smallskip

\bibitem{IsPu96}
V.\,Iskovskikh, A.\,Pukhlikov, \emph{Birational automorphisms of multidimensional algebraic manifolds}\\
Journal of Mathematical Sciences \textbf{82} (1996), 3528--3613

\smallskip

\bibitem{Kaw97}
Y.\,Kawamata, \emph{On Fujita's freeness conjecture for $3$-folds and $4$-folds}\\
Mathematische Annalen \textbf{308} (1997), 491--505%

\smallskip

\bibitem{Kaw98}
Y.\,Kawamata, \emph{Subadjunction of log canonical divisors II}\\
American Journal of Mathematics \textbf{120} (1998), 893--899

\smallskip

\bibitem{Ko97}
J.\,Koll\'ar, \emph{Singularities of pairs}\\
Proceedings of Symposia in Pure Mathematics \textbf{62} (1997), 221--287%

\smallskip

\bibitem{La04}
R.\,Lazarsfeld, \emph{Positivity in algebraic geometry} \textbf{II}\\
Springer-Verlag, Berlin, 2004%

\smallskip


\bibitem{LPR}
N.\,Lemire, V.\,Popov, Z.\,Rechstein, \emph{Cayley groups} \\
Journal of the American Mathematical Society \textbf{19} (2006), 921--967%

\smallskip

\bibitem{Manetti}
M.\,Manetti, \emph{Normal degenerations of the~complex projective plane} \\
Journal f\"ur die Reine und Angewandte Mathematik \textbf{419} (1991), 89--11%

\smallskip

\bibitem{Me03}
M.\,Mella, \emph{Birational geometry of quartic 3-folds II: the~importance of being $\mathbb{Q}$-factorial}\\ %
Mathematische Annalen \textbf{330} (2004), 107--126

\smallskip

\bibitem{MoriProkhorov}
S.\,Mori, Yu.\,Prokhorov, \emph{Multiple fibers of del Pezzo fibrations}\\
Proceedings of the~Steklov Institute of Mathematics \textbf{264} (2009), 131--145%


\smallskip

\bibitem{Muk92}
S.\,Mukai, \emph{Curves and symmetric spaces}\\ %
Proceedings of the~Japan Academy, Series A, Mathematical Sciences \textbf{68} (1992), 7--10%


\smallskip

\bibitem{Nie92}
I.\,Nieto, \emph{The normalizer of the~level $(2,2)$-Heisenberg group}\\
Manuscripta Mathematica \textbf{76} (1992), 257--267

\smallskip

\bibitem{Pet98}
K.\,Pettersen, \emph{On nodal determinantal quartic hypersurfaces in $\mathbb{P}^4$}\\
Thesis, University of Oslo (1998)

\smallskip

\bibitem{Pr09}
Yu.\,Prokhorov, \emph{Simple finite subgroups of the~Cremona group of rank $3$}\\
Journal of Algebraic Geometry, to appear

\smallskip

\bibitem{YPG}
M.\,Reid, \emph{Young person's guide to canonical singularities}\\
Proceedings of Symposia in Pure Mathematics \textbf{46} (1987), 345--414 %

\smallskip


\bibitem{Serre}
J.-P.\,Serre, \emph{A Minkowski-style bound for the~order of the~finite subgroups of the~Cremona group of rank 2\\ over an arbitrary field}\\
Moscow Mathematical Journal \textbf{9} (2009), 183--198

\smallskip

\bibitem{Sho93}
V.\,Shokurov, \emph{Three-fold log flips}\\
Russian Academy of Sciences, Izvestiya Mathematics \textbf{40} (1993), 95--202 %

\smallskip

\bibitem{Shr08}
C.\,Shramov, \emph{ Birational automorphisms of nodal quartic threefolds}\\
arXiv:math/0803.4348 (2008)

\smallskip

\bibitem{Ti87}
G.\,Tian, \emph{On K\"ahler--Einstein metrics on certain K\"ahler manifolds with $c_{1}(M)>0$}\\
Inventiones Mathematicae \textbf{89} (1987), 225--246%

\smallskip

\bibitem{TiYa87}
G.\,Tian, S.-T.\,Yau, \emph{K\"ahler--Einstein metrics metrics on complex surfaces with $\mathrm{C}_{1}>0$}\\
Communications in Mathematical Physics  \textbf{112} (1987), 175--203%

\smallskip

\bibitem{To33}
J.\,A.\,Todd, \emph{Configurations defined by six lines in space of three dimensions}\\
Mathematical Proceedings of the~Cambridge Philosophical Society \textbf{29} (1933), 52--68%

\smallskip

\bibitem{Va01}
D.\,Vazzana, \emph{Invariants and projections of six lines in projective space}\\
Transactions of the~American Mathematical Society  \textbf{353} (2001), 2673--2688%

\end{thebibliography}
\end{document}